\documentclass[reqno,12pt, centertags]{amsart}
\usepackage[left=1in,right=1in,bottom=1in,top=1in]{geometry}
\usepackage{amsmath,amsthm,amscd,amssymb}
\usepackage{amsmath,amsthm,amscd,amssymb,color,eucal,enumitem,latexsym,lscape,graphicx,multirow,mathrsfs,mathtools}
\usepackage{upref}
\usepackage{latexsym}
\usepackage{lscape}
\usepackage{graphicx}
\usepackage{multirow}
\usepackage{mathrsfs} 

\usepackage{mathtools}

\usepackage{stackengine}

\newcommand{\bbN}{{\mathbb{N}}}
\newcommand{\bbR}{{\mathbb{R}}}

\newcommand{\cA}{{\mathcal A}}
\newcommand{\cB}{{\mathcal B}}

\newcommand{\cE}{{\mathcal E}}
\newcommand{\cF}{{\mathcal F}}
\newcommand{\cG}{{\Gamma}}
\newcommand{\cH}{{\mathcal H}}

\newcommand{\cK}{{\mathcal K}}

\newcommand{\cP}{{\mathcal P}}

\newcommand{\cV}{{\mathcal V}}

\newcommand{\cX}{{\mathcal X}}
\newcommand{\cY}{{\mathcal Y}}

\newcommand{\no}{\nonumber}
\newcommand{\lb}{\label}

\newcommand{\wti}{\widetilde  }

\newcommand{\hatt}{\widehat}

\newcommand{\lp}{L^p(\Gamma)}

\newcommand{\elpee}{L^p(\Gamma)}
\newcommand{\elpgamma}{L^{{p}/{\gamma}}(\Gamma)}
\newcommand{\wpee}{\hatt W^{1, p}(\Gamma)}
\newcommand{\elq}{L^q(\Gamma)}
\newcommand{\elr}{L^r(\Gamma)}
\newcommand{\elpeet}[1]{C([0, {#1}], \elpee )}
\newcommand{\wpeet}[1]{C([0, {#1}], \wpee )}

\newcommand{\doublehookrightarrow}{%
    \mathrel{\mathrlap{{\mspace{4mu}\lhook}}{\hookrightarrow}}
}

\newcommand{\el}[1]{L^{#1}(\Gamma)}

\numberwithin{equation}{section}

\newcommand{\dom}{\operatorname{dom}}

	\newcommand\sel[1]{ #1}

\newcommand{\lap}{\partial_x^2}

\theoremstyle{plain}
\newtheorem{theorem}{Theorem}[section]

\newtheorem{lemma}{Lemma}[section]
\newtheorem{corollary}{Corollary}[section]
\newtheorem{proposition}{Proposition}[section]

\theoremstyle{definition}
\newtheorem{definition}{Definition}[section]

\newtheorem{remark}{Remark}[section]

\newtheorem{hypothesis}{Hypothesis}[section]

\begin{document}

\allowdisplaybreaks

\setcounter{tocdepth}{4}
\setcounter{secnumdepth}{4}

\numberwithin{equation}{section}
\allowdisplaybreaks

\title[Chemotaxis on metric graphs]{Well--posedness of Keller--Segel systems on compact metric graphs}

\author[H. Shemtaga]{Hewan Shemtaga}
\address{Department of Mathematics and Statistics,
	Auburn University, Auburn, AL 36849, USA}
\email{hms0069@auburn.edu}

\author[W.\ Shen]{Wenxian Shen}
\address{Department of Mathematics and Statistics,
	Auburn University, Auburn, AL 36849, USA}
\email{wenxish@auburn.edu}

\author[S.\ Sukhtaiev]{Selim Sukhtaiev	}
\address{Department of Mathematics and Statistics,
	Auburn University, Auburn, AL 36849, USA}
\email{szs0266@auburn.edu}


\date{\today}
\thanks {We thank Gregory Berkolaiko for numerous discussions. HS and SS were supported in part by NSF grant DMS-2243027, Simons Foundation grant MP-TSM-00002897, and by the Office of the Vice President for Research \& Economic Development (OVPRED) at Auburn University through the Research Support Program grant.}


\begin{abstract}
Chemotaxis phenomena govern the directed movement of micro-organisms in response to chemical stimuli. In this paper, we investigate two Keller--Segel systems of reaction-advection-diffusion equations modeling chemotaxis on thin networks. The distinction between two systems is driven by the rate of diffusion of the chemo-attractant. The intermediate rate of diffusion { is}  modeled by a coupled pair of parabolic equations, while the rapid rate is described by a parabolic equation coupled with an elliptic one. Assuming the polynomial rate of growth of the chemotaxis sensitivity coefficient, we prove local well-posedness of both systems on compact metric graphs, and, in particular, prove existence of unique classical solutions. This is achieved by constructing  sufficiently regular mild solutions via analytic semigroup methods and combinatorial description of the heat kernel on metric graphs. The regularity of mild solutions is shown by applying abstract semigroup results to semi-linear parabolic equations on compact graphs. In addition, for logistic type Keller--Segel systems we prove global well-posedness and, in some special cases, global uniform boundedness of solutions.
\end{abstract}

\maketitle

{\scriptsize{\tableofcontents}}


\section{Introduction}

\subsection{Overview} Biological species often exhibit tendency to drift, spread or localize within complex media subject to external signals. This phenomenon plays a vital role in formation, ordering, and stability of  patterns observed experimentally and numerically in models of bio-engineering (dermal tissue recovery \cite{MaKu}), ecology (population dynamics \cite{KJP}), and social sciences (e.g., urban crime dynamics \cite{RW1, RW2}). The external signals and cues in question are generated by an attractant or a repellent facilitating the migration. In the biological context, the directed movement of bacteria, cells or other microorganisms in response to chemical attractant is often referred to as chemotaxis, which is a widely studied phenomenon in pure and applied mathematics \cite{MR3351175, MR3698165, MR2409228, MR2448428, MR2013508, MR3925816, KS1, MR3294344, MR3620027, MR3397319, MR2334836, MR1654389, MR3147229, MR2445771, MR2644137, MR2754053, MR2825180, MR3115832, MR3210023, MR3462549, MR3335922, MR3286576}. 

The two quantities central to the mathematical modeling of such processes are the density of bio-species $u(t,x)$ and the concentration of attracting (repelling) substance $v(t,x)$.  Their coupled evolution is described by the Keller--Segel model, that is, by a pair of nonlinear partial differential equations stemming from the diffusion equations
\begin{equation}
\begin{cases}\lb{cases1}
\partial_tu+\partial_ xJ_{u,v}-\varphi(u,v)=0,\\
\tau \partial_tv+\partial_ x J_v-\psi (u,v)=0,\\
\end{cases}
\end{equation}
where 
\begin{itemize}
\item  $ J_{u,v}=-\partial_x u+\chi(u,v)\partial_xu$ is the density flux which consist of the taxis-flux term $\chi(u,v)\partial_xv$ that governs the population drifts in response to attractant $v$ and the standard flux term $-\partial_x u$ given by Fick's law,
\item  $ J_v= - \partial_x v$   is the standard flux term describing diffusion of the chemo-attractant by Fick's law; importantly, in the case of rapid diffusion, that is, $0<\tau\ll 1$ the second equation \eqref{cases1} can be approximated by $\partial_{xx}^2 v-\psi (u,v)=0$, 
\item $\varphi(u,v)$, $\psi(u,v)$ describe the rates at which the biospecies, respectively, the chemo-attractant are produced. 
\end{itemize}
In this paper, we provide a systematic treatment of general chemotaxis systems with moderate and  fast rates of diffusion of the attractant on arbitrary compact metric graphs. Concretely, we focus on the following two systems of reaction--advection--diffusion equations of parabolic-parabolic type 
\begin{equation}
\begin{cases}\lb{1.2}
	\partial_tu=\partial_ x(\partial_x u-\chi(u,v)\partial_x v)+\varphi(u,v),\\
	\tau \partial_tv=\partial_ {xx}^2 v+\psi (u,v),\\
\end{cases}
\end{equation}
and parabolic-elliptic type
\begin{equation}
\begin{cases} \lb{1.3}
	\partial_tu=\partial_ x(\partial_x u-\chi(u,v)\partial_x v)+\varphi(u,v),\\
	0=\partial_ {xx}^2 v-\sigma v +\psi(u), \\
\end{cases}
\end{equation}
with nonlinear terms $\chi, \varphi, \psi$ that exhibit polynomial growth conditions at infinity and vanish to the first order at zero, see Hypothesis \ref{hyp} for more details.

\subsection{Description chemotaxis models on graphs} Throughout this paper we assume that $\cG=(\cV, \cE)$ is a compact metric graph. That is, the set of vertices $\cV$ is finite, the set edges $\cE$ is finite and consists of edges of finite length. We assume that $\cG$ is a connected and oriented graph, that is, each edge $e$ is equipped with the initial vertex $i(e)\in\cV$ and the terminal vertex $t(e)\in \cV$. An edge $e\in\cE$ is identified with the interval $(0,|e|)$  so that every point on the graph is described by a pair of coordinates $(e, \xi)$, $\xi\in (0,|e|)$. A function $u:\Gamma\rightarrow \bbR$ (resp. $u:\bar\Gamma\rightarrow\bbR$)  is given by a collection of functions $\{u_e| e\in\cE\}$  where $u_e:(0,|e|)\rightarrow\bbR$
(resp. $u_e:[0,|e|]\rightarrow\bbR$).  Differentiation and integration are defined edge-wise with respect to the chosen orientation of edges, that is, $\partial_x u({ e,x}):=\partial_x u_e(x),$ $x\in (0,e)$ and 
$
\int_{\Gamma} u(x)dx:= \sum\limits_{e\in\cE} \int_0^{|e|} u_e(x)dx.
$
Throughout this paper, $\partial_{\nu} u_e{(\vartheta)}$ (with $\vartheta=(e,0)$ or $\vartheta=(e,|e|)$)  denotes the inward~pointed normal derivative of $u$  at the end points of the edge $e$, i.e., 
\begin{equation}
	\partial_{\nu}u (\vartheta)=\begin{cases}
		u_e'(0)& \vartheta=({e,0}),\\
		-u_e'(|e|)&  \vartheta=({ e, |e|}).
	\end{cases} 
\end{equation} 
The functional spaces are often considered as direct sums of the corresponding functional spaces on individual edges. That is,  if  $u_e\in \cX(0, |e|)$ for all $e\in\cE$ then we write $u\in\hatt \cX(\Gamma)$ where $\cX$ is one of the functional spaces $C_0^{\infty}$, $C$, $C^{\nu}$, $L^p$, $W^{k,p}$. 


The class of chemotaxis models under investigation is parameterized by a scalar $\tau\geq 0$ and given by
\begin{equation}
\label{parabolic-parabolic-eq}
\begin{cases}
u_t=\partial_x\big(\partial_{x} u-f_1(u,v)\partial_xv\big)+f_2(u,v), \quad {t>0,\,\, x\in\Gamma,}\cr
\tau v_t=\partial_{xx}^2 v+f_3(u,v), \quad  { t>0,\,\, x\in\Gamma,}\cr
u(0, x)=u_0(x),\ v(0,x)=v_0(x),\quad { x\in\Gamma,} \cr
\end{cases}
\end{equation}
where the functions
\begin{equation}\no
u:[0,\infty)\times \overline\Gamma\rightarrow\bbR, v:[0,\infty)\times \overline\Gamma\rightarrow\bbR, 
\end{equation} are subject to natural Neumann--Kirchhoff vertex conditions for $t>0$, that is, they obey
\begin{align}
	&\text{continuity at vertices:\ }u_e( t,\vartheta)=u_{e'}({ t,} \vartheta),\
{ v_e( t, \vartheta)=v_{e'}({t,}\vartheta)},\
  \vartheta\sim e, e',\lb{contcond}\\
	&\text{current conservation:\ }\sum_{e\sim \vartheta}\partial_{\nu}u_e({ t,}\vartheta)=0, \,\,\,  \  { \sum_{e\sim \vartheta}\partial_{\nu}v_e({t,} \vartheta)=0}, \ \lb{fluxcond}
\end{align}
here and throughout the paper $\vartheta\sim e$ means that edge $e$ is incident with vertex $\vartheta$, see Figure \ref{fig1}, where, for example, $\vartheta_1\sim e_k$, $k=1,2,3$. 
The first condition \eqref{contcond} guarantees continuity of $u$ over the graph $\Gamma$ considered as a metric space. {The second condition \eqref{fluxcond} yields preservation of current through the vertex, and, in particular, facilitates preservation of total mass for the minimal model, i.e.,  $f_2, f_3\equiv 0$. }
\begin{figure}
	\includegraphics[scale=0.2]{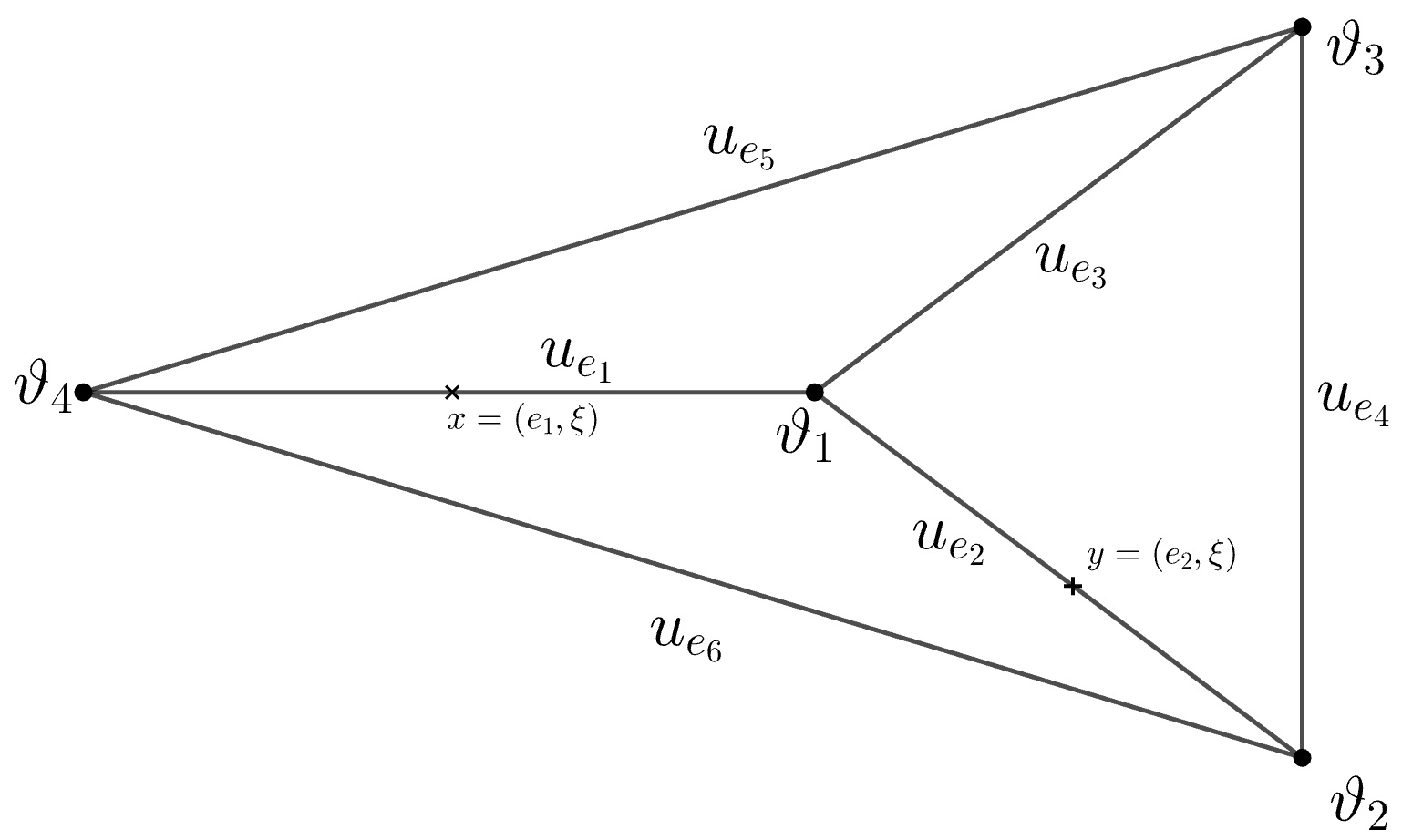}
	\caption{A metric graph with 6 edges and  4 vertices.  A function $u$ is given by a vector $u_{e_1}, ..., u_{e_6}$.}\lb{fig1}
\end{figure}

 The distinction between two regimes $\tau=0$ and $\tau>0$ corresponding to parabolic--elliptic and parabolic--parabolic PDEs, respectively,  plays an important role throughout the paper.  
The nonlinear terms $f_1, f_2, f_3$ under consideration satisfy the following assumptions. 
\begin{hypothesis}\lb{hyp}
Let $f_k:\bbR^2\rightarrow\bbR$ and denote 
\begin{equation}
g_1(s,{ r}):=s^{-1}f_1(s,{ r}), g_2(s,{r}):=s^{-1}f_2(s,{ r}), g_3(s,{ r}):=s^{-1}\partial_{ r}f_1(s,{ r}). 
\end{equation}
For $k=1,2,3$, assume that $f_k, g_k\in C^{1,1}(\bbR^2)$ and that
\begin{align}\lb{polbnd}
&\sum_{\substack{|\alpha|\leq 1\\ 1\leq k\leq 3}}|D^{\alpha} f_k(s,{ r})|+|D^{\alpha} g_k(s,{ r})| \leq C(1+|s|^{\mu_1}+|{ r}|^{\mu_2}),\ (s,{ r})\in\bbR^2,
\end{align}
for some $C=C(f_1, f_2, f_3, \mu_1, \mu_2)>0$, $\mu_1, \mu_2\geq 0$ with $\gamma:=\mu_1+\mu_2\geq 1$.   

If $\tau=0$ assume that  $f_3(u,v)=-\sigma v+f(u)$ with $\sigma>0$.  Finally, throughout this paper we assume that $p> 8\gamma$.   
\end{hypothesis} 
\begin{remark}The polynomial growth assumption \eqref{polbnd} is used to prove various $L^p-L^q$ mapping properties of the nonlinear operators of the type $(u,v)\mapsto D^{\alpha}(f(u,v)D^{\beta}v)$ that naturally stem from  the right-hand side of \eqref{parabolic-parabolic-eq}, see, for example, Propositions \ref{intconv}, \ref{prop1.3} and Appendix \ref{app:LpLq}. This level of generality includes, in particular, the following models of chemotaxis: the minimal model \cite{MR3698165, MR2448428}, the density and signal dependent sensitivity models \cite[(M2b), (M3a)]{MR2448428}, the cell kinetics model \cite[(M8)]{MR2448428}, the nonlinear diffusion \cite{MR3537839} and secretion \cite{MR4043491} models as well as some others, cf. \cite[Table 1]{MR4188348}.  
\end{remark}

\subsection{Main results} Our first main result concerns well-posedness of the chemotaxis model on compact  metric graphs in two regimes: $\tau=0$ and $\tau>0$. 
In the following,   $\cX^{\alpha}_p$ ($\alpha\in(0,1)$, $1\leq p< \infty$) denotes the fractional power spaces generated by the Neumann--Kirchhoff Laplacian on $\hatt L^p(\Gamma)$.  We provide a more detailed discussion of functional spaces on graphs in Appendix \ref{app:FunctionalSpaces}.\footnote{A more  comprehensive introduction into differential operators and their spectral theory on metric graphs can be found, for example, in \cite{BK}}

\begin{theorem}\lb{Lp}Assume Hypothesis \ref{hyp} and fix arbitrary $r\geq1$, $\beta\in(0,1/8)$, $\nu<\beta$. 

(1) Assume that $\tau=0$.  Then for arbitrary $u_0\in \hatt L^p(\Gamma)$ there exists $T_{\max}=T_{\max}(u_0)\in (0,\infty]$ such  that \eqref{parabolic-parabolic-eq} has a unique classical solution 
 $(u(t,x),v(t,x))=(u(t, x; u_0),v(t, x; u_0))$, $t\in [0, T_{\max})$ satisfying the Neumann--Kirchhoff vertex conditions { \eqref{contcond}, \eqref{fluxcond}} for $t>0$, $u,v\in \hatt C^{1,2}( (0, T_{\max})\times \overline{\Gamma})$, and 
		\begin{align}\lb{rubp}
		u\in C^{{\beta}}((0, T_{\max}), \hatt C^{\nu}(\overline{\Gamma}))\cap  C([0, T_{\max}), \hatt L^p(\Gamma))\cap C^{\beta}((0,  T_{\max}), \cX^{\beta}_r).
	\end{align}	
	If $T_{\max}<\infty$ then $\limsup\limits_{t\rightarrow T_{\max}^-}\|u(t,\cdot)\|_{\elr}=\infty$. Finally, if  
 {$u_0\geq 0$}, then $u,v$ are non-negative as well.

(2) Assume that $\tau>0$ and let $(u_0, v_0)\in \hatt L^p(\Gamma)\times \wpee$. Then there exists $T_{\max}=T_{\max}(u_0, v_0)\in (0,\infty]$ such  that \eqref{parabolic-parabolic-eq} has a unique classical solution 
\begin{equation}
(u(t,x),v(t,x))=(u({ t,x}; u_0, v_0),v({ t,x}; u_0, v_0)), t\in [0, T_{\max}),
\end{equation}
 satisfying the  Neumann--Kirchhoff vertex conditions  {\eqref{contcond}, \eqref{fluxcond}}  for $t>0$,  $u,v\in \hatt C^{1,2}( (0, T_{\max})\times \overline{\Gamma})$, and 
	\begin{align}
	\label{rmilduvsolsmax-u}
		u\in C([0, T_{\max}), \hatt L^p(\Gamma))\cap  C^{\beta}((0, T_{\max}), \cX^{\beta}_r)\cap C^{\beta}((0, T_{\max}), \hatt C^{\nu}(\overline{\Gamma})),
\end{align}
\begin{align}
\label{rmilduvsolsmax-v}
		{ v\in	C([0,T], \hatt L^p(\Gamma))\cap C^{1+\beta} ((0,T_{\max}), \cX^{\beta}_r)\cap C^{\beta}((0, T_{\max}), \hatt C^{2+\nu}(\overline{\Gamma})). }
\end{align}	
If  $T_{\max}<\infty$  then $\limsup\limits_{t\rightarrow T_{\max}^-}\|(u(t,\cdot),v(t,\cdot))\|_{L^p(\Gamma)\times\wpee}=\infty.$ Finally, if  $u_0\geq 0, v_0\ge 0$, then $u,v$ are non-negative.
\end{theorem}

Our second principal result concerns global existence of unique classical solutions for logistic-type models given by
\begin{equation}
	\label{logparabolic-parabolic-eq}
	\begin{cases}
		u_t=\partial_x\big(\partial_{x} u-\chi u\partial_xv\big)+g(u),\quad { t>0,\,\, x\in\Gamma,} \cr
		\tau v_t=\partial_{xx}^2 v-v+u, \quad { t>0,\,\, x\in\Gamma,}\cr
		u(0, x)=u_0(x),\ v(0,x)=v_0(x),\quad {x\in\Gamma}\cr
	\end{cases}
\end{equation}
{ subject to the Neumann--Kirchhoff vertex conditions  \eqref{contcond}, \eqref{fluxcond} for $t>0$.}
\begin{theorem}
	\label{global-existence-thm}
	Let $g: \bbR\rightarrow \bbR$ and $\wti  g(s):=s^{-1}g(s)$. Assume that $g, \wti g\in C^1(\bbR)$ and suppose that for some constants $C>0$, $k\ge 0$, $l\ge 0$, $\gamma\geq 1$ one has
	\begin{align}
	& |g(s)|+|g'(s)|+|\wti g(s)|+|\wti g'(s)|\leq C(1+|s^{\gamma}|),\\
	&g(s)\leq k+ls,\text{ for all\ } s\geq 0.
	\end{align}

(1) Assume that $\tau=0$. Let $u_0\in   \hatt L^p(\Gamma)$  with $u_0\geq 0$.   Then the chemotaxis system \eqref{logparabolic-parabolic-eq} has a unique global non-negative classical solution $(u(t,x),v(t,x))=(u(t,x; u_0),v(t,x; u_0))$,  satisfying  the Neumann--Kirchhoff vertex conditions { \eqref{contcond}, \eqref{fluxcond}}  for $t>0$, $u,v\in \hatt C^{1,2}( (0, \infty)  \times \overline{\Gamma})$.  Suppose, in addition, that for some $m\geq 0$ and $\varepsilon>0$ one has
		\begin{align}
		&g(s)\leq k+ls-m s^{1+\varepsilon},\text{ for all\ } s\geq 0. \lb{gest3}
	\end{align}
If either $m>0$ or $k=l=m=0$ then the solution is globally uniformly bounded, that is, 
	\begin{align}
			&\sup_{t\geq 0}\|u(t, \cdot)\|_{\hatt L^{\infty}(\Gamma)}<\infty. \lb{1.18}
		\end{align}
(2) Assume that $\tau>0$. Let    $(u_0, v_0)\in \hatt L^p(\Gamma)\times \wpee$, $u_0\geq 0, v_0\ge 0$. 
Then the chemotaxis system \eqref{logparabolic-parabolic-eq}  has a unique global non-negative classical solution 
\begin{equation}
(u(t,x),v(t,x))=(u(t,x; u_0,v_0),v(t,x; u_0,v_0)), u,v\in \hatt C^{1,2}( (0, \infty)\times \overline{\Gamma}),
\end{equation}
satisfying the Neumann--Kirchhoff vertex conditions { \eqref{contcond}, \eqref{fluxcond}}   for $t>0$. Suppose, in addition, that $g$ satisfies \eqref{gest3}. If either $m>0$ or $k=l=m=0$ then the solution is globally uniformly bounded, that is, 
\begin{align}
	&\sup_{t\geq 0}\left(\|u(t, \cdot)\|_{\hatt L^{\infty}(\Gamma)}+\|v(t, \cdot)\|_{\hatt L^{\infty}(\Gamma)}\right)<\infty. \lb{1.20}
\end{align}

\end{theorem}

{Theorem \ref{Lp} is proved  by showing 
  the existence,  uniqueness, and regularity  of mild solutions  of \eqref{parabolic-parabolic-eq} subject to \eqref{contcond}, \eqref{fluxcond}. Theorem  \ref{global-existence-thm} is proved via  the boundedness of
$\|u(t,\cdot)\|_{\hatt L^q(\Gamma)}+\|v(t,\cdot)\|_{\hatt L^q(\Gamma)}$ for all $q\ge 1$.  The proofs of  the existence,  uniqueness, and regularity  of mild solutions\footnote{ see Section \ref{outline1} for a more detailed outline of the proof} of  \eqref{parabolic-parabolic-eq} subject to \eqref{contcond}, \eqref{fluxcond},  and
  the boundedness\footnote{ see Section \ref{outline2} for a more detailed outline of the proof} of
$\|u(t,\cdot)\|_{\hatt L^q(\Gamma)}+\|v(t,\cdot)\|_{\hatt L^q(\Gamma)}$ 
 heavily rely on
various quantitative and qualitative properties of the analytic semigroup  $e^{\Delta t}$ generated by  Neumann--Kirchhoff Laplacian and  the evolution operator  generated by linear non-autonomous parabolic equations on metric graphs. One of the important results along this line is an $L^p-W^{1, q}$ bound for the norm of the  heat semigroup, which we present next.}  

\begin{theorem}\lb{lplq estimate}
	Let $\{e^{\Delta t}\}_{t\geq 0}$ be analytic semigroup generated by the Neumann--Kirchhoff Laplacian in $\hatt L^p(\Gamma)$, $1\leq p<\infty$. Let $q\in [p, \infty)$ and  {$T>0$}. There exists a constant $C=C(\sel{ T}, p, q, \Gamma)>0$  such that the following statements hold. 
\begin{itemize}
\item[(1)]
For arbitrary $u\in C_0^\infty(\Gamma)$ and   ${t\in(0, T]}$ one has
	\begin{align}
		&\|e^{\Delta t}\partial_x u\|_{\hatt L^q(\Gamma)}\leq  Ct^{-\frac12-\frac12\left(\frac1p-\frac1q \right)}\|u\|_{\hatt L^p(\Gamma)},\label{tgradlplq}\\
		&\|\partial_xe^{\Delta t} u\|_{\hatt L^q(\Gamma)}\leq Ct^{-\frac12-\frac12\left(\frac1p-\frac1q \right)}\|u\|_{\hatt L^p(\Gamma)}.\label{tgradlplqnew}
	\end{align}
	In particular, the operators $e^{\Delta t}\partial_x$, $\partial_xe^{\Delta t}$ originally defined on $C_0^{\infty}(\Gamma)$ can be uniquely extended to bounded operators in\footnote{ $\cB(\cX, \cY)$ denotes the space of bounded linear operators between Banach spaces $\cX, \cY$} $\cB(\hatt L^p(\Gamma), \hatt L^q(\Gamma))$ with 
	\begin{align}
	&\|e^{\Delta t}\partial_x \|_{\cB(\hatt L^p(\Gamma), \hatt L^q(\Gamma))}\leq Ct^{-\frac12-\frac12\left(\frac1p-\frac1q \right)},\ {t\in(0, T]},\lb{normtdx}\\
	&\|\partial_x e^{\Delta t} \|_{\cB(\hatt L^p(\Gamma),\hatt L^q(\Gamma))}\leq Ct^{-\frac12-\frac12\left(\frac1p-\frac1q \right)},\ {t\in(0, T]}. \lb{normtdxnew}
	\end{align}

\item[(2)]For arbitrary $u\in\hatt L^p(\Gamma)$, ${t\in(0, T]}$  one has
\begin{equation}
	\lb{new-norm-estimate1}
	||e^{\Delta  t}u||_{\hatt L^q(\Gamma)} \leq C t^{-\frac12(\frac1p -\frac1q)}||u||_{\hatt L^p(\Gamma)},
\end{equation}
and
\begin{equation}
	\lb{new-norm-estimate2}
	\|e^{\Delta t} \|_{\cB(\hatt L^p(\Gamma), \hatt W^{1,q}(\Gamma))}\leq C\max \left\{t^{-\frac12-\frac12\left(\frac1p-\frac1q \right)}, t^{{-\frac12} \left(\frac1p-\frac1q \right)}\right\}. 
\end{equation}
\end{itemize}
\end{theorem}

The analogous estimates for the heat semigroup on $\bbR^n$ and on bounded domains (with Neumann boundary conditions) have been obtained in \cite{MR3620027} and \cite{MR2146345}, respectively. In the settings of non-compact metric graphs, the inequality \eqref{new-norm-estimate1} with $t\in(0, T]$ replaced by $t>0$ and $T-$independent   constant $C$ has been established in \cite{BFM} for unbounded graphs\footnote{In the setting of compact graphs the constant $C$ must be $T$-dependent. Indeed, $u\equiv 1$ does not satisfy the inequality with any $T$-independent $C$ for $p<q$. However, Corollary \ref{cor2.1} does provide some variations of \eqref{tgradlplq}--\eqref{new-norm-estimate2} the operator $e^{(\Delta-\sigma)t}$ with $\sigma\gg 1$ and time-independent constant $C$.}. In this context, we also mention the estimates on the heat kernel discussed in \cite[Proposition 2.2]{MR3698165}. Our proof of Theorem \ref{lplq estimate} employs the combinatorial description of the heat kernel given, discussed, for example, in \cite{BFM, MR1748215}, and integration by parts. 

Another important object in the proof of Theorem \ref{Lp} is the evolution operator for a general linear non-autonomous parabolic equation on a compact graph which is of independent interest. 
\begin{theorem}\lb{evolution-operator-thm}
Let $0< \kappa<T$, $\sigma>0$, and $\psi_0\in \hatt L^q(\Gamma)$ for some $q\geq 1$.  Assume that for some  $\beta>0$  the mappings $[\kappa,T)\ni t\mapsto a(t,\cdot )\in \hatt L^q(\Gamma)$ and $[\kappa,T)\ni t\mapsto b(t,\cdot )\in \hatt L^q(\Gamma)$ belong to  ${ C^\beta}\big([\kappa, T), \hatt L^q(\Gamma)\big)$. Let us consider the non-autonomous parabolic equation on an arbitrary compact metric graph $\Gamma$
\begin{equation}\lb{evolution-main-eq}
\begin{cases}
\partial_t\Psi=(\Delta-\sigma)\Psi+a(t,x)\partial_x\Psi+b(t,x)\Psi,\  \, \,
{ t>\kappa,}\,\,  x\in \Gamma,\\
\sum\limits_{\vartheta \sim e} \partial_{\nu}\Psi_e(t, \vartheta)=0,\ \Psi_e(t, \vartheta)=\Psi_{e'}(t, \vartheta),\ e, e'\sim\vartheta,\,\, t>\kappa\\
\Psi({\kappa}, \cdot)=\psi_0(\cdot)\in \hatt L^q(\Gamma). 
\end{cases}
\end{equation}

Then  the following assertions hold. For arbitrary $\gamma\in(0,1)$, $\tau\in(\kappa, T)$ and some $\varepsilon>0$  there exists a unique ${ \Psi(t)= \Psi(t; \psi_0,\kappa)}$ such that 
\begin{align}
\begin{split}\lb{evolution-eq1}
&\Psi\in C((\kappa, \tau], {\hatt C^{\varepsilon}}(\overline{\Gamma}))\cap C([\kappa, \tau], \el{q}),\,\,  \Psi(t)\in \cX^{\gamma }_q\,\, {\rm for}\,\,  t\in(\kappa,\tau],\\
&\Psi(t)=e^{(\Delta-\sigma)(t-\kappa)}\psi_0+\int_{\kappa}^{t}e^{(\Delta-\sigma)(t-\kappa-s)}\left({  a( s},\cdot)\partial_x\Psi{ (s)}+{  b( s},\cdot)\Psi{(s)}\right)ds,\  t\in[\kappa, \tau],
\end{split}
\end{align}
where the integral converges in $\hatt L^q(\Gamma)$.  Moreover, for $t\in(\kappa,\tau]$ the evolution operator $T(t,\kappa)$  defined by
\begin{align}
\label{evolution-eq2}
&T(t,\kappa): \elq\rightarrow \cX^{\gamma}_q,\ T(t,\kappa)\psi_0={\Psi(t;\psi_0,\kappa)}, \psi_0\in \hatt L^q(\Gamma),
\end{align}
is bounded, i.e., $T(t,\kappa)\in \cB(\hatt L^q(\Gamma), \cX^{\gamma}_q)$,  with $\|T(t,\kappa)\|_{\cB(\hatt L^q(\Gamma), \cX^{\gamma}_q)}\leq C(t-\kappa)^{-\gamma}.$ 

In addition, if the mappings $t\mapsto a(t, \cdot), b(t,\cdot)$ belong to ${ C^{\beta}}\big([\kappa, T), \hatt C(\hatt \Gamma)\big)$ then {$\Psi(t,x)=\Psi(t,x;\psi_0,\kappa)$ defined by
$\Psi(t,x;\psi_0,\kappa):=\Psi(t;\psi_0,\kappa)(x)$}  is a unique classical solution of the Cauchy problem
\eqref{evolution-main-eq} on the interval $[\kappa,{ T)}$.
\end{theorem}

The two key features of the above result that play an important role in the sequel are boundedness of the evolution operator $T(t, \kappa)$ from $\hatt L^q(\Gamma)$ to $\cX^{\gamma}_p$ for {\it arbitrary} $\gamma\in(0,1)$ and the explicit estimate on its norm. We refer the reader to the classical text \cite{Henry} where a version of Theorem \ref{evolution-main-eq} is discussed for parabolic equations of the type $u_t=Au+B(t)u$ on abstract Banach spaces and the evolution operator is shown to act boundedly from $\cX^{\alpha}_p$ to itself (whereas in our case, that is for special parabolic equation, we prove boundedness from $\cX^{0}_p$ to $\cX^{\gamma}_p$). 


To conclude the introduction, let us provide a brief review of related recent results. The vast mathematical literature on chemotaxis models,  Keller--Segel systems includes \cite{MR3351175, MR3698165, MR4434022, MR4646576, PYZ, MR2448428, MR2013508,  IsSh1, IsSh2,  MR3925816, KS1, MR4612099, MR3026554, MR3544323,   MR3294344, MR3620027, MR4019929, MR4506792, MR4470553, MR4421515, MR4392478, MR4104955, MR3397319, MR4019929,  MR2334836, WaYaGa, MR2445771, MR2644137, MR2754053, MR2825180, MR3115832, MR3210023}. A large portion of these papers deals with versions of \eqref{parabolic-parabolic-eq} posed on $\Omega\subseteq\bbR^d$, $d\geq 1$ with specific choices of non-linear terms $f_1, f_2, f_3$ and boundary conditions (typically Neumann); often only one of the two cases, $\tau>0$ or $\tau=0$, is considered. 

One of the main goals of the current work is to prove well-posedness of parabolic-parabolic and parabolic-elliptic Keller--Segel systems on metric graphs. Our motivation is threefold. First, such systems are widely used in bio-engineering models of dermal tissue reparation \cite{HKZ, MaKu, SRD}, where the propagation of fibroblasts (the cells responsible for reparation of dermal tissue) through artificial scaffolds (a network of fibers) is an important topic. In this context, the edges of the graph represent individual fibers whose crossings form the vertices of the model network. Second, in connection with the applied studies, the numerical analysis of the Keller--Segel model on networks has been used to demonstrate interesting dependence of the dynamics on the geometry of the graph \cite{MR3149314, Bretti}. The current work naturally supplements the numerical studies with the rigorous proof of global well-posedness for rather general non-linearities and completely general compact graphs. Our third motivation is to link a large body of work on elliptic differential operators on metric graphs see, for example, \cite{MR2279143, MR2459860, MR1109248, MR2151598, MR1853353, MR3889868, MR2423580, KS2003, MR2022737, Kuch04, VONBELOW1985309, MR0932369} to a rapidly developing theory of parabolic operators  \cite{BFM, MR4581124, MR3698165, MR3431130, MR3243602} and evolution equations on metric graphs. The two papers on evolution equations that are especially closely related to ours are \cite{BFM, MR3698165}. In \cite{MR3698165} the authors considered \eqref{parabolic-parabolic-eq} with $f_1(u,v)=u$, $f_2(u,v)=0$, $f_3(u,v)=u-v$, $\tau\geq 0$ and constructed, via a fixed point argument, a mild solution of the parabolic-parabolic system in the space $L^{\infty}\left((0, T); \hatt C(\Gamma)\right)\times L^{\infty}\left((0, T); \hatt W^{1,\infty}(\Gamma)\right)$ and a global mild solution of the parabolic-elliptic system in $L^{\infty}\left((0, \infty); \hatt C(\Gamma)\right)\times L^{\infty}\left((0, T); \hatt C^{2}(\Gamma)\right)$. In \cite{MR3698165}, the authors also provided a discussion (stemming from \cite{MR1171460, Roth}) of the heat kernel on compact graphs. An in-depth analysis of the heat semigroup, however, appeared  in \cite{BFM} where the authors established analyticity, ultracontractivity, and some heat kernel estimates all of which are employed in our work. We supplement the ultracontractivity estimate in \cite[eq. (4.1)]{MR3698165} by { its} analogue for the operator $e^{t(\Delta-\sigma)}\partial_x$, see Theorem \ref{lplq estimate}. This is a key technical estimate that allows us to treat the chemotaxis sensitivity term $\partial_x(\chi(u,v)\partial_xv)$ which critically distinguishes the Keller--Segel model from the usual reaction-diffusion system.

\smallskip

{\bf Outline of the paper.} In Section \ref{Sec:LptoLq} we discuss the operator $e^{t\Delta}\partial_x$ as a mapping between $\hatt {L}^p(\Gamma)$ and ${\hatt {L}^q(\Gamma)}$ {and prove Theorem \ref{lplq estimate}. In Section \ref{evolution-sec}, we study linear non-autonomous parabolic equations of the form \eqref{evolution-main-eq} on metric graphs and
prove Theorem \ref{evolution-operator-thm}}.  In Sections \ref{mildsol} and \ref{classolsec} we give the proofs of Theorems \ref{Lp} and \ref{global-existence-thm} correspondingly, in particular, subsections \ref{outline1}, \ref{outline2} contain brief outline { of } our proofs. In Appendix \ref{app:FunctionalSpaces} we recall fractional power spaces and interpolation theorems on metric graphs, in Appendix \ref{app:LpLq} we discuss $L^p-L^q$ mapping properties of nonlinear mappings arising in the right-hand side of \eqref{parabolic-parabolic-eq}, in Appendix \ref{appA} we  summarize some important facts about general semi-linear parabolic equations on compact metric graphs.  {Throughout the rest of this paper, if no confusion occurs, we write $\hatt L^p(\Gamma)$  and $\hatt L^q(\Gamma)$ as $L^p(\Gamma)$ and $L^q(\Gamma)$, respectively.}

\section{$L^p-L^q$  estimates for the operator $e^{\Delta t}\partial_x$ and proof of Theorem \ref{lplq estimate}}\label{Sec:LptoLq}

In this section we derive  principal estimates on the norm of the operators $e^{\Delta t}$ and   $e^{\Delta t}\partial_x$, which
will be used to establish well-posedness of \eqref{parabolic-parabolic-eq} subject to \eqref{contcond},\eqref{fluxcond}, and prove Theorem \ref{lplq estimate}.

We first recall from \cite{BFM} the construction of the analytic semigroup or heat semigroup $e^{\Delta t}$.  Let $\Delta$ be the  Neumann--Kirchhoff Laplacian in  $L^p(\Gamma)$, $1\leq p<\infty$ defined as follows
\begin{align}
\begin{split}\lb{apnum}
&\Delta:\dom(\Delta)\subset \lp \rightarrow \lp,\ \Delta u:=\lap u,\  u\in \dom(\Delta),\\
&\dom(\Delta):=\left\{u\in \hatt W^{2,p}(\Gamma):  \sum\limits_{\vartheta \sim e} \partial_{\nu}u_e(\vartheta)=0,\ 
u_e(\vartheta)=u_{e'}(\vartheta), \vartheta\sim e, e'\right\}.
\end{split}
\end{align}

For a given edge $e$ we let $-e$ denote the same edge equipped with opposite orientation, i.e., $t(\pm e)=i(\mp e)$. The graph consisting of all edges $\pm e$ is denoted by $\wti \cG$. We say that edges $e, e'\in \cE(\wti \cG)$ are connected by a path $P$ along $m\in\bbN_0$  edges whenever there exists $m+1$ edges $(e_0, e_1,..., e_m)\subset (\cE(\wti \cG))^{m+1}$ with $e_0:= e$, $e_m:= e'$ and $t(e_k)=i(e_{k+1})$, $0\leq k\leq m-1$, the length of such path is dented by $|P|:=\sum_{k=0}^{m-1}|e_k|$. The collection of all paths connecting $e, e'$ along $m$ edges is denoted by $\cP_{e, e'}(m)$. The scattering coefficient matrix of size $2|\cE|\times 2|\cE|$ (here $|\cE|$ denotes the number of edges of $\Gamma$)  between edges $e, e'\in \cE(\wti \cG)$ is defined via 
\begin{equation}
	S_{e, e'}=
	\begin{cases}
		\frac{2}{{ \deg(t(e))}}, &e\not=-e', t(e)=i(e'),\\
		\frac{2}{{ \deg(t(e))}}-1, &e=-e', t(e)=i(e'),\\
		0, &\text{ otherwise},
	\end{cases}
\end{equation}
where ${ \deg(t(e))}$ denotes the number of edges incident to $t(e)$. 
\begin{figure}[h]
	\includegraphics[scale=0.2]{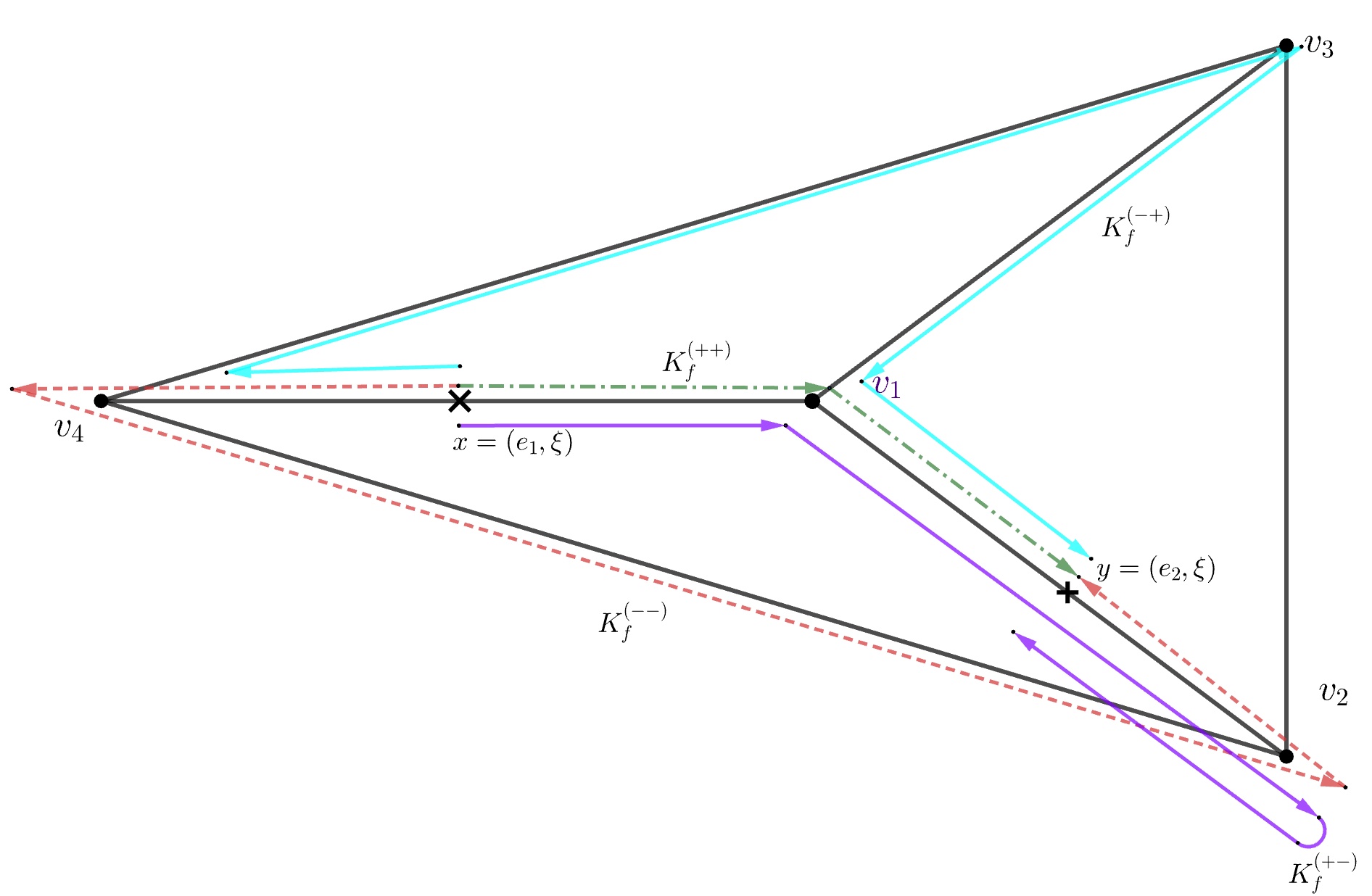}
	\caption{The orientation of edge $e_1$ is given by the choice $i(e_1)=v_4, t(e_1)=v_1$. The orientation of edge $e_2$ is given by the choice $ i(e_2)=v_1, t(e_2)=v_2$.  The four paths shown illustrate four (out of infinitely many) itineraries $x=(e_1, \xi)$, $y=(e_2, \xi')$ that appear in \eqref{itin}. For example, the red dotted path corresponds to $K_f^{(--)}$ while the green one corresponds to $K_f^{(++)}$.}\lb{fig2}
\end{figure}

The scattering coefficient along the path $P=(e_0,..., e_m)\in \cP_{e, e'}(m)$ is given by  $S_P:={ \Pi_{k=0}^{m-1}}S_{e_k, e_{k+1}}$. With this notation in hand let us recall the definition of a convolution operator on a metric graph with kernel $K_f(x,y), x,y\in\Gamma$ for  $f:\bbR\rightarrow\bbR$, $f\in L^1(\bbR)$ satisfying 
\begin{equation}
\sum_{m=0}^{\infty}3^m \|f\|_{L^1(m, \infty)}<\infty. 
\end{equation}
Following \cite[Definition 3.4]{BFM} we  define kernel $K_f$ by the formulas
\begin{align}
\begin{split}\lb{intker}
&K_f: \Gamma\times \Gamma\rightarrow \bbR,\  K_f=K_f^{(++)}+K_f^{(+-)}+K_f^{(-+)}+K_f^{(--)},
\end{split}
\end{align}
where
\begin{align}
&K_f^{(++)}((e, \xi), (e', \xi')):= \delta_{e,e'}f(\xi'-\xi)+\sum_{\substack{m\in\bbN\\ P\in\cP_{e, e'}(m)}} S_Pf(\xi'+|P|-\xi),\lb{pp}
\end{align}
with $\delta_{e,e'}=0$ if the edges $e$ and $e'$ are different, and $\delta_{e, e'}=1$ otherwise, 
\begin{align}
&K_f^{(+-)}((e, \xi), (e', \xi')):= \sum_{\substack{m\in\bbN\\ P\in\cP_{e, -e'}(m)}} S_Pf(|e'|-\xi'+|P|-\xi),\\
& K_f^{(-+)}((e, \xi), (e', \xi')):= \sum_{\substack{m\in\bbN\\ P\in\cP_{-e, e'}(m)}} S_Pf(\xi'+|P|-(|e|-\xi)),\\
&K_f^{(--)}((e, \xi), (e', \xi')):= \sum_{\substack{m\in\bbN\\ P\in\cP_{-e, -e'}(m)}} S_Pf(|e'|-\xi'+|P|-(|e|-\xi)).
\end{align}
The above four terms represent sums along all possible finite paths on $\wti \Gamma$ connecting points $(e, \xi)\in \Gamma$ and $(e', \xi')\in\Gamma$, see Figure \ref{fig2}. Such paths are grouped into four categories depending on their itinerary which we briefly describe as follows
\begin{align}
\begin{split}\lb{itin}
&K^{(++)}_f\text{ terms correspond to paths } (e, \xi)\rightarrow(e, |e|)\rightarrow...(e', 0)\rightarrow(e', \xi'),\\
&K^{(+-)}_f \text{ terms correspond to paths } (e, \xi)\rightarrow(e, |e|)\rightarrow...(e', |e'|)\rightarrow(e', \xi'),\\
&K^{(-+)}_f \text{ terms correspond to paths } (e, \xi)\rightarrow(e, 0)\rightarrow...(e', 0)\rightarrow(e', \xi'),\\
&K^{(--)}_f \text{ terms correspond to paths } (e, \xi)\rightarrow(e, 0)\rightarrow...(e', |e'|)\rightarrow(e', \xi').
\end{split}
\end{align}

Then the convolution operator $f*$ is defined by
\begin{equation}
(f*u)(e, \xi):=\int_{\Gamma}K_f((e,\xi), (e',\xi'))u(e',\xi')d(e',\xi'), \,\,\, u\in \lp,
\end{equation}
or simply
\begin{equation}\lb{hsg}
(	f*u)(\xi):=\int_{\Gamma}K_f(\xi,\xi')u(\xi')d\xi'. 
\end{equation}

The Neumann--Kirchhoff Laplacian generates analytic semigroup $e^{t\Delta}$ {in $L^p(\Gamma)$}  defined by the following convolution operator
\begin{equation}\lb{1.17}
(e^{\Delta t}u)(\xi):=\int_{\Gamma} K_t(\xi, \xi')u(\xi')d\xi', \text{ with\  }K_t(\xi, \xi'):=K_{f_t}(\xi, \xi'), \,\, t>0,
\end{equation}
where
\begin{equation}
\lb{f-t-eq}
f_t(x):=\frac{1}{\sqrt{4\pi t}}e^{-\frac{x^2}{4t}},\,\,\, x\in\bbR,\,\,\,  t>0,
\end{equation}
see \cite[Theorem 4]{BFM}. With these preliminary discussions we are now ready to prove Theorem \ref{lplq estimate}.

\begin{proof} [Proof of Theorem \ref{lplq estimate}]
	(1)
	Let $ u \in C_0^{\infty}(\Gamma)$.   Using \eqref{1.17}, integrating by parts, and noting that the boundary terms vanish,  we obtain
	\begin{equation}
		(e^{\Delta t}\partial_x u)(x)=\int_{\Gamma} K_t(x, y)\partial_y u(y)dy = - \int_{\Gamma} (\partial_y K_t(x, y) )u(y)dy.
	\end{equation}
	
	Recalling  \eqref{intker}, it suffices to show 
	\begin{equation}
		\bigg \|\int_{\Gamma} \partial_y K_t^{(i,j)}(x, y) u(y)dy \bigg \|_{L^q(\Gamma, dx)} \leq { C(T, p, q, \Gamma)}t^{-\frac12-\frac12\left(\frac1p-\frac1q  \right)}\|u\|_{\el{p}},\  i,j\in\{+,-\}
	\end{equation}
	{for all $t\in(0, T]$ and some $C(T, p, q, \Gamma)>0$.}
	
	Below we will derive this inequality for $i=j=+$, the other three cases can be proved  similarly. Using \eqref{pp} we obtain
	\begin{align}
		\lb{i12}
		& \bigg \|\int_{\Gamma} \partial_y K_t^{(++)}(x, y) u(y)dy \bigg \|_{L^q(\Gamma, dx)}\nonumber \\ 
		&=\sum_{e\in\cE}  \bigg\|\sum_{e'\in\cE} \int_{e'} \partial_y \big( \delta_{e,e'}{f_t}(y-x)+\sum_{\substack{m\in\bbN\\ P\in\cP_{e, e'}(m)}}S_P{f_t}(y+|P|-x)   \big) u(y) dy \bigg\|_{L^q(e, dx)}\nonumber \\
		& \leq \sum_{e\in\cE}\bigg \| \int_{e}  \partial_y {f_t}(y-x)u(y) dy \bigg \|_{L^q(e, dx)} +
		{ \sum_{e,e'\in\cE}}\bigg \| \int_{e'} \sum_{\substack{m\in\bbN\\ P\in\cP_{e, e'}(m)}}S_P \partial_y {f_t}(y+|P|-x) u(y) dy \bigg \|_{L^q(e, dx)}\nonumber\\
		&=I_1+I_2,
	\end{align} 
	{where $f_t$ is as in \eqref{f-t-eq}.}
	To estimate $I_1$ let us fix an arbitrary edge $e \in \cE$. Then one has
	\begin{align}
		\bigg\|\int_{e} \partial_y {f_t}(y-x)u(y) dy \bigg \|_{L^q(e,dx)} & \leq  \bigg\|\int_{e} \frac{|y-x|}{2t\sqrt{4\pi t}} e^{\frac{-(y-x)^2}{4t}} |u(y)| dy \bigg \|_{L^q(e, dx)}\nonumber \\
		& \leq \bigg\|\int_{\mathbb{R}} \frac{|y-x|}{2t\sqrt{4\pi t}} e^{\frac{-(y-x)^2}{4t}} |{ \tilde u_e(y)}| dy \bigg \|_{L^q(\mathbb{R}, dx)}, \label{conv_elq}
	\end{align}
	where ${ \tilde u_e(y)}$ is a function defined on $\mathbb{R}$ such that ${ \tilde  u_e}(y)= u(y)$ for $y \in [0, |e|]$ and ${ \tilde u_e}(y)=0$, $y\in\bbR\setminus[0, |e|]$. 	Next, introducing an auxiliary function 
	\begin{equation}\lb{hzt}
		h_t(z):=\frac{|z|}{2t\sqrt{4\pi t}} e^{\frac{-z^2}{4t}},
	\end{equation}
	we notice that the right-hand side of \eqref{conv_elq} is equal to $\|h_t*|u_e|\|_{L^q(\mathbb{R})}$, hence, for $r\in[1,\infty)$ satisfying $q^{-1} +1 = r^{-1} + p^{-1}$ one has
	\begin{align}
		\lb{I1bounde}
		\bigg\|\int_{e} \partial_y{ f_t}(y-x)u(y) dy \bigg \|_{L^q(e,dx)}&\leq \|h_t*|{ \tilde u_e}|\|_{L^q(\mathbb{R})} \leq  \|h_z \|_{L^r(\mathbb{R})}  \| { \tilde u_e}\|_{L^p(\mathbb{R})}\nonumber\\
		&=\|h_z \|_{L^r(\mathbb{R})}  \| u\|_{L^p(e)}\leq   C t^{-\frac{1}{2}-\frac{1}{2}(\frac1p-\frac1q)}\| u\|_{L^p(e)},  
	\end{align}
	where we used 
	\begin{align}
		\|h_t \|_{L^r(\mathbb{R})} &= C t^{-\frac{1}{2}-\frac{1}{2}(\frac1p-\frac1q)},\ C=\frac{2^{1/r} \||y| e^{-y^2} \|_{L^r(\bbR, dy)}}{2\sqrt\pi}.
	\end{align}
	Summing over all (finitely many) edges $e\in\cE$ and using \eqref{I1bounde} we obtain
	\begin{equation}\lb{I1boundG}
		I_1\leq  C t^{-\frac{1}{2}-\frac{1}{2}(\frac1p-\frac1q)}\| u\|_{L^p(\Gamma)}
	\end{equation}
	for $C=C(p,q, \Gamma)$, $t>0$. 
	
	Let us now derive a similar estimate for $I_2.$ To that end, let us fix  a pair of edges $e,e'\in\cE$ and a path $P$, $|P|>0$, connecting $e, e'$. Then for $x\in e$ one has
	\begin{align}
		\lb{1.8}
		& \int_{e'} \big| S_P \partial_y {f_t}(y+|P|-x) u(y)\big| dy\leq \int_{e'}\frac{|y+|P|-x|}{2t\sqrt{4\pi t}} e^{\frac{-(y+|p|-x)^2}{4t}}   |u(y)| dy\nonumber \\
		& \leq \frac1 {\sqrt{t}} \int_{e'} \bigg(\frac{|y+|p|-x|}{2\sqrt{t}} e^{\frac{-(y+|p|-x)^2}{8t}}\bigg) \bigg(\frac{1}{\sqrt{4\pi t}}e^{\frac{-(y+|p|-x)^2}{8t}}  |u(y)|\bigg) dy\nonumber \\
		& \leq \frac1 {\sqrt{t}}C_0 \int_{e'}  \frac{1}{\sqrt{4\pi t}}e^{\frac{-(y+|p|-x)^2}{8t}}  |u(y)|dy,\ t>0,
	\end{align}
	where
	\begin{equation}
		C_0:=\sup_{z\geq 0, t>0}\bigg(\frac{z}{2\sqrt{t}} e^{\frac{-z^2}{8t}}\bigg) \leq  \sqrt{2} \sup_{\zeta \in [0,\infty)} \big(\zeta e^{-\zeta^2} \big) <\infty.
	\end{equation}
	Let us notice that the sum appearing in $I_2$ is taken over all paths in $\cE(\wti \Gamma)$ connecting $e$ and $e'$. In general, the number of such paths grows exponentially with respect to the length of the path, that is, the number of paths of length $m\in\bbN$ can be as high as $\deg^m$   where $\deg$ is the largest vertex degree. The argument we give below stems from the fact that this rapid growth is balanced by an even faster decay of the heat kernel which is of order $e^{-m^2}$. To facilitate this argument we split the sum in $I_2$ into the sum over paths of large length and those whose length is order one. Concretely,  let 
	\begin{equation}
		R:=(2\ell_+^2 + \sel{8T}\log{(\deg)}+1)\ell_-^{-2},
	\end{equation}
	where $\ell_{\pm}$ are lengths of the longest and shortest edges. For $e, e'\in\cE$ and $x\in e$, using \eqref{1.8}, we obtain
	\begin{align}
		& \sum_{\substack{m\in\bbN\\ P\in\cP_{e, e'}(m)}} \int_{e'} \big| S_P \partial_y f(y+|P|-x) u(y)\big| dy   \leq   \frac{C_0} {\sqrt{t}} \sum_{\substack{m\in\bbN\\ P\in\cP_{e, e'}(m)}} \int_{e'}  \frac{1}{\sqrt{4\pi t}}e^{\frac{-(y+|P|-x)^2}{8t}}  |u(y)|dy \nonumber\\
		& \leq \frac{C_0} {\sqrt{t}} \bigg(\sum_{\substack{m \leq R\\ P\in\cP_{e, e'}(m)}} \int_{e'}  \frac{1}{\sqrt{4\pi t}}e^{\frac{-(y+|P|-x)^2}{8t}}  |u(y)|dy + \sum_{\substack{m > R\\ P\in\cP_{e, e'}(m)}} \int_{e'}  \frac{1}{\sqrt{4\pi t}}e^{\frac{-(y+|P|-x)^2}{8t}}  |u(y)|dy \bigg)\nonumber\\
		&= \frac{C_0} {\sqrt{t}} \bigg(A+B\bigg). \lb{abformula}
	\end{align}
	\sel{Let us note that for a path $P\in\cP_{e,e'}(m)$ one has  $|P| \geq m\ell_-$ and $(y-x)|P| \geq -m\ell_+^2$. Then for $t\in(0,T]$ one has}
	\begin{align}
		B=&\sum_{\substack{m > R \\ P\in\cP_{e, e'}(m)}} \int_{e'}  \frac{1}{\sqrt{4\pi t}}e^{\frac{-(y+|P|-x)^2}{8t}}  |u(y)|dy \nonumber\\
		& \leq \sum_{m > R} \deg^m \int_{e'}  \frac{1}{\sqrt{4\pi t}}e^{\frac{-(y-x)^2 -2(y-x)|P|-|P|^2}{8t}}  |u(y)|dy \nonumber\\
		& \leq \sum_{m > R} \deg^m \int_{e'}  \frac{1}{\sqrt{4\pi t}}e^{\frac{-(y-x)^2 +2m\ell_+^2-(ml_-)^2}{8t}}  |u(y)|dy \nonumber\\
		&\leq  \sum_{m\in\bbN,m > R }  e^{\frac{m(2\ell_+^2+8t\log(\deg) - R\ell_-^2)}{8t}} \int_{e'}  \frac{1}{\sqrt{4\pi t}}e^{\frac{-(y-x)^2}{8t}}  |u(y)|dy \nonumber \\
		&\leq  \sum_{m\in\bbN }  e^{-\frac{m}{8t}} \int_{e'}  \frac{1}{\sqrt{4\pi t}}e^{\frac{-(y-x)^2}{8t}}  |u(y)|dy \leq \frac{1}{1-e^{-\frac{1}{8t}}} \int_{\mathbb{R}}  \frac{1}{\sqrt{4\pi t}}e^{\frac{-(y-x)^2}{8t}}  |{ \tilde u_{e'}(y)}|dy,\lb{aux1.9}
	\end{align}
	where, as before, we let  ${ \tilde u_{e'}}$ denote  the extension of $u$ by zero outside of $(0, |e'|)$.   Next, introducing an auxiliary function 
	\begin{equation}
		\varphi_t(z):=\frac{1}{\sqrt{4\pi t}}e^{-z^2/8t}, \text{ with\  } \|\varphi_t\|_{L^r(\bbR)}\leq C t^{{-\frac12}+\frac1{2r}}, r\geq 1, C=C(r)>0, 
	\end{equation}
	we use \eqref{aux1.9} to obtain
	\begin{align}
		\lb{bterm}
		\|B\|_{L^q(e)}&\leq \frac{\|\varphi_t*(|u_{e'}|)\|_{L^q(e)}}{1-e^{-\frac{1}{8t}}}\leq \frac{\|\varphi_t\|_{L^r(\bbR)} \|{ \tilde u_{e'}}\|_{L^p(e')}}{1-e^{-\frac{1}{8t}}}\nonumber\\
		& \leq C\frac{t^{{-\frac12}+\frac1{2r}}}{1-e^{-\frac{1}{8t}}}{ \|\tilde u_{e'}\|_{L^p(e')}},
	\end{align}
	where $1+q^{-1}=r^{-1}+p^{-1}$. 
	
	In order to estimate term $A$ in \eqref{abformula} we change variables $z= x-y-|P|$ and obtain
	\begin{align}
		\lb{finpart}
		& \sum_{\substack{m\in\bbN, m \leq R\\ P\in\cP_{e, e'}(m)}} \int_{e'}  \frac{1}{\sqrt{4\pi t}}e^{\frac{-(y+|P|-x)^2}{8t}}  |u(y)|dy =\sum_{\substack{m\in\bbN, m \leq R\\ P\in\cP_{e, e'}(m)}} \int_{\bbR}  \frac{1}{\sqrt{4\pi t}}e^{\frac{-(y+|P|-x)^2}{8t}}  |{ \tilde u_{e'}(y)}|dy\nonumber \\
		&\leq  \sum_{\substack{m\in\bbN, m \leq R\\ P\in\cP_{e, e'}(m)}}\int_{\mathbb{R}}  \frac{1}{\sqrt{4\pi t}}e^{\frac{-z^2}{8t}}  |{ \tilde u_{e'}}(x-z-|P|)|dz, 
	\end{align}
	for arbitrary $x\in e$. Then one has
	\begin{align}\lb{aterm}
		\|A\|_{L^q(e)}\leq  \sum_{\substack{m\in\bbN, m \leq R\\ P\in\cP_{e, e'}(m)}} C\| \varphi_t\|_{L^r(\mathbb{R})}\| { \tilde u_{e'}}(\cdot-|P|) \|_{L^p(\mathbb{R})}
		\leq C(\sel{T}, p,q, \Gamma)t^{-\frac12(\frac1p-\frac1q)} { \|u\|_{L^p(e')}.}
	\end{align}
	Finally, combining \eqref{I1boundG}, \eqref{abformula}, \eqref{bterm}, \eqref{aterm} we obtain
	\begin{align}
		\|e^{(\Delta -\sigma )t}\partial_x u\|_{L^q(\Gamma)}\leq C t^{-1/2}\left(t^{-\frac{1}{2}(\frac1p-\frac1q)}+\frac{t^{-\frac{1}{2}(\frac1p-\frac1q)}}{1-e^{-\frac{1}{8t}}}\right)\| u\|_{L^p(\Gamma)}, \sel{t\in(0, T]},
	\end{align}
	where $C=C(T, p, q, \Gamma)$ as asserted in \eqref{tgradlplq} and \eqref{normtdx}. The proof of  \eqref{tgradlplqnew} and \eqref{normtdxnew} is analogous.

	(2) the derivation of \eqref{new-norm-estimate1} is analogous to that of \eqref{tgradlplq}. The inequality \eqref{new-norm-estimate2}
	follows from \eqref{normtdxnew} and  \eqref{new-norm-estimate1}.
\end{proof}

\begin{corollary}\lb{cor2.1}Assume the setting of Theorem \ref{lplq estimate}. Then there exists $\sigma_0$ such that the following assertions hold. For every $\sigma>\sigma_0$ there exists $\delta\in(0,\sigma)$\footnote{\sel{in general, $\delta\ll 1$}} and $C=C(\sigma, \delta, p, q, \Gamma)>0$ such that for all $t>0$ one has
	\begin{align}
		&\lb{nnew-norm-estimate1}
		||e^{(\Delta -\sigma )t}||_{\cB(  L^p(\Gamma),  L^q(\Gamma))} \leq Ce^{-\delta t} t^{-\frac12(\frac1p -\frac1q)},\\
		&\|e^{(\Delta-\sigma) t}\partial_x \|_{\cB(  L^p(\Gamma),   L^q(\Gamma))}\leq Ce^{-t\delta}t^{-\frac12-\frac12\left(\frac1p-\frac1q \right)}, \lb{nnormtdx}\\
			&\|\partial_xe^{(\Delta-\sigma) t} u\|_{\cB(  L^p(\Gamma),   L^q(\Gamma))}\leq Ce^{-t\delta}t^{-\frac12-\frac12\left(\frac1p-\frac1q \right)}.\label{ntgradlplqnew}
		\end{align}
\end{corollary}
\begin{proof}Let us prove \eqref{nnew-norm-estimate1}, the other two inequalities can be proved analogously. Let $C(T)>0$ abbreviate the constant $C(T, p,q, \Gamma)$ from Theorem \ref{lplq estimate}. For arbitrary $t>0$ let us write $t=\tau +n$ for $\tau\in(0, 1]$ and $n\in\bbN_0$. 
	
	If $n=0$ then by \eqref{new-norm-estimate1} one has
	\begin{align}
	||e^{(\Delta -\sigma )t}||_{\cB( L^p(\Gamma), L^q(\Gamma))}\leq C(1) e^{-\sigma t} t^{-\frac12(\frac1p -\frac1q)}.
	\end{align} 
Assume now that $n\geq 1$ then one has
\begin{align}
	&||e^{(\Delta -\sigma )t}||_{\cB( L^p(\Gamma), L^q(\Gamma))} =	||e^{(\Delta -\sigma )(\tau+n)}||_{\cB( L^p(\Gamma), L^q(\Gamma))}\\
	&\leq (C(1))^{n-1} e^{-\sigma t}\|e^{  \Delta(\tau+1)  }||_{\cB( L^p(\Gamma), L^q(\Gamma))}\\
	&\leq (C(1))^{n-1} C(2)e^{-\sigma t}(1+\tau)^{-\frac12\left(\frac1p-\frac1q\right)}. 
\end{align}
Without loss of generality we may assume that $\alpha:=\log(\max(C(1), C(2)))>0$. Then for $\sigma> \alpha$ the above chain of inequalities yields
\begin{align}
	||e^{(\Delta -\sigma )t}||_{\cB( L^p(\Gamma), L^q(\Gamma))}&=	||e^{(\Delta -\sigma )(\tau+n)}||_{\cB( L^p(\Gamma), L^q(\Gamma))}\leq e^{n\alpha-\sigma n-\sigma \tau}(1+\tau)^{-\frac12\left(\frac1p-\frac1q\right)}\\
	& =Ce^{-(\sigma-\alpha)t}\leq C e^{-(\sigma-\alpha)t/2}t^{{-\frac12\left(\frac1p-\frac1q\right)}}, 
\end{align}
where $C=C(\alpha,\sigma, p,q)>0$. 
\end{proof}

\section{Linear non-autonomous parabolic equations   and proof of Theorem \ref{evolution-operator-thm}} \label{evolution-sec}

In this section,  we study the linear non-autonomous parabolic equation \eqref{evolution-main-eq} and prove Theorem \ref{evolution-operator-thm}.

\begin{proof} [Proof of Theorem \ref{evolution-operator-thm}]
Let us define an operator
\begin{equation}
	\label{veq}
	 V(t)\psi:=a({  t,\cdot}) \psi_x+b({  t,\cdot})\psi, \quad \psi\in \cX^{\alpha}_q.
\end{equation}
We notice that if $\alpha>0$ is such that $ \cX^{\alpha}_q\hookrightarrow \hatt C^1(\Gamma)$ then  one has
\begin{align}
\|V(t)\psi\|_{\elq}\leq C \|\psi\|_{\cX^{\alpha}_q}\left( \|a(t,\cdot)\|_{\elq}+\|b(t,\cdot)\|_{\elq}\right),
\end{align}
and
\begin{equation}\lb{vee1}
\cV\in C^{\beta}\left((0, T), \cB(\cX^{\alpha}_q,  L^q(\Gamma)) \right),
\end{equation} 
where we used H\"older continuity of $a,b$ in time. 

Next, consider  integral equation \eqref{evolution-eq1}. Our first objective is to prove existence of solutions with initial data drawn from a fractional power space $\cX^{\alpha}_q$. To that end, fix arbitrary $\psi_0\in\cX^{\alpha}_q$. Then the  H\"older continuity of $t\mapsto V(t)$ on $[\kappa,\tau]\subset (0,T)$ and
Theorem \ref{general-existence-thm}\footnote{we also invoke \cite[Cor. 3.3.5]{Henry} to get solution $\Psi$ not just on a small interval near $\tau$ but on $[\tau, T)$.} yield the existence of a unique $\Psi= \Psi(t; \psi_0,\kappa)$ such that 
\begin{align}
&\hspace{1.5cm}\Psi\in C([\kappa, \tau], \elq), \Psi(t; \psi_0,\kappa)\in \cX^{\alpha }_q,\\
&\Psi(t;\psi_0,\kappa)=e^{(\Delta-\sigma)(t-\kappa)}\psi_0+\int_{\kappa}^{t}e^{(\Delta-\sigma)(t-\kappa-s)}V(s)\Psi(s;\psi_0,\kappa)ds, t\in[\kappa, \tau],\lb{psiint}
\end{align}
where the integral converges in $\elq$. By \cite[Theorem 7.1.1 (b)]{Henry}, for arbitrary $\gamma\in(0,1)$ we have  $\Psi(t;\psi_0,\kappa)\in \cX^{\gamma}_q$ whenever $\psi_0\in \dom(\Delta)$. This gives rise to the evolution operator $T(t,\kappa)$  defined by
\begin{align}\lb{ttau}
&T(t,\kappa): \dom(T(t,\kappa)):=\dom(\Delta)\subset  \elq\rightarrow \cX^{\gamma}_q,\quad  T(t,\kappa)\psi_0=\Psi(t;\psi_0,\kappa).
\end{align}

We then show that $T(t,\tau)$ can be extended to a bounded linear operator in $\cB(\elq, \cX^{\gamma}_q)$ satisfying 
$$\|T(t,\kappa)\|_{\cB(\elq, \cX^{\gamma}_q)}\leq C(t-\kappa)^{-\gamma}.
$$
To that end, we may assume without loss of generality that $1<2\gamma-\frac1q$ so that $\cX^{\gamma}_q\hookrightarrow  \hatt C^1(\Gamma)$. We  first show that there exists a constant $C>0$ such that 
\begin{equation}\lb{evolbound}
\left\| \Psi(t; f,\kappa)-\Psi(t; g,\kappa) \right\|_{\cX^{\gamma}_q}\leq C (t-\kappa)^{-\gamma}\|f-g\|_{\elq},
\end{equation}
for all $t\in(\kappa,\tau]$, $f,g\in \dom(\cA_q)$. Since $\Psi(\cdot; f,\kappa)$, $\Psi(\cdot; g,\kappa)$ satisfy  \eqref{psiint}, one has
\begin{align}\lb{gronprep1}
&\|\Psi(t;f,\kappa)-\Psi(t;g,\kappa)\|_{\cX^{\gamma}_q}\leq C\|(\sigma-\Delta)^{\gamma}e^{(\Delta-\sigma)(t-\kappa)}(f-g)\|_{\elq}\nonumber\\
&\quad +C\int_{\kappa}^t\|(\sigma-\Delta)^{\gamma}e^{(\Delta-\sigma)(t-\kappa-s)}V(s)(\Psi(s;f,\kappa)-\Psi(s;g,\kappa))\|_{\elq}ds,
\end{align}  
and, using \eqref{fracpow}, 
\begin{align}
\lb{psiminuspsi}
&\|\Psi(t;f,\kappa)-\Psi(t;g,\kappa)\|_{\cX^{\gamma}_q}\leq C(t-\kappa)^{-\gamma}\|(f-g)\|_{\elq}\nonumber\\
&\quad +C \int_{\kappa}^t(t-\kappa-s)^{-\gamma}\|V(s)(\Psi(s;f,\kappa)-\Psi(s;g,\kappa))\|_{\elq}ds.
\end{align} 
We claim 
\begin{equation}\lb{vpsiqg}
\|V(s)(\Psi(s;f,\kappa)-\Psi(s;g,\kappa))\|_{\elq}\leq C\|(\Psi(s;f,\kappa)-\Psi(s;g,\kappa))\|_{\cX^{\gamma}_q}. 
\end{equation}
Indeed,  one has
\begin{align}
\lb{vpsi}
&\|V(s)(\Psi(s;f,\kappa)-\Psi(s;g,\kappa))\|_{\elq}\nonumber\\
&\leq \| a(s,\cdot)(\Psi_x(s;f,\kappa)-\Psi_x(s;g,\kappa))\|_{\elq} + \|b(s,\cdot)(\Psi(s;f,\kappa)-\Psi(s;g,\kappa))\|_{\elq}\nonumber\\
&\leq C{ (\|a(s,\cdot)\|_{\elq}+\|b(s,\cdot)\|_{\elq})}\|(\Psi(s;f,\kappa)-\Psi(s;g,\kappa))\|_{\cX^{\gamma}_q},
\end{align}
where in the last step we used the embedding $\cX^{\gamma}_q\hookrightarrow \hatt C^1(\Gamma)$.  Since $s\mapsto a(s,\cdot)$, $s\mapsto b(s,\cdot)$ are H\"older continuous on $[\kappa,\tau]$, the required inequality \eqref{vpsiqg} follows from \eqref{vpsi}. Finally, combining  \eqref{psiminuspsi}, \eqref{vpsiqg} and Gronwall's inequality  \eqref{gronwell1}, \eqref{gronwell2} we obtain \eqref{evolbound}.

Next, we construct a bounded extension of $T(t,\tau)$ introduced in \eqref{ttau}. Let $\psi_n\in\dom(\cA_q)$ be such that $\|\psi_n-\psi_0\|_{\elq}\rightarrow 0$ as $ n\rightarrow\infty$. Then as we just proved one has $\Psi(t;\psi_n,\kappa)\in \cX^{\gamma}_q$ and 
\begin{align}\lb{paslim}
\Psi(t;\psi_n,\kappa)=e^{(\Delta-\sigma)(t-\kappa)}\psi_n+\int_{\kappa}^{t}e^{(\Delta-\sigma)(t-\kappa-s)}V(s)\Psi(s; \psi_n,\kappa)ds, t\in [\kappa, \tau].
\end{align}	
Using \eqref{evolbound} with $f=\psi_n$, $g=\psi_m$ we infer that the sequence $\{\Psi(t; \psi_n,\kappa)\}_{n=1}^{\infty}$ converges in $\cX^{\gamma}_q$ to $\Psi_\infty(t)$ for $t\in(\tau, \kappa]$. To pass to the limit in \eqref{paslim} we observe that 
\begin{align}
&\|V(s)\Psi(s;\psi_n,\kappa)\|_{\elq}\leq  \|V(s)\|_{\cB(\cX^{\gamma}_q, \elq)}\|\Psi(s; \psi_n,\kappa)\|_{\cX^{\gamma}_q}\nonumber\\
&\quad \leq  C (s-\kappa)^{{-\gamma}}\|V(s)\|_{\cB(\cX^{\gamma}_q, \elq)}\|\psi_n\|_{\elq}\leq  C (s-\kappa)^{{-\gamma}},  s\in(\kappa, \tau],
\end{align}
where we used continuity of $s\mapsto V(s)\in \cB(\cX^{\gamma}_q, \elq)$, boundedness of  $\{\|\psi_n\|_{\elq}\}_{n=1}^{\infty}$, and $\gamma\in(0,1)$. Finally, recalling that $\psi_n\rightarrow \psi_0$ in $\elq$, we pass to the limit as  $n\rightarrow\infty$ in \eqref{paslim} which yields
\begin{align}
\label{approx-eq1}
\Psi_\infty(t)=e^{(\Delta-I)(t-\kappa)}\psi_0+\int_{\kappa}^{t}e^{(\Delta-I)(t-\kappa-s)}V(s)\Psi_\infty(s)ds,\  t\in [\kappa, \tau].
\end{align}	
We claim  that $\Psi_\infty(t)$ does not depend on the choice of the approximating sequence $\{\psi_n\}$ but only on $\psi_0$. To that end,
let $\wti \psi_n\in\dom(\cA_q)$ be such that $\|\wti \psi_n-\psi_0\|_{\elq}\rightarrow 0, n\rightarrow\infty$. Proceeding as above, we note that  $\{\Psi(t;\wti  \psi_n,\kappa)\}_{n=1}^{\infty}$ converges in $\cX^{\gamma}_q$ to $\wti \Psi_\infty(t)$ for $t\in(\kappa,\tau]$
	and 
	\begin{align}
	\label{approx-eq2}
	\tilde \Psi_\infty(t)=e^{(\Delta-\sigma)(t-\kappa)}\psi_0+\int_{\kappa}^{t}e^{(\Delta-\sigma)(t-\kappa-s)}V(s)\tilde \Psi_\infty(s)ds,\  t\in [\kappa, \tau].
	\end{align}	
	By \eqref{approx-eq1} and \eqref{approx-eq2}, 
	\begin{equation}
	\Psi_\infty(t)-\wti\Psi_\infty(t)=\int_{\kappa}^t e^{(\Delta-\sigma)(t-\kappa-s)}V(s)\big(\Psi_\infty(s)-\wti\Psi_\infty(s)\big)ds,\  t\in (\kappa,\tau],
	\end{equation}
	hence, for $t\in (\kappa, \tau]$ one has
	\begin{align}
	&\|\Psi_\infty(t)-\wti \Psi_\infty(t)\|_{\elq} \le C\int_\kappa^t \|\Psi_\infty(s)-\wti \Psi_\infty(s)\|_{\mathcal{X}_q^\gamma}ds\nonumber\\
	&\quad=C\int_\kappa^t \lim_{n\to\infty} \|\Psi(s;\psi_n)-\Psi(s;\tilde\psi_n)\|_{\mathcal{X}_q^\gamma} ds\nonumber\\
&\quad \le  C\int_{\kappa}^t (s-\kappa)^{-\gamma}\lim_{n\to\infty}\|\psi_n-\wti\psi_n\|_{L^q}ds=0. 
	\end{align}
	It then follows that
$
	\Psi_\infty(t)= \tilde \Psi_\infty(t), t\in[\kappa,\tau].
$
Then we define
$
	T(t,\kappa)\psi_0:=\lim_{n\to\infty}T(t,\kappa)\psi_n,
$
	and note that $T(t,\kappa)$ is a single-valued linear operator defined on $\elq$. Moreover, denoting $\Psi(t;\psi_0,\kappa):=T(t,\kappa)\psi_0$ one has
	$$
	\Psi(t;\psi_0,\kappa)=e^{(\Delta-\sigma)(t-\kappa)}\psi_0+\int_\kappa^t  e^{(\Delta-\sigma)(t-\kappa-s)} V(s) \Psi(s;\psi_0,\kappa)ds,\  t\in [\kappa,\tau],
	$$
	and by \eqref{evolbound} with $f=\psi_0$, $g=0$
	$$
	\left\| T(t,\kappa)\psi_0\right\|_{\cX^{\gamma}_q}=\left\| \Psi(t; \psi_0,\kappa) \right\|_{\cX^{\gamma}_q}\leq C (t-\kappa)^{-\gamma}\|\psi_0\|_{\elq}.
	$$
This gives a unique solution of \eqref{evolution-eq1} with initial condition in $\elq$ as well as the asserted bound for the norm of the evolution operator. 

Let us now prove the last part of the theorem. Since $s\mapsto V(s)\in\cB(\cX^{\alpha}_q, \elq)$ is H\"older continuous on $[\kappa, \tau]$, Theorem \ref{strong-solu-thm} with $f(t,x, u,v):=a(t,x)v+b(t,x)u,
$
yields a strong solution, cf. Definition \ref{strong-solution-def}, ${  \Psi(t,x; \psi_0,\kappa):=\Psi(t;\psi_0,\kappa)(x)}$ of \eqref{evolution-main-eq} in $\elq$. We claim that this solution is in fact classical. To prove this claim, let us fix an arbitrary $\varepsilon\in(0, { T-\kappa})$ and note that ${  \Psi_0:=}\Psi(t,\cdot; \kappa+\varepsilon, \psi_0)\in \cX^{\gamma}_q$ for all $\gamma\in(0,1)$. Finally, by picking $\gamma>0$ so that $2\gamma-p^{-1}>1$ and using Theorem  \ref{general-existence-thm} with initial time $\kappa+\varepsilon$ we conclude that 
$
{  \Psi(\cdot+\varepsilon,\cdot;\psi_0,\kappa)=}\Psi(\cdot, \cdot;\Psi_0,\kappa+\varepsilon)\in\hatt C^{1,2}({  (\kappa+\varepsilon,T)} \times\overline\Gamma),
$
as required. 
\end{proof}

\section{Local existence of classical solutions for general chemotaxis model}\lb{mildsol}

The main purpose of this section is to provide the proof of Theorem \ref{Lp}. 

\subsection{Outline of the proof of Theorem \ref{Lp}}\lb{outline1}Noting that it suffices to consider two cases, $\tau=0$ and $\tau=1$ (the proof for $\tau>0$ is similar), we outline the strategy of the proof next. To begin, we observe that the Duhamel's principle applied to the reaction-advection-diffusion system  \eqref{parabolic-parabolic-eq} with $\tau=1$ gives rise to a system of integral equations
\begin{equation}
	\label{tauone}
	\begin{cases}
		u(t)=e^{(\Delta-\sigma)t}u_0-\int_{0}^{t}  e^{(\Delta-\sigma)(t-s)}\partial_{x}\big(f_1(u(s), v(s))\partial_xv(s)\big)ds\cr
		\hspace{3cm}+ \int_{0}^{t}  e^{(\Delta-\sigma)(t-s)}(f_2(u(s), v(s))+\sigma u(s))ds,\cr
		v(t)=e^{(\Delta-\sigma I)t}v_0+ \int_{0}^{t}  e^{(\Delta-\sigma)(t-s)}\left(f_3(u(s), v(s))+\sigma v(s)\right)ds, \cr
	\end{cases}
\end{equation}
and, similarly, in case $\tau=0,$ to the following integral equation
 \begin{align}
 	\begin{cases}\lb{duh}
 		&u(t)=e^{(\Delta-\sigma)t}u_0+\int_{0}^{t}  e^{(\Delta-\sigma)(t-s)}\partial_x(f_1(u(s), v(s))\partial_x v(s))ds\\
 		&\hspace{4cm}+\int_{0}^{t}  e^{(\Delta-\sigma)(t-s)}\left(f_1(u(s), v(s))+\sigma v(s)\right)ds,\\
 		&v(s):=-(\Delta-\sigma )^{-1}{f(u(s))},
 	\end{cases}
 \end{align}
where  we used  $f_3(u,v)=f(u)-\sigma v$  {  in \eqref{duh}} (see Hypothesis \ref{hyp}).

Our first objective, see Section \ref{sec:inteq}, is to  show that the integrals in \eqref{tauone}, \eqref{duh} are convergent in suitable Sobolev and fractional power spaces $\cX^{\beta}_p$ whenever $u\in \elpeet{T}$, $v\in\wpeet{T}$, see Propositions \ref{intconv}, \ref{prop1.3}. In Section \ref{sec:inteq}, we also show that these integrals, as functions of $t$, are H\"older continuous, see Proposition \ref{prop1.4}. Our proofs rely on the $L^p-L^q$ mapping properties of nonlinear operators discussed in Appendix \ref{app:LpLq}. 

Our second objective is to construct solutions of \eqref{tauone} in 
{\begin{equation}\lb{intdata1}
	C([0, T], L^p(\Gamma)\times L^p(\Gamma))\cap C((0, T_{\max}), L^p(\Gamma)\times\wpee),
\end{equation} with initial data $(u_0, v_0)\in L^p(\Gamma)\times\wpee$} and solutions of \eqref{duh} in
\begin{equation}\lb{intdata2}
	\text{$C([0, T_{\max}), L^p(\Gamma))$ with initial data $u_0\in\elpee$,}
\end{equation}  
for a maximal $T_{\max}\in(0,\infty]$ via a fixed point argument and an extension argument. {These solutions will be referred to as {\it mild solutions of \eqref{parabolic-parabolic-eq} subject to \eqref{contcond},\eqref{fluxcond}}}. To that end, for $T>0$ and $(u_0, v_0)\in L^p(\Gamma)\times\wpee$, we formally define a pair of mappings 
\begin{align}
	&\Lambda(u_0,v_0): (u,v)\mapsto \Lambda (u,v;u_0,v_0)=(\Phi (u,v;u_0)), \Psi(u,v;v_0)) \label{lambda}
\end{align}
for  $ (u,v)\in C([0, {T]}, L^p(\Gamma)\times\wpee)$, and
\begin{align}
	&\Theta(u_0): u\mapsto \Theta(u;u_0)\quad {\rm for}\quad u\in C([0, {T]}, L^p(\Gamma))\label{theta}
\end{align}
by the following formulas
\begin{align}
	\begin{split}\label{phi}
		&(\Phi(u,v;u_0)(t)=e^{(\Delta-\sigma)t}u_0-\int_{0}^{t}  e^{(\Delta-\sigma)(t-s)}\partial_{x}\big(f_1(u(s), v(s))\partial_xv(s)\big)ds\\
		&\hspace{4cm}+ \int_{0}^{t}  e^{(\Delta-\sigma)(t-s)}(f_2(u(s), v(s))+\sigma u(s))ds,\\
	\end{split}\\
	&(\Psi(u,v;v_0))(t)=e^{(\Delta-\sigma I)t}v_0+ \int_{0}^{t}  e^{(\Delta-\sigma)(t-s)}\left(f_3(u(s), v(s))+\sigma v(s)\right)ds,\label{psi}\\
	\begin{split}
		&(\Theta(u;u_0))(t):= (\Phi(u, -(\Delta-\sigma )^{-1}f(u))(t),\lb{geq} t\in {[0,T]}.
	\end{split}
\end{align}
In Section \ref{Sec:mildsol}, we prove that $\Lambda, \Theta$ are contractions for small $T$ provided that {$v_0\in \cX_p^{\frac{1}{2}+\epsilon}$, for some $\epsilon\in(0,1)$} (see Proposition  \ref{mild-solution-prop1}). {Then we use the fixed points of these mappings  together with some approximation arguments, to obtain a unique solution $(u(t),v(t))=(u(t;u_0,v_0),v(t;u_0,v_0))$ to \eqref{tauone} (resp. a unique solution $(u(t),v(t))=(u(t;u_0),v(t;u_0))$ to \eqref{duh}) on a small interval $[0,T]$ for initial data $(u_0, v_0)\in\el{p}\times\wpee$, see Proposition  \ref{mild-solution-prop2}.
By  an extension argument we obtain a unique solution $(u(t),v(t))=(u(t;u_0,v_0),v(t;u_0,v_0))$ to \eqref{tauone} (resp. a unique solution $(u(t),v(t))=(u(t;u_0),v(t;u_0))$ to \eqref{duh}) { on the maximal interval  $[0,T_{\max})=[0,T_{\max}(u_0,v_0))$ (resp.   on the maximal interval $[0,T_{\max})=[0,T_{\max}(u_0))$} with the following regularity properties (see Theorem \ref{thm3.1}),}
\begin{align*}
		u\in {C^{{ \beta}}}((0, T_{\max}), \hatt C^{\nu}(\overline{\Gamma}))\cap { C^{\beta}}((0,  T_{\max}), \cX^{\beta}_r)\text{\ for arbitrary\ } \, r\ge 1,\,\,\, 0<\nu<\beta<\frac{1}{8},
	\end{align*}	
and
\begin{align*}
	&v\in	{C^{1+\beta}} ((0,T_{\max}), \cX^{\beta}_r)\cap {C^{\beta}}((0, T_{\max}), \hatt C^{2+\nu}(\overline{\Gamma}))\quad \text{when}\,\,  (u,v)\,\, \text{ is the solution of \eqref{tauone}},
\end{align*}	
\begin{align*}
	&v\in  {C^{\beta}}((0, T_{\max}), \hatt C^{2+\nu}(\overline{\Gamma}))\quad \text{when},\,  (u,v)\,\, \text{ is the solution of  \eqref{duh}}.
\end{align*}	
 This range of values of parameter $\beta$ does not yield the embedding $\cX^{\beta}_p\hookrightarrow \hatt C^1(\Gamma)$, which is an obstacle to showing that the obtained mild solutions are actually classical solutions of \eqref{parabolic-parabolic-eq} subject to \eqref{contcond},\eqref{fluxcond}.

Our third, and final, objective is to show higher { time and space}  regularity of the mild  solution { and then prove Theorem \ref{Lp}.}  To that end, we formally rewrite \eqref{parabolic-parabolic-eq} as follows
\begin{align}
	&u_t=(\Delta -\sigma)u-a(t,x)\partial_xu+b(t,x)u,\ \ \  \label{lintauone} \\
	&\tau v_t=\partial_{xx}^2 v+f_3(u,v),\label{lintauone2}
\end{align}
where
\begin{align}
	\begin{split}\lb{abcoef}
		&a(t,\cdot)=\partial_uf_1(u(t),v(t))\partial_x v(t),\\
		&b(t,\cdot)=\left(\frac{f_2(u,v)+\sigma u- f_1(u,v)\partial_{xx}v-\partial_vf_1(u,v)(\partial_{x}v)^2}{u}\right).
	\end{split}
\end{align}
Applying Theorem \ref{evolution-operator-thm}, 
we  first show the existence and uniqueness of classical solutions to the above linear system. We then show that the obtained classical solution to the linear system coincides with the mild solution 
of    \eqref{parabolic-parabolic-eq} subject to \eqref{contcond},\eqref{fluxcond} from the previous step. To carry out this item we use explicit estimate on the norm of the evolution operator from Theorem \ref{evolution-operator-thm} and a version of Gronwall's inequality recalled in Proposition \ref{propHenry}.  { This line of arguments leads to $(u(t),v(t))\in \cX_r^\alpha$ for arbitrary $t\in (0,T_{\max})$, $r\ge 1$ and $\alpha\in (0, 1)$. This regularity is enough for us to show  that the obtained mild solutions are actually classical solutions of \eqref{parabolic-parabolic-eq} subject to \eqref{contcond}, \eqref{fluxcond}.}

\subsection{Rigorous formulation of integral equations}\label{sec:inteq}

In Propositions \ref{intconv}, \ref{prop1.3}, \ref{prop1.4} below we prove that the integral in \eqref{phi} converges absolutely in $\elpee$, the integral in  \eqref{psi} converges in $\wpee$ 
both to elements of $\cX^{\beta}_p(\Gamma)$. We also examine regularity properties of the mappings $t\mapsto \Phi(t), \Psi(t), \Theta(t)$ in $L^p(\Gamma)$ and in $\cX_p^{\beta}$ for a suitable range of parameters $\beta, p$.
\begin{hypothesis}\label{hypnew3.1}
Assume Hypothesis \ref{hyp}. Let us fix $
	T>0$, $u\in \elpeet{T}$, $v\in\wpeet{T}$, $t\in(0,T)$, $s\in(0,t). $
\end{hypothesis}

\begin{proposition}\lb{intconv}  Assume Hypothesis \ref{hypnew3.1}. Then one has 
	\begin{equation}\lb{f1dlpprop}
	f_1(u(s), v(s))\partial_xv(s)\in L^{\frac{p}{2\gamma}}(\Gamma).
	\end{equation}
	 The operator $e^{(\Delta-\sigma)(t-s)}\partial_x$ originally defined on $\hatt C_0^{\infty}(\Gamma)$ can be extended to a bounded linear operator in $\cB(L^{\frac{p}{2\gamma}}(\Gamma), \elpee)$. For such an extension one has
		\begin{align}
			&F_1(s):= e^{(\Delta-\sigma)(t-s)}\partial_x\big(f_1(u(s), v(s))\partial_xv(s) \big)\in\elpee,\lb{f1dlpprop2}
		\end{align}
and, in addition, we have
	\begin{align}
			&F_2(s):=e^{(\Delta-{\sigma}I)(t-s)}(f_2(u(s),v(s))+\sigma u(s))\in \wpee,\lb{f1dlpprop4}\\
			&F_3(s):=e^{(\Delta-{ \sigma }I)(t-s)}{ \big(f_3(u(s),v(s))+\sigma u(s)\big)}\in \wpee,\lb{f1dlpprop3}
			\end{align}
	and the functions $(0,t)\ni s\mapsto F_1(s)\in\elpee$, $(0,t)\ni s\mapsto  F_2(s)\in \wpee$, $(0,t)\ni s\mapsto  F_3(s)\in \wpee$ are integrable on $(0,t)$.

\end{proposition}
\begin{proof}
Inclusions \eqref{f1dlpprop}, \eqref{f1dlpprop2} follow from Proposition \ref{prop3.1new} with $F=f_1$ and 
	\begin{equation}\lb{rest}
	R:=\max_{s\in[0,T]} \|(u(s), v(s))\|_{L^{p}(\Gamma)\times\wpee}<\infty.
\end{equation} 
Next, 
{ by \eqref{ehu} and \eqref{3.23}} one has $F_1\in C\big((0,t), \elpee \big)$. To prove that $F_1$ is integrable we write
	\begin{align*}
		&\int_0^t \| F_1(s) \|_{L^p(\Gamma)}ds  \leq \int_0^t \big\|e^{(\Delta-\sigma)(t-s)}\partial_x \big(f_1(u(s), v(s))\partial_xv(s) \big)\big\|_{L^p(\Gamma)}  ds\nonumber\\
		& \leq C \int_0^t (t-s)^{-\frac{p+2\gamma-1}{2p}}\|(u(s), v(s))\|_{L^{p}(\Gamma)\times\wpee} ds\nonumber\\
&  \leq  C  \int_0^t (t-s)^{-\frac{p+2\gamma-1}{2p}}ds<\infty,
	\end{align*}
	where we used \eqref{3.22}. 

Letting  ${ F(u,v)=f_2(u,v)+\sigma u}$ in Proposition \ref{cor3.1}  and $R$ be as in \eqref{rest}, then \eqref{efbdd} yields \eqref{f1dlpprop4},  and { \eqref{eflip} and  \eqref{eminusone}}  yield the continuity of the mapping
$
(0, t)\ni s\mapsto F_2(s)\in \wpee.
$
To prove { the integrability of }  $F_2$ we employ { \eqref{efbdd} }
to obtain
\begin{equation*}
\left\|F_2(s)\right\|_{\wpee}\leq C\max\{(t-s)^{-\frac{p+\gamma-1}{2p}}, (t-s)^{-\frac{\gamma-1}{2p}}\} \|(u(s), v(s))\|_{L^{p}(\Gamma)\times\wpee}
\end{equation*}
and then use the dominated convergence theorem together with \eqref{rest} to get
\begin{equation*}
\int_0^t\max\{(t-s)^{-\frac{p+\gamma-1}{2p}}, -(t-s)^{\frac{\gamma-1}{2p}}\}ds<\infty.
\end{equation*}
A similar argument yields  {\eqref{f1dlpprop3}}. 
\end{proof}

\begin{proposition}\lb{prop1.3} Assume Hypothesis \ref{hypnew3.1}. Let $u_0\in L^p(\Gamma)$ and
$v_0\in W^{1,p}(\Gamma)$.  Then for $\beta\in(0,1/4), t>0$ one has
\begin{align}
&(\Phi(u,v;u_0))(t)\in \cX^{\beta}_p,\lb{xbeta}\\
&(\Psi(u,v;v_0))(t)\in  \cX^{{\frac{1}{2}+\beta}}_p,\lb{xbeta2}\\
&(\Theta(u;u_0))(t)\in \cX^{\beta}_p.\lb{xbetanew}
\end{align}	
\end{proposition}
\begin{proof}
We first prove \eqref{xbeta}.
Let us recall $F_1, F_2$ from \eqref{f1dlpprop}, \eqref{f1dlpprop4} and notice that
	\begin{align}
\label{phithreeterms}
			&(\Phi(u,v;u_0)(t))=e^{(\Delta-\sigma)t}u_0-\int_{0}^{t}  F_1(s)ds+ \int_{0}^{t}  F_2(s)ds.
	\end{align}
Since $e^{(\Delta-\sigma)t}$ is an analytic semigroup in $\elpee$, we have 
\begin{equation}
\label{u-0-term-eq}
{e^{(\Delta-\sigma)t}u_0\in \cX^{\alpha}_p,\  \alpha>0.}
\end{equation}
 By Proposition \ref{prop3.1new}, specifically \eqref{3.22}, with $F:=f_1$ , we have $F_1(s)\in\cX^{\beta}_p$ and
	\begin{align*}
			&\int_0^t \| F_1(s) \|_{\cX^{\beta}_p}ds  \leq \int_0^t \left\|e^{(\Delta-\sigma)(t-s)}\partial_x \big(f_1(u(s),v(s))\partial_x v(s)\big)\right\|_{\cX_p^{\beta}}  ds\nonumber\\
			& \leq 	C\int_0^t (t-s)^{^{-\beta-\frac{p+2\gamma-1}{2p}} }\|(u(s), v(s))\|_{L^{p}(\Gamma)\times\wpee}ds<\infty,	
	\end{align*}
where  the last inequality holds due to  $\beta+(p+2\gamma-1)/2p<1$. Hence, the second term in  \eqref{phithreeterms} belongs to $\cX^{\beta}_p$. 
 By Proposition \ref{cor3.1}, specifically \eqref{3.19}, with $F(u,v)=f_2(u,v)+\sigma u$, we have $F_2(s)\in\cX^{{\frac{1}{2}+\beta}}_p$ and 
	\begin{align}
\label{f-2-term-eq}
		&\int_{0}^{t} \|F_2(s)\|_{\cX^{{\frac{1}{2}+\beta}}_p} ds  \leq C\int_{0}^{t} (t-s)^{{ -\frac{1}{2}}-\beta-\frac{\gamma-1}{2p}} (1+\|(u(s), v(s))\|_{L^{p}(\Gamma)\times\wpee}) ds<\infty,
	\end{align}
	where we used $\frac{1}{2}+\beta+(\gamma-1)/(2p)<1$.  Hence, the third term in  \eqref{phithreeterms} belongs to $\cX^{{\frac{1}{2}+\beta}}_p$.   { It then follows that \eqref{xbeta} holds}.

{Next,  note that
\begin{align*}
			&(\Psi(u,v;v_0)(t))=e^{(\Delta-\sigma)t}v_0+ \int_{0}^{t}  F_3(s)ds.
	\end{align*}
Then \eqref{xbeta2} follows from the arguments of \eqref{u-0-term-eq} and \eqref{f-2-term-eq}.}

 Finally,  \eqref{xbetanew} follows from \eqref{geq}, \eqref{xbeta}, and Corollary \ref{cor3.2}.
\end{proof}

\begin{proposition}\lb{prop1.4}  Assume Hypothesis \ref{hypnew3.1}. 
Let $\beta\in(0,1/8)$, $u_0\in L^p(\Gamma)$, and $v_0\in W^{1,p}(\Gamma)$. Then one has
{
	\begin{align}
\label{Phi-Theta-regularity-eq}
\Phi(u,v;u_0),  \Theta(u;u_0)\in C\left([0,T], \elpee\right) \cap  {  C^{\beta}}\left((0, T], \cX_p^{\beta}\right),
	\end{align}
and  if $v_0\in \cX_p^{\frac{1}{2}+\beta}$, then
\begin{align}
\label{Psi-regularity-eq}
\Psi(u,v;v_0)\in    C\left([0,T], \hatt W^{1,p}(\Gamma)\right)\cap {C^{\beta}}\left((0, T], \cX_p^{\frac{1}{2}+\beta}\right).
	\end{align}
}
\end{proposition}

\begin{proof}
 We first prove the H\"older continuity.
Recalling   \eqref{phithreeterms} and \eqref{f1dlpprop2}, \eqref{f1dlpprop4}, \eqref{f1dlpprop3}, let us denote
\begin{align}
\begin{split}\lb{theis}
&(I_0(u_0))(t):= e^{(\Delta-\sigma)t}u_0, \ I_1(t)=\int_0^t e^{(\Delta-\sigma)(t-s)}\partial_x(f_1(u(s), v(s))\partial_x v(s))ds,\\
&I_k(t)=\int_0^t e^{(\Delta-\sigma)(t-s)}(f_k(u(s), v(s))+\sigma u(s)) ds,\ k=2,3,
\end{split}	
\end{align}
so that  
$$(\Phi(u,v;u_0))(t)=(I_0(u_0))(t){-}I_1(t)+I_2(t), \quad (\Psi(u,v;v_0))(t)=(I_0(v_0))(t)+I_3(t).
$$

We will show H\"older continuity of each $I_k$ separately. Throughout the rest of this proof we fix  $t\in(0,T)$ and $h\in (0,T-t)$.

For $I_0$ term one has
\begin{align}
\label{I-0-cont-eq}
	\|I_0(u_0)(t+h)-I_0(u_0)(t)\|_{\cX^{{ \alpha}}_p}&=\|(e^{(\Delta-\sigma)h}-I)(\sigma-\Delta)^{{\alpha}} e^{(\Delta-\sigma)t}u_0\|_{L^p(\Gamma)}\nonumber\\
	&\underset{\eqref{fracpow2}}{\leq} C h^{{\alpha}}\|(\sigma-\Delta)^{{\alpha}}e^{(\Delta-\sigma)t} u_0\|_{\cX^{\alpha}_p}\nonumber\\
&\leq C h^{\alpha}t^{-2\alpha}\|u_0\|_{\elpee},\  \alpha\in (0,1),
\end{align}
which implies H\"older continuity of $I_0$ { in $\cX^\alpha$  with exponent $\alpha$ for any $\alpha\in (0,1)$}. 

To treat $I_1$ we use Proposition \ref{prop3.1new} with $F=f_1$. First, let us notice that
\begin{align*}
	I_1(t+h)-I_1(t)= \int_0^t (e^{(\Delta-I)h}-I)&e^{(\Delta-I)(t-s)}\partial_x \cH(U(s))ds \nonumber\\
	&{+} \int_t^{t+h} e^{(\Delta-I)(t-s+h)}\partial_x \cH(U(s))ds
\end{align*}
where $\cH$ is as in Proposition \ref{prop3.1new}. 
Then by \eqref{3.23} one has
\begin{align*}
&\int_0^t \|(e^{(\Delta-\sigma)h}-I)e^{(\Delta-\sigma)(t-s)}\partial_x \cH(U(s))\|_{\cX^{\beta}} ds\leq \nonumber\\
&\leq C\int_0^t h^{\beta}(t-s)^{^{-2\beta-\frac{p+2\gamma-1}{2p}} }\|(u(s), v(s))\|_{L^{p}(\Gamma)\times\wpee}  ds\leq Ch^{\beta}, 
\end{align*}
and by \eqref{3.22}
\begin{align*}
	& \int_t^{t+h} \|e^{(\Delta-\sigma)(t-s+h)}\partial_x \cH(U(s))\|_{\cX^{\beta}}ds\nonumber\\
	&\leq  C\int_t^{t+h} (t-s+h)^{^{-\beta-\frac{p+2\gamma-1}{2p}} }\|(u(s), v(s))\|_{L^{p}(\Gamma)\times\wpee}ds< Ch^{\frac12-\beta}.
\end{align*}
{ Therefore  
\begin{equation}
\label{I-1-cont-eq}
\|I_1(t+h)-I_1(t)\|_{\cX^\beta}\le C h^\beta,\  0<h\ll 1.
\end{equation}
}

To treat $I_2$ we note
\begin{align*}
		I_2(t+h)-I_2(t)
		=& \int_0^t (e^{(\Delta-\sigma)h}-I)e^{(\Delta-I)(t-s)}(f_2(u(s),v(s))+\sigma u)ds\\
		&{ +} \int_t^{t+h} e^{(\Delta-\sigma)(t-s+h)}(f_2(u(s),v(s))+\sigma u(s))ds.
\end{align*}
Then Proposition \ref{cor3.1}, concretely  \eqref{3.19}, with $F(u,v)=f_2(u,v)+\sigma u$  yields
\begin{align*}
	&\int_0^t \|(e^{(\Delta-\sigma)h}-I)e^{(\Delta-I)(t-s)}(f_2(u(s),v(s))+\sigma u)\|_{\cX^{{\frac{1}{2}+\beta}}_p} ds\nonumber\\
&  \underset{(A.11)}\le \int_0^t h^\beta  \|e^{(\Delta-I)(t-s)}(f_2(u(s),v(s))+\sigma u)\|_{\cX^{\frac{1}{2}+2\beta}_p} ds\nonumber\\
&\underset{\eqref{3.19}}\leq C \int_0^t h^{\beta}(t-s)^{-\frac{1}{2}-2\beta-\frac{\gamma-1}{2p}}(1+\|(u(s), v(s))\|_{L^{p}(\Gamma)\times\wpee})ds<C h^{\beta},
\end{align*}
and 
\begin{align*}
	&\int_t^{t+h} \|e^{(\Delta-\sigma)(t-s+h)}(f_2(u(s),v(s))+\sigma u)\|_{\cX^{\frac12+\beta}_p} ds \nonumber\\
	&\underset{\eqref{3.19}}{\leq } C\int_t^{t+h} (t-s+h)^{^{-\beta-\frac12-\frac{\gamma-1}{2p}} }(1+\|(u(s), v(s))\|_{L^{p}(\Gamma)\times\wpee})ds< Ch^{\frac12-\beta}.
\end{align*}
{ Therefore 
\begin{equation}
\label{I-2-cont-eq}
\|I_2(t+h)-I_2(t)\|_{\cX^{\frac{1}{2}+\beta}_p}\le C h^\beta, \, 0<h\ll 1.
\end{equation}
}

{ By the similar arguments of  \eqref{I-2-cont-eq}, we have
\begin{equation}
\label{I-3-cont-eq}
\|I_3(t+h)-I_3(t)\|_{\cX^{\frac{1}{2}+\beta}_p}\le C h^\beta, \, 0<h\ll 1.
\end{equation}
}

{ By \eqref{I-0-cont-eq}, \eqref{I-1-cont-eq}, and \eqref{I-2-cont-eq}, we have
  $\Phi(u,v;u_0)\in { C^{\beta}}\left((0,T], \cX_p^{\beta}\right)$.
By \eqref{I-0-cont-eq} with $u_0$ being replaced by $v_0$ and \eqref{I-3-cont-eq}, we have
$\Psi(u,v;v_0)\in {C^{\beta}}\left((0, T], \cX_p^{\frac{1}{2}+\beta}\right)$.
 By   \eqref{geq} and  Corollary  \ref{cor3.2}, we have  $\Theta(u;u_0)\in {C^{\beta}}\left((0,T], \cX_p^{\beta}\right)$. }

\smallskip

Next, we  prove  $\Phi(u,v;u_0),\Theta(u;u_0)\in C\left([0,T], \elpee\right)$ and  {if $v_0\in \cX_p^{\frac{1}{2}+\beta}$, then}
$\Psi(u,v;v_0)\in   C\big([0,T], \hatt W^{1,p}(\Gamma)\big)$. 
Note that  
$$
\cX^\beta \hookrightarrow L^p(\Gamma),\quad \cX^{\frac{1}{2}+\beta}\hookrightarrow W^{1,p}(\Gamma).
$$
It then suffices to prove that $\Phi(u,v;u_0)(t)$ and $\Theta(u;u_0)(t)$ are continuous at $t=0$ in $L^p(\Gamma)$, and
$\Psi(u,v;v_0)(t)$ is continuous at $t=0$ in $W^{1,p}(\Gamma)$ provided $v_0\in \cX_p^{\frac{1}{2}+\beta}$.

  Note that
\begin{align}
\lb{cont1}
&\|(\Phi(u,v;u_0)(t)-(\Phi(u,v;u_0)(0)\|_{L^p(\Gamma)}=\|(\Phi(u,v;u_0)(t)-u_0\|_{L^p(\Gamma)}\nonumber\\
&\leq \|e^{(\Delta-\sigma)t}u_0-u_0\|_{L^p(\Gamma)} +\int_{0}^{t} \|e^{(\Delta-\sigma)(t-s)}\partial_x(f_1(u(s), v(s))\partial_x v(s))\L_{L^p(\Gamma)}ds \nonumber\\ &\quad + \int_{0}^{t} \|e^{(\Delta-\sigma)(t-s)}(f_2(u(s), v(s))+\sigma u(s))\|_{L^p(\Gamma)}ds.
\end{align}
{ By the continuity of $e^{(\Delta-\sigma)t}$ at $t\ge 0$, we have
\begin{equation}
\label{cont001}
\lim_{t\to 0} \|e^{(\Delta-\sigma)t}u_0-u_0\|_{L^p(\Gamma)}=0.
\end{equation}}
Using Proposition \ref{prop3.1new}, concretely \eqref{3.22}, with $F=f_1$ we obtain
\begin{align}
\lb{cont01}
&\int_{0}^{t} \|e^{(\Delta-\sigma)(t-s)}\partial_x(f_1(u(s), v(s))\partial_x v(s))\|_{L^p(\Gamma)}ds\nonumber\\
&\leq C \int_0^t (t-s)^{-\frac{p+2\gamma-1}{2p}}\|(u(s), v(s))\|_{\elpee\times\wpee}ds\rightarrow 0\,\, {\rm as}\,\, t\rightarrow0.
\end{align}
Using Proposition \ref{cor3.1}, specifically \eqref{efbdd}, with $F(u,v)=f_2(u,v)+\sigma u$ one obtains
\begin{align}
\lb{cont02}
&\int_{0}^{t} \|e^{(\Delta-\sigma)(t-s)}(f_2(u(s),v(s))+\sigma u(s))\|_{\elpee}ds\nonumber\\
&\leq C \int_0^t \max((t-s)^{-\frac{p+\gamma-1}{2p}}, (t-s)^{-\frac{\gamma-1}{2p}})(1+\|(u(s), v(s))\|_{\elpee\times\wpee})ds\nonumber\\
&\quad \rightarrow 0\,\, {\rm as}\,\,  t\rightarrow0.
\end{align}
It then follows that  $\Phi(u,v;u_0)(t)$ is continuous at $t=0$ in  $L^p(\Gamma)$.

 Similarly, it can be proved that $\Theta(u;u_0)(t)$ is continuous at $t=0$ in  $L^p(\Gamma)$.

To prove the continuity of $\Psi(u,v;v_0)(t)$ at $t=0$ in $\wpee$ we first note
\begin{align*}
 &\|(\Psi(u,v;v_0))(t)-(\Psi(u,v;v_0))(0)\|_{W^{1,p}(\Gamma)}\leq \\
 &\|e^{(\Delta-\sigma)t }v_0-v_0\|_{W^{1,p}(\Gamma)}+\left\|\int_0^t e^{(\Delta-\sigma)(t-s)}\Big(f_3(u(s),v(s))+\sigma v(s)\Big)ds\right\|_{W^{1,p}(\Gamma)}.
\end{align*}
{Furthermore, for $v_0\in \cX_p^{\frac{1}{2}+\beta}$ one has
\begin{align}
\label{new-cont01}
\|e^{(\Delta -\sigma)t} v_0-v_0\|_{W^{1,p}(\Gamma)}&\le C\| (\sigma-\Delta)^{\frac{1}{2}+\beta}\big( e^{(\Delta -\sigma)t}v_0-v_0\big)\|_{L^p(\Gamma)}\nonumber\\
&= C\| e^{(\Delta-\sigma)t} (\sigma-\Delta)^{\frac{1}{2}+\beta}v_0-(\sigma-\Delta)^{\frac{1}{2}+\beta}v_0\|_{L^p(\Gamma)}\nonumber\\
&  \to 0\quad {\rm as }\quad t\to 0,
\end{align} }
{where in the first inequality we used the embedding $\cX_p^{\beta+\frac{1}{2}}\hookrightarrow \wpee$.}
{ By \eqref{efbdd} with $F(u,v)=f_3(u,v)+\sigma v$, we have
\begin{align}
\label{new-cont02}
&\int_{0}^{t} \|e^{(\Delta-\sigma)(t-s)}(f_3(u(s),v(s))+\sigma v(s))\|_{W^{1,p}(\Gamma)}ds\nonumber\\
&\leq C \int_0^t \max((t-s)^{-\frac{p+\gamma-1}{2p}}, (t-s)^{-\frac{\gamma-1}{2p}})(1+\|(u(s), v(s))\|_{\elpee\times\wpee})ds\nonumber\\
&\quad \rightarrow 0\quad {\rm as}\quad t\rightarrow0.
\end{align}
Therefore, 
$\Psi(u,v;v_0)(t)$ is continuous at $t=0$ in $W^{1,p}(\Gamma)$ { provided $v_0\in \cX_p^{\frac{1}{2}+\beta}$.}  This completes the proof of the proposition.}
\end{proof}

\subsection{The fixed point argument and an extension argument}\label{Sec:mildsol}

In this subsection, we show existence of mild solutions of \eqref{tauone} and \eqref{duh}. 

For fixed  $T>0$, $R>0$ let us introduce
\begin{equation}\lb{btr}
	B_k(T,R):=\left\{ u\in C\big([0,T], \hatt W^{k,p}(\Gamma) \big) : \|u\|_{L^{\infty}_t W^{k,p}_x}:=\max_{t\in[0,T]}\|u(t)\|_{\hatt W^{k,p}(\Gamma)}\leq R \right\}, k=0,1.
\end{equation}
Below, for example in Proposition  \ref{mild-solution-prop1}, we consider $B_0(T, R)\times B_1(T, R)$ as a complete metric space with respect to the metric
\begin{equation}
\rho((u_1, v_1), (u_2,v_2))=\|u_1-u_2\|_{L^{\infty}_t L^p_x}+\|v_1-v_2\|_{L^{\infty}_t W^{1,p}_x},
\end{equation}
$(u_k,v_k)\in B_0(T, R)\times B_1(T, R)$. 

Let us fix $(u_0, v_0)\in\elpee\times\wpee$, recall  \eqref{lambda}, \eqref{theta}, and note that by Proposition \ref{prop1.4}, { if $v_0\in\cX_p^{\frac{1}{2}+\beta}$},  one has
\begin{align}
&\Lambda(u_0,v_0): C([0, T], L^p(\Gamma)\times\wpee)\rightarrow C([0, T], L^p(\Gamma)\times\wpee),\lb{4.39}\\
&\Theta(u_0): C([0, T], L^p(\Gamma))\rightarrow C([0, T], L^p(\Gamma)).\lb{4.40}
\end{align}

\begin{proposition}\lb{mild-solution-prop1}
Assume Hypothesis \ref{hyp}.  Then for arbitrary $R{ >} \|(u_0, v_0)\|_{\elpee\times\wpee}$ there exists $T(R)>0$ such that the following assertions hold: 
\begin{enumerate}
\item\label{item1lem2.1}  { If {$v_0\in \cX_p^{\frac{1}{2}+\epsilon}$  for some $0<\epsilon\ll 1$}, then } $\Lambda(u_0,v_0)$ is a contraction mapping on $B_0(T(R), R)\times B_1(T(R), R)$,
\item\label{item2lem2.1}  $\Theta(u_0)$ is a contraction mapping on $B_0(T(R), R)$.
\end{enumerate}
\end{proposition}

\begin{proof} 
	\eqref{item1lem2.1}. { First we note that  $\cX_p^{\frac{1}{2}+\epsilon}\hookrightarrow\wpee$, hence, the mapping properties \eqref{4.39}, \eqref{4.40} of $\Lambda, \Theta$ hold as asserted. Next,} we show that $\Lambda(u_0,v_0)$ maps $B_0(T(R), R)\times B_1(T(R), R)$ into itself.
For $(u,v)\in B_0(T(R), R)\times B_1(T(R), R)$ we have
\begin{align*}
\|(\Phi (u,v;u_0))(t)\|_{L^p(\Gamma)}& \leq \|e^{(\Delta-\sigma)t}u_0\|_{L^p(\Gamma)}+ \int_{0}^{t} \|e^{(\Delta-\sigma)(t-s)}\partial_x(f_1(u(s), v(s))\partial_x v(s))ds \nonumber\\ &\quad + \int_{0}^{t} \|e^{(\Delta-\sigma)(t-s)}(f_2(u(s), v(s))+\sigma u(s))\|_{L^p(\Gamma)}ds
\end{align*}
{ and
\begin{align*}
\|(\Psi(u,v;v_0))(t)\|_{\wpee}&\le \|e^{(\Delta-\sigma I)t}v_0\|_{\wpee}\\
&\quad + \int_{0}^{t}  \|e^{(\Delta-\sigma)(t-s)}\left(f_3(u(s), v(s))+\sigma v(s)\right)\|_{\wpee} ds.
\end{align*}}
Then by  \eqref{cont001}, \eqref{cont01}, and \eqref{cont02}, for sufficiently small $T(R) $ one has 
\begin{equation}
\|(\Phi (u,v;u_0))(t)\|_{\elpee}\leq R,\ t\in [0, T(R)].
\end{equation}
{ Similarly,  by \eqref{new-cont01} and \eqref{new-cont02}, 
$$
\|(\Psi(u,v;v_0))(t)\|_{\hatt W^{1,p}(\Gamma)}\le R\quad\,\, t\in [0,T(R)]
$$
for $T(R)$ sufficiently small.  Therefore, $\Lambda(u_0,v_0)$ maps $B_0(T, R)\times B_1(T, R)$ into itself for 
$0<T\ll1 $.}

To show that  $\Lambda(u_0,v_0): B_0(T(R), R)\times B_1(T, R) \to  B_0(T,R)\times B_1(T,R)$   is a contraction mapping 
for $0<T\ll1 $, let us first observe that
\begin{align}
 \label{phipsidif}
		& \|(\Phi(u_1,v_1;u_0))(t) - (\Phi(u_2, v_2;u_0))(t)\|_{L^p(\Gamma)} +\|(\Psi(u_1,v_1;v_0))(t) - (\Psi(u_2, v_2;v_0))(t)\|_{\wpee} \nonumber\\
		& \leq \int_{0}^{t} \|e^{(\Delta-\sigma)(t-s)}\partial_x(f_1(u_1(s), v_1(s))\partial_x v_1(s)) - e^{(\Delta-\sigma)(t-s)}\partial_x(f_1(u_2(s), v_2(s))\partial_x v_2(s))\|_{L^p(\Gamma)}ds \nonumber\\ 
		&+ \int_{0}^{t} \|e^{(\Delta-\sigma)(t-s)}(f_2(u_1(s), v_1(s))+\sigma u_1(s))-e^{(\Delta-\sigma)(t-s)}(f_2(u_2(s), v_2(s))+\sigma u_2(s))\|_{L^p(\Gamma)}ds\nonumber\\
		&+\int_{0}^{t} \|e^{(\Delta-\sigma)(t-s)}({f_3}(u_1(s), v_1(s))+\sigma v_1(s))-e^{(\Delta-\sigma)(t-s)}({f_3}(u_2(s), v_2(s))+\sigma v_2(s))\|_{\wpee}ds\nonumber\\
		&=I+II+III.
\end{align}
Using Proposition \ref{prop3.1new}, concretely \eqref{lipcontedh}, with $F=f_1$  one obtains
\begin{equation}\label{onephipsi}
I\leq C(R)\int_0^t(t-s)^{-\frac{p+2\gamma-1}{2p}}\|(u_1(s), v_1(s))-(u_2(s), v_2(s))\|_{{\elpee\times\wpee}}ds.
\end{equation}
Using Proposition \ref{cor3.1}, concretely \eqref{efdlip}, with $F(u,v)=f_2(u,v)+\sigma u$ one obtains
\begin{equation}\label{twophipsi}
II\leq C(R)\int_0^t(t-s)^{-\frac{p+\gamma-1}{2p}}\|(u_1(s), v_1(s))-(u_2(s), v_2(s))\|_{{\elpee\times\wpee}}ds.
\end{equation}
Using Proposition \ref{cor3.1}, concretely \eqref{efdlip} again, with $F(u,v)=f_3(u,v)+\sigma v$ one obtains
\begin{equation}\label{threephipsi}
III\leq C(R)\int_0^t(t-s)^{-\frac{p+\gamma-1}{2p}}\|(u_1(s), v_1(s))-(u_2(s), v_2(s))\|_{{\elpee\times\wpee}}ds.
\end{equation}
Combining \eqref{phipsidif}, \eqref{onephipsi}, \eqref{twophipsi}, \eqref{threephipsi} and choosing $T(R)$ sufficiently small we obtain 
\begin{align}
&\|(\Phi(u_1,v_1;u_0))(t) - (\Phi(u_2, v_2;u_0))(t)\|_{L^p(\Gamma)} +\|(\Psi(u_1,v_1;v_0))(t) - (\Psi(u_2, v_2;v_0))(t)\|_{\wpee} \\
&\quad\leq \frac 12 \|(u_1(t), v_1(t))-(u_2(t), v_2(t))\|_{{\elpee\times\wpee}},\  t\in [0, T(R)]\lb{4.46}
\end{align}
as asserted. 

Part \eqref{item2lem2.1} follows from part \eqref{item1lem2.1}, \eqref{geq}, and Corollary \ref{cor3.2}. 
\end{proof}

{\begin{remark}
Recall \eqref{lambda} and note that $t\mapsto \Psi(u, v, v_0)(t)$ is considered as a mapping into $\wpee$. In view of these facts, it is natural to conjecture that the condition $v_0\in \cX_p^{\frac{1}{2}+\epsilon}$ in Theorem \ref{mild-solution-prop1} (1) should be relaxed to $v_0\in\wpee$. We claim, however, that the latter inclusion is not sufficient for $\Psi(u, v, v_0)(t)$ to be continuous in $\wpee$ at $t=0$. Indeed, $\Psi(u, v, v_0)(t)$, $t>0$ satisfies Neumann--Kirchhoff vertex conditions. If $t\mapsto\Psi(u, v, v_0)(t)$ was continuous in $\wpee$ at $t=0$ then $v_0$ would have to satisfy the same vertex conditions (since Dirichlet and Neumann traces are continuous with respect to $\wpee$ norm). Picking $v_0\in \wpee$ that does not satisfy the Neumann--Kirchhoff vertex conditions yields a contradiction. 

That said, our next result shows that the requirement $v_0\in \cX_p^{\frac{1}{2}+\epsilon}$ is not necessary for existence of mild solutions of \eqref{tauone}, \eqref{duh} that are continuous in $\wpee$ for strictly positive time $t\in(0, T)$ but not $t=0$, in which case the continuity is claimed only with respect $\el{p}$ norm. 
\end{remark}
}
\begin{proposition}\lb{mild-solution-prop2}
Assume Hypothesis \ref{hyp}. Let $\beta\in(0, 1/8)$ and suppose that the initial data satisfies   $\|(u_0, v_0)\|_{\elpee\times\wpee}< R$. 
\begin{itemize}
\item[(1)]  There exists $T>0$ such that integral equation \eqref{tauone} has a unique solution $(u(t),v(t))=(u(t;u_0,v_0),v(t;u_0,v_0))$ satisfying
\begin{align}
\begin{split}\lb{milduvsols}
&(u,v) \in  { C([0, T], L^p(\Gamma)\times L^p(\Gamma))\cap  C^{\beta}((0, T], \cX^{\beta}_p \times \cX_p^{\frac{1}{2}+\beta})},\\
&\sup_{t\in(0,T]}\|u(t), v(t)\|_{{L^p(\Gamma)}\times \hatt W^{1,p}(\Gamma)}\leq R.
\end{split}
\end{align}

\item[(2)] There is $T>0$ such that \eqref{duh} has a unique solution $(u(t),v(t))=(u(t;u_0),v(t;u_0))$ satisfying 
\begin{align}
\begin{split}\lb{mildusol}
&u\in C([0, T], L^p(\Gamma))\cap C^{\beta}((0, T], \cX^{\beta}_p),\\
&\hspace{1.5cm}\max_{t\in[0,T]}\|u(t)\|_{L^p(\Gamma)}\leq R. 
\end{split}
\end{align}
\end{itemize}
\end{proposition}

\begin{proof}
(1) First, note that  $\cX_p^{\frac{1}{2}+\beta}$  is dense in $\hatt W^{1,p}(\Gamma)$. Hence there is a sequence $\{v_n\}\subset \cX_p^{\frac{1}{2}+\beta}$ such that
$$
\|v_n- v_0\|_{\hatt W^{1,p}(\Gamma)}\to 0\quad {\rm as}\quad n\to\infty.
$$
Then there is $N\in\mathbb{N}$ such that
 $$\|(u_0, v_n)\|_{\elpee\times\wpee}< R,\,\, n\ge N.
$$
By Proposition \ref{mild-solution-prop1},  there is $T=T(R)>0$ such that
$\{\Lambda(u_0,v_n)\}_{n\ge N}$  is a family of uniform contractions on  $B_0(T, R)\times B_1(T, R)$.
Then by fixed point theorem,  the integral equation \eqref{tauone} with $v_0$ being replaced by $v_n$ ($n\ge N$)  has a unique solution $(u_n(t),v_n(t))=(u(t;u_0,v_n),v(t;u_0,v_n))$ satisfying 
\begin{align*}
&(u_n,v_n)\in C([0, T], L^p(\Gamma)\times\wpee)\cap { C^{\beta}}((0, T], \cX^{\beta}_p\times \cX^{{\frac{1}{2}+\beta}}_p),\\
&\hspace{1.5cm}\max_{t\in[0,T]}\|u_n(t), v_n(t)\|_{{L^p(\Gamma)}\times \hatt W^{1,p}(\Gamma)}\leq R
\end{align*}
 for $n\ge N$.


Next, let us note that\footnote{the first embedding holds for $\beta>\frac{1}{2p}$. Since $p\geq 8$, the first embedding holds for $\beta\in(\frac{1}{16}, \frac18)$. Throughout the rest of the proof we assume this range for $\beta$ without loss of generality.} 
{
$$
\cX^{\beta}_p\hookrightarrow \hatt C({\overline\Gamma})\hookrightarrow L^p(\Gamma)\quad {\rm and}\quad \cX^{\frac{1}{2}+\beta }_p\hookrightarrow \hatt C^1({\overline\Gamma})  \hookrightarrow W^{1,p}(\Gamma).
$$

By Arzela–Ascoli Theorem there exists $(u^*,v^*)\in C((0,T],L^p(\Gamma)\times \wpee)$} and a subsequence
$(u_{n_k}(t),v_{n_k}(t))$ such that
$$
\lim_{n_k\to\infty} (u_{n_k}(t),v_{n_k}(t))=(u^*(t),v^*(t)) \quad {\rm in}\quad L^p(\Gamma)\times W^{1,p}(\Gamma),
$$
locally uniformly in $(0,T]$, and
$$
\sup_{t\in (0,T]}\|(u^*(t),v^*(t))\|_{L^p(\Gamma)\times W^{1,p}(\Gamma)}\le R.
$$
 Note also that
$$
\lim_{n\to\infty} e^{(\Delta -\sigma)t}v_n=e^{(\Delta-\sigma)t}v_0 \quad {\rm in}\quad L^p(\Gamma)
$$
uniformly in $t\in [0,T]$.
It then follows that $(u^*(t),v^*(t))$ satisfies 
\begin{align*}
	u^*(t)&=e^{(\Delta-\sigma)t}u_0-\int_{0}^{t}  e^{(\Delta-\sigma)(t-s)}\partial_{x}\big(f_1(u^*(s), v^*(s))\partial_xv^*(s)\big)ds\\
		&\quad + \int_{0}^{t}  e^{(\Delta-\sigma)(t-s)}(f_2(u^*(s), v^*(s))+\sigma u^*(s))ds,\\
	v^*(t)& =e^{(\Delta-\sigma I)t}v_0+ \int_{0}^{t}  e^{(\Delta-\sigma)(t-s)}\left(f_3(u^*(s), v^*(s))+\sigma v^*(s)\right)ds
\end{align*}
for $t\in (0,T]$.  Therefore, 
 $(u(t;u_0,v_0),v(t;u_0,v_0)):=(u^*(t),v^*(t))$ is a solution of the integral equation \eqref{tauone}
satisfying \eqref{milduvsols}.  {The uniqueness of such solutions follows readily from the estimates similar to \eqref{phipsidif}, \eqref{4.46}. }

(2)  The fixed point theorem for contraction mapping $\Theta(u_0): B_0(T, R)\rightarrow B_0(T, R)$ for $0<T\ll 1$, as in Proposition  \ref{mild-solution-prop1}, yields \eqref{mildusol}. 
\end{proof}

\begin{theorem}\lb{thm3.1} Assume Hypothesis \ref{hyp} and let $\beta\in(0, 1/8)$.
\begin{enumerate}
\item Let $u_0\in \elpee$. Then there exists $T_{\max}=T_{\max}(u_0)\in (0,\infty]$ such that integral equation \eqref{duh} has a unique  solution  $(u(t),v(t))=(u(t;u_0),v(t;u_0))$ satisfying 
\begin{align}\lb{ubp}
u(\cdot)\in   C([0,  T_{\max}), \elpee)\cap C^{\beta}((0,  T_{\max}), \cX^{\beta}_p),\ u(0;u_0)=u_0.
\end{align}	
If $T_{\max}<\infty$ then $\limsup\limits_{t\rightarrow T_{\max}^-}\|u(t)\|_{\elpee}=\infty$.   Moreover, for arbitrary $r\geq 1$ and $\nu <\beta$, {
\begin{align}\lb{rubp-1}
		u\in C^{{ \beta}}((0, T_{\max}), \hatt C^{\nu}(\overline{\Gamma}))\cap C^{\beta}((0,  T_{\max}), \cX^{\beta}_r),
	\end{align}	
(hence \eqref{rubp} holds), and
\begin{align}
\label{rvbp-1}
	&v\in	C^{\beta}((0, T_{\max}), \hatt C^{2+\nu}(\overline{\Gamma})).
\end{align}	
}

\item Let $(u_0, v_0)\in\elpee\times\wpee$. Then there exists $T_{\max}=T_{\max}(u_0,v_0)\in (0,\infty]$ such that the integral equation \eqref{tauone} has a unique  solution  $$(u(t),v(t))=(u(t;u_0,v_0),v(t;u_0,v_0))$$ satisfying 
\begin{align}
\begin{split}\label{milduvsolsmax}
&(u,v) \in C([0, T_{\max}), L^p(\Gamma))   \times L^p(\Gamma) )\cap C^{0,\beta}((0, T_{\max}), \cX^{\beta}_p\times \cX^{{\frac{1}{2}+\beta}}_p),
\end{split}
\end{align}	
with $u(0;u_0)=u_0, v(0;v_0)=v_0$. If  $T_{\max}<\infty$  then 
\begin{equation}
\limsup\limits_{t\rightarrow T_{\max}^-}\|(u(t),v(t))\|_{L^p(\Gamma)\times\wpee}=\infty.
\end{equation} 
 Furthermore, 
 for $r\geq 1$, $\nu< \beta$, $u, v$ satisfy {
	\begin{align}
	\label{rmilduvsolsmax-1}
		u\in C^{\beta}((0, T_{\max}), \cX^{\beta}_r)\cap C^{\beta}((0, T_{\max}), \hatt C^{\nu}(\overline{\Gamma})),
\end{align}
and
\begin{align}
\label{rmilduvsolsmax-2}
	&v\in	C^{1+\beta} ((0,T_{\max}), \cX^{\beta}_r)\cap C^{\beta}((0, T_{\max}), \hatt C^{2+\nu}(\overline{\Gamma}))
\end{align}	
(hence 
\eqref{rmilduvsolsmax-u} and \eqref{rmilduvsolsmax-v} hold).}
\end{enumerate}	

\end{theorem}

\begin{proof}
 (1)  First of all,  the existence of solution \eqref{ubp} to \eqref{duh} for large $p$  follows from Proposition \ref{mild-solution-prop2}  and standard extension arguments.

Next,  let us prove \eqref{rubp-1} for arbitrary $r\geq 1$, { which together with \eqref{ubp} implies \eqref{rubp}}. Due to compactness of the graph, it suffices to verify the assertion for $r>p$. To that end, combining \eqref{sobolev-embedding-eq3-0} and \eqref{sobolev-embedding-eq5-0} we note that for any $q> p$, $\beta\in (0,\frac{1}{8})$, and  $\theta\in (0,1)$ with $\frac{1}{q}>\frac{1}{p}-\frac14$, one has
	$$
	u(\cdot;u_0)\in C^{\beta}((0,T_{\max}),\hatt L^{q}(\Gamma)).
	$$
Next, for the partition $p=\tilde p_0<\tilde p_1<\tilde p_2<\cdots<\tilde p_N=r$ such that
	$$
	\frac{1}{\tilde q_i}>\frac{1}{\tilde q_{i-1}}-\frac{1}{4},\quad i=1,2,\cdots,N,
	$$
one obtains
	\begin{equation}
		\label{bootstrap-eq3}
		u(\cdot;u_0)\in C^{{ \beta}}((0,T_{\max}),\hatt L^{\tilde q_i}(\Gamma)),\quad i=1,2,\cdots,N.
	\end{equation}
Hence, 	$u(\cdot;u_0)\in C^{\beta}((0,T_{\max}),\hatt L^{r}(\Gamma))$ for arbitrary $r\geq 1$. Then using $u(\varepsilon, u_0)$, $\varepsilon>0$ as an initial condition and employing \eqref{ubp} with $p=r$ we obtain 
\begin{equation}\lb{bootstrap-eq0}
u(\cdot; u_0)\in C^{\beta}((\varepsilon,  T_{\max}(u_0, v_0)), \cX^{\beta}_r),
\end{equation}
as asserted. Finally, \eqref{sobolev-embedding-eq5} and \eqref{bootstrap-eq0} with  $0<\nu<\frac{1}{8}$, $\theta\in (0,1)$, $\beta\in (\nu,\frac{1}{8})$, and $r>1$ such that $\nu<2\beta\theta-r^{-1}$, one obtains $u(\cdot;u_0)\in C^{\beta}((0,T_{\max}),\hatt C^\nu(\overline{\Gamma})).$	\eqref{rubp-1} then follows.

{  Now, we prove \eqref{rvbp-1}. Recall that   $v(t;u_0)= -(\Delta-\sigma )^{-1}f(u(t;u_0))$. By \eqref{rubp-1} and Proposition \ref{prop2.1}, we have
$$
v\in	 C^{\beta}((0, T_{\max}), \hatt W^{2,r}(\Gamma)).
$$
By \eqref{sobolev-embedding-eq2}, 
we have
\begin{equation}
\label{aux-eq11}
\hatt{W}^{2,r}(\Gamma)\hookrightarrow \hatt{C}^{\tilde \nu} (\overline \Gamma)
				\quad\tilde  \nu <2-\frac{1}{r}.
\end{equation}
It then follows that
\begin{equation}
\label{rvbp-1-1}
v\in	 C^{\beta}((0, T_{\max}), \hatt{C}^{\tilde \nu}(\overline \Gamma))\quad \forall\,\tilde\nu <2-\frac{1}{r}.
\end{equation}
Note that
$$
v_{xx}=v-f(u).
$$
This together with \eqref{rubp-1} and \eqref{rvbp-1-1} implies \eqref{rvbp-1}.
}

(2)  First, as in (1),  the existence of solution  \eqref{milduvsolsmax}) to \eqref{tauone} for large $p$  follows from Proposition \ref{mild-solution-prop2}  and standard extension arguments.  Property  \eqref{rmilduvsolsmax-1}  follows from the arguments of \eqref{rubp-1}.

Next, we prove \eqref{rmilduvsolsmax-2}.
To this end,  we first prove
\begin{equation}\lb{vw2p}
{
	v\in C^{\beta} ((0,T_{\max}), \hatt W^{2,p}(\Gamma))\cap C^{1+\beta} ((0,T_{\max}), \cX^{\beta}_p).}
\end{equation}
For arbitrary $[\kappa,\tau]\subset (0, T_{\max}(u_0, v_0))$,  consider an auxiliary non-homogeneous Cauchy problem
\begin{equation}\lb{auxcauchy}
w_t=w_{xx}+f_3(u,v), w(\tau)=v(\tau), t\geq \tau. 
\end{equation}
We claim this equation posses unique strong solution satisfying $w\in C^{1+\beta}([\kappa,\tau], \cX^{\beta}_p)$ and that this solution coincides with $v$. To prove the former claim let us note that $F(s):=f_3(u(s), v(s))$ is a H\"older continuous mapping from $[\kappa,\tau]$ to $\elpee$ due to \eqref{milduvsolsmax} and $\cX^{\beta}_p\hookrightarrow C^{\nu}$ for some $\nu>0$. Hence, by Theorem \ref{strong-solu-thm} item \eqref{thma1item2}, a unique strong solution of \eqref{auxcauchy} exists and satisfies the integral equation
\begin{equation}\lb{vewint}
w(t)=e^{\Delta (t-\tau)}{ w({\tau})}+\int_{\tau}^te^{\Delta(t-s-\tau)}f_3(u(s), v(s))ds, t\in [\tau, \kappa]. 
\end{equation}
Hence, to prove $w\in C^{1+\beta}([\kappa,\tau], \cX^{\beta}_p)$ it suffices to show that the integral term, as a function of $t$, in the right-hand side above belongs  to $C^{1+\beta}([\kappa,\tau], \cX^{\beta}_p)$. This, in turn, follows from H\"older continuity of $F$ combined with \cite[Lemma 3.5.1]{Henry}. Therefore, \eqref{auxcauchy} shows that the mapping $t\mapsto w_{xx}(t)$ belongs to $C^{\beta}([\kappa,\tau], \elpee)$, hence, $w\in C^{\beta}([\kappa,\tau], \hatt W^{2,p}(\Gamma))$. Since both $v$ and $w$ satisfy \eqref{vewint}, by Theorem \ref{strong-solu-thm}, they coincide, thus \eqref{vw2p} holds as asserted.

{
By \eqref{rmilduvsolsmax-1} and \eqref{vw2p},  repeating the arguments of  \eqref{vw2p}, we have
\begin{equation}\lb{vw2p-1}
	v\in C^{\beta} ((0,T_{\max}), \hatt W^{2,r}(\Gamma))\cap C^{1+\beta} ((0,T_{\max}), \cX^{\beta}_r)\quad \forall\, r>1.
\end{equation}
Note that
$$
v_{xx}=v_t-f_3(u,v).
$$
This together with \eqref{vw2p-1} and \eqref{aux-eq11}  implies  \eqref{rmilduvsolsmax-2}.}
\end{proof}

\subsection{Proof of Theorem \ref{Lp}}

In this subsection, we prove Theorem \ref{Lp}.

\begin{proof}[Proof of Theorem \ref{Lp}]
We divide the proof of the theorem into three steps.

\smallskip

\noindent {\bf Step 1.}  In this step,  let $(u(t), v(t))$ be as in Theorem \ref{thm3.1}, that is, $(u(t),v(t))$ is the  solution to \eqref{duh}  if $\tau=0$ and to \eqref{tauone} if $\tau=1$. 
 In both cases one has
 \begin{align}
\lb{raz}
		u(t)&= e^{(\Delta-\sigma)(t-\kappa)}u(\kappa)-\int_{\kappa}^{t}  e^{(\Delta-\sigma)(t-\kappa-s)}\partial_x\left(f_1(u(s), v(s))\partial_x v(s)\right)ds\nonumber\\
		&\quad+\int_{\kappa}^{t}  e^{(\Delta-\sigma)(t-\kappa-s)}\left(\frac{f_2(u(s),v(s))+\sigma u(s)}{u(s)}\right)u(s), t\in[\kappa, T_{\max}), 
\end{align}
with the integral convergent in $\elq$ for arbitrary $q\in[1, p]$, $\kappa>0$.

 Let $\kappa\in (0,T_{max})$.
We prove  the existence and uniqueness of  classical solutions of
\begin{equation}
\label{aux-eq1}
\begin{cases}
	u_t=(\Delta -\sigma)u-a(t, x)\partial_xu+b(t,x)u,\quad x\in \Gamma\cr
 \sum\limits_{\vartheta \sim e} \partial_{\nu}u_e(\vartheta)=0,\ u_e(\vartheta)=u_{e'}(\vartheta), v_e(\vartheta)=v_{e'}(\vartheta), e\sim \vartheta, e'\sim\vartheta,\cr
u(\kappa,x)=u(\kappa)(x),
\end{cases}
\end{equation}
where $a(t,x)=a(t)(x)$ and $b(t,x)=b(t)(x)$ and $a(t),b(t)$  are given as in \eqref{abcoef}.
By Theorem \ref{thm3.1}, we have
$$
a(t,\cdot),b(t,\cdot)\in \hatt C(\overline\Gamma)\,\,\,{\rm for}\,\,  t\in (0,T_{\max}),
$$
$$
a(t,\cdot), b(t,\cdot)\in C^{\beta}((0,T_{\max}), \hatt C(\overline \Gamma)),\, 0<\beta<\frac{1}{8},
$$
and
$$u(\kappa)\in \cX_r^\beta,\, r\ge 1,\,\, 0<\beta<\frac{1}{8},
$$

Let $\Psi(t; u(\kappa)):=T(t,\kappa)u(\kappa)$, where the evolution operator  $T(t, \kappa)$ is as in  Theorem \ref{evolution-operator-thm}. We will show that $u(t)=\Psi(t, u(\kappa))$, $t\in[\kappa, \tau]$, for every $\kappa\in (0,T_{\max})$ and $\tau$ sufficiently close to $\kappa$. 
To that end, abbreviating $\Psi(s)=\Psi(s; u(\kappa))$ we note that for  $t>\kappa$ one has
\begin{equation}\lb{dva}
	\Psi(t)=e^{(\Delta-I)(t-\kappa)}u(\kappa)-\int_{\kappa}^{t}e^{(\Delta-\sigma)(t-\kappa-s)}a(s)\partial_s\Psi(s)ds+\int_{\kappa}^{t}e^{(\Delta-\sigma)(t-\kappa-s)}b(s)\Psi(s)ds
\end{equation}
Let us also notice that
\begin{align}\lb{tri}
&a\partial_s\Psi=\partial_x\left(f_1(\Psi,v)\partial_xv\right)-\frac{\partial_vf_1(\Psi,v)(\partial_xv)^2+f_1(\Psi,v)\partial_{xx}v}{\Psi} \Psi,\\
&b\Psi=\left(\frac{f_2(u,v)+\sigma u}{u}\right)\Psi-\frac{\partial_vf_1(u,v)(\partial_xv)^2+f_1(u,v)\partial_{xx}v}{u} \Psi.
\end{align}

Combining \eqref{raz}, \eqref{dva}, \eqref{tri} we arrive at
\begin{align}
\lb{umpsi}
u(t)-\Psi(t)=&\int_{\kappa}^t e^{(\Delta-\sigma)(t-\kappa-s)}\partial_x\big(f_1(\Psi(s), v(s))-f_1(u(s), v(s))\partial_xv(s)\big)ds\nonumber\\
&+\int_{\kappa}^t e^{(\Delta-\sigma)(t-\kappa-s)}\left(\frac{f_2(u(s),v(s))+\sigma u(s)}{u(s)}\right)(u(s)-\Psi(s))ds\nonumber\\
&+\int_{\kappa}^t e^{(\Delta-\sigma)(t-\kappa-s)}\left(\frac{\partial_vf_1(u,v)}{u}-\frac{\partial_vf_1(\Psi,v)}{\Psi} \right)(\partial_xv)^2\Psi(s)ds\nonumber\\
&+\int_{\kappa}^t e^{(\Delta-\sigma)(t-\kappa-s)}\left(\frac{f_1(u,v)}{u}-\frac{f_1(\Psi,v)}{\Psi} \right)\partial_{xx}v\Psi(s)ds\nonumber\\
&=I+II+III+IV.
\end{align}
Our next objective is to estimate $\elq$ norm of both sides of the above identity and employ Gronwall-type inequality from Proposition \ref{propHenry}. To estimate each term we use Proposition \ref{kprop3.1new} with special choices of $F$ as listed below
\begin{align}
&\text{I term:}\ F(u,\phi)=f_1(u,\phi), \phi:=v, \psi=\partial_xv,\\
&\text{II term:}\ F(u,\phi)= u,  \psi=(f_2(u(s),v(s))+\sigma u(s))(u(s))^{-1}, \\
&\text{III term:}\ F(u,\phi)= u^{-1}\partial_vf_1(u, \phi),\ \phi=v,  \psi=(\partial_x v)^2\Psi,\\
&\text{IV term:}\ F(u,\phi)= \partial_vf_1(u, \phi),\ \phi=v,\  \psi=(\partial_x v)^2\Psi,
\end{align}
to arrive at 
\begin{equation}
\|u(t)-\Psi(t)\|_{\elq}\leq C\int_{\kappa}^t(t-\kappa-s)^{-\alpha}\|u(s)-\Psi(s)\|_{\elq}ds,\  C>0, \alpha\in(0,1).
\end{equation}
Thus by Proposition \ref{propHenry} $u(t)=\Psi(t)$ for $t\geq \kappa$ with sufficiently small $|t-\kappa|$.  Hence, by Theorem \ref{evolution-operator-thm},  $u=\Psi\in \hatt  C^{1,2}([\kappa, \tau]\times\overline{\Gamma})$. 

\noindent{\bf Step 2.} The uniqueness follows from the fact that classical solutions satisfy the integral equation \eqref{parabolic-parabolic-eq}.

\noindent {\bf Step 3.} { In this step, we prove that if $u_0\ge 0$ and $v_0\ge 0$, then $u(t)\ge 0$ and $v(t)\ge 0$ for all $t\in (0,T_{\max})$. We prove this for the case $\tau>0$. The case $\tau=0$ can be proved similarly.
Note that for any $u_0\in L^p(\Gamma)$, $v_0\in \hatt W^{1,p}(\Gamma)$  with $u_0\ge 0, v_0\ge 0$,  there are $u_n\in C(\overline \Gamma)$, 
$v_n\in C^1(\overline\Gamma)$ such that  $u_n\ge 0,v_n\ge 0$ and 
$$
\lim_{n\to\infty} \Big(\|u_n-u_0\|_{L^p(\Gamma)}+\|v_n-v_0\|_{\hatt W^{1,p}(\Gamma)}\Big)=0.
$$
It is not difficult to see that
$$
\lim_{n\to \infty} \|u(t;u_n,v_0)-u(t;u_0,v_0)\|_{L^p(\Gamma)}+\|v(t;u_n,v_n)-v(t;u_0,v_0)\|_{L^p(\Gamma)}=0\quad \forall\, t\in [0,T_{\max}).
$$
By Theorem \ref{general-existence-thm} (3), we have $u(t;u_n.v_n)\ge 0$ and $v(t;u_0,v_n)\ge 0$. It then follows that $u(t;u_0,v_0)\ge 0$ and $v(t;u_0,v_0)\ge 0$.}

\end{proof}

\section{Global well-posedness of logistic-type Keller--Segel systems}\lb{classolsec}
The main goal of this section is to provide the proof of Theorem \ref{global-existence-thm}. Let us outline the strategy of the proof. 

\subsection{Outline of the proof of Theorem \ref{global-existence-thm}}\lb{outline2}

Employing local existence results from Section \ref{mildsol} yields classical solutions of \eqref{logparabolic-parabolic-eq} on the maximal existence time interval $[0,T_{\max})$\footnote{for brevity we will abbreviate the maximal time of existence in both cases $\tau=0$ and $\tau>0$ by $T_{\max}$ instead of $T_{\max}(u_0)$ or $T_{\max}(u_0, v_0)$}. To prove global well-posedness, i.e., $T_{\max}=\infty$ we rule out finite time blow up of 	$t\mapsto(u(t), v(t))$ in ${\elq\times \hatt W^{1,q}(\Gamma)}$ for all $q\geq 1$. This is achieved by induction in $q$. The base case $q=1$ is proved in Lemma \ref{L1-boundedness-lm} by deriving from \eqref{logparabolic-parabolic-eq} the following differential inequality 
\begin{equation}
	\label{L1-estimate-eq1}
	\frac{d}{dt}\int_\Gamma u(t;u_0)dx\le k|\Gamma|+l\int_\Gamma u(t;u_0)dx-\frac{m}{|\Gamma|^{\varepsilon}}\left(\int_\Gamma u(t;u_0)dx\right)^{1+\varepsilon}.
\end{equation}
Then a bound on $\|u(t)\|_{\el{1}}$ follows from the comparison principle for ODEs. 
To make the induction step, we derive the following inequality
\begin{align}
\frac{d}{dt}\int_{\Gamma} u^{n}dx\leq C\left( 1+\int_{\Gamma} u^ndx-\left(\int_\Gamma u^{n}dx\right)^{1+1/n}\right),
\end{align}
where $C>0$ depends on  $\|(u(t), v(t))\|_{\el{n-1}\times \hatt W^{1,n-1}(\Gamma)}$. Using the induction hypothesis and the comparison principle we then obtain a bound on $\|(u(t), v(t))\|_{\el{n}\times \hatt W^{1,n}(\Gamma)}$. For $\tau=0$ and either $m>0$ or $m=k=l=0$ we show that the above norms are uniformly bounded with respect to $t$.  This argument is carefully carried out in Lemma \ref{L1-boundedness-lm}. 

To prove global boundedness in case $\tau=0$ with $m>0$ or $m=k=l=0$ we use the integral equation
\begin{align}\lb{reshur}
		u(t)=e^{t(\Delta-\sigma) }u_0&+\int_{0}^{t}  \chi e^{(t-s)(\Delta-\sigma)}\partial_x\left(u(s)\partial_xv(s)\right) ds\nonumber\\
		&+\int_{0}^{t} e^{(t-s)(\Delta-\sigma)}(g(u(s))+\sigma u(s))ds.
\end{align}
The estimates derived in Lemmas \ref{L1-boundedness-lm} and \ref{Lp-boundedness-lm} together with \eqref{reshur} yield bound on $\|u(t)\|_{\cX^{\beta}_q}$ in terms of $\|(u(t), v(t))\|_{\el{q}\times \hatt W^{1,q}(\Gamma)}$. Combining this bound with the embedding $\cX^{\beta}_{q}\hookrightarrow L^{\infty}(\Gamma)$ gives uniform boundedness of solutions. This argument is spelled out at the end of current section. 
\subsection{Norm estimates for solutions of the logistic type model.}
\begin{lemma}
\label{L1-boundedness-lm}
Assume hypotheses of Theorem \ref{global-existence-thm} and let $T>0$. 
\begin{enumerate}
	\item\lb{L1item1} If $\tau=0$ then one has
	\begin{equation}\lb{5.4}
\sup\limits_{t\in [0,\min\{T,T_{\max}(u_0)\})} \|u(t;u_0)\|_{\el{1}}\leq {  M_0},
	\end{equation}
	 where  
	\begin{equation}
		M_{0}{  =M_0(u_0,k,l,m,T)}:= \begin{cases}
			C(u_0, k,l,m), &{\rm if}\,m>0,\\
			e^{{  l\hat T}}\int u_0dx+\frac{k|\Gamma|}{l}\Big(e^{{  l\hat T}}-1\Big) ,& {\rm if}\,\, m=0,\,\, {  l}>0,\\
			\int_\Gamma u_0(x)dx+k|\Gamma| {  \hat T},& {\rm if}\,\, m={ l}=0,
		\end{cases}
	\end{equation}
for some $T$-independent constant $C(u_0, k,l,m)>0$, {  $\hat T=\min\{T,T_{\max}(u_0)\}$.}
	\item\lb{L1item2} If  $\tau>0$ then $0\leq \int u(t;u_0, v_0)dx\leq {  M_0}$, for all $t\in [0,\min\{T,T_{\max}(u_0, v_0)\})$ and ${  M_0}$ as above. Moreover,  for arbitrary $q\ge 1$ there is $M_1=M_1(u_0, v_0, k,l,m, {  T})>0$ such that
	\begin{equation}\lb{5.6}
	\sup\limits_{t\in [0,\min\{T,T_{\max}(u_0, v_0)\})}\|v(t;u_0,v_0)\|_{\hatt W^{1,q}(\Gamma)}\le M_1.
	\end{equation}
\end{enumerate}
\end{lemma}

\begin{proof}
\eqref{L1item1} Integrating both sides of the first equation in \eqref{logparabolic-parabolic-eq} over $\Gamma$ and using \eqref{gest3} we obtain
$$
\frac{d}{dt}\int_\Gamma u(t;u_0)dx\leq k|\Gamma|+l\int_\Gamma u(t;u_0)dx-m\int_\Gamma u^{1+\varepsilon}(t;u_0)dx.
$$
This inequality together with
$$
\int_\Gamma u^{1+\varepsilon}(t;u_0)dx\ge {|\Gamma|}^{-\varepsilon}\left(\int_\Gamma u(t;u_0)dx\right)^{1+\varepsilon}
$$
yields \eqref{L1-estimate-eq1}.

Then \eqref{L1-estimate-eq1} and comparison principle for scalar ODEs yield the following assertions. 
If $m>0$, then
\begin{equation}
\label{m-pos-eq}
\int_\Gamma u(t;u_0)dx\le \max\left\{\int_\Gamma u_0dx, r(k,l,m)\right\}, t\in [0,\min\{T,T_{\max}(u_0)\}),
\end{equation} where $r(k,l,m)>0$ is the positive zero of $\varphi(x):= k|\Gamma|+lx-m|\Gamma|^{-\varepsilon}x^{1+\varepsilon}$. If $m=0$ and $l>0$ then
\begin{equation}
\label{m-zero-l-pos-eq}
\int_\Gamma u(t;u_0)dx\le 
e^{lt}\int_{\Gamma}u_0dx+\frac{k|\Gamma|}{l}\Big(e^{lt}-1\Big),\  t\in [0,\min\{T,T_{\max}(u_0)\}).
\end{equation}
If $m=l=0$, then 
\begin{equation}
\label{m-zero-l-zero-eq}
\int_\Gamma u(t;u_0)\le \int_\Gamma u_0dx+k|\Gamma|t,\  t\in [0,\min\{T,T_{\max}(u_0)\}).
\end{equation}
{  The assertion \eqref{gest3} follows from \eqref{m-pos-eq}, \eqref{m-zero-l-pos-eq}, and \eqref{m-zero-l-zero-eq}.}

\eqref{L1item2}.  The first assertion is proved as in \eqref{L1item1}. Let us prove the second assertion. Integrating the second equation in \eqref{logparabolic-parabolic-eq} we obtain
\begin{equation}
\frac{d}{dt}\int_\Gamma v(t;u_0, v_0)dx\leq \int_\Gamma u(t;u_0, v_0)dx-\int_\Gamma v(t;u_0, v_0)dx\leq M_0-\int_\Gamma v(t)dx,
\end{equation}
hence, for some $C=C(v_0)>0$ one has
\begin{equation}\lb{5.1}
\int_\Gamma v(t;u_0, v_0)dx\leq \max\left\{\int_{\Gamma} v_0 dx, M_0  \right\}:={  M_0'}. 
\end{equation}
Next, let us re-write the second equation in \eqref{logparabolic-parabolic-eq} as follows
\begin{equation}
v_t=v_{xx}-(1+\sigma_0)v+u+\sigma_0v,
\end{equation}
where $\sigma_0>0$ is as in Corollary \ref{cor2.1}. Then denoting $\omega:=1+\sigma_0$ one obtains
\[\partial_xv(t)=\partial_xe^{(\Delta-\omega)t}v_0+\int_{0}^{t}\partial_xe^{(\Delta-\omega)(t-s)}\big(u(s)+\sigma_0v(s)\big)ds.\]
Since $\omega>\sigma_0$ we can employ Corollary \ref{cor2.1} to get some $\delta\in(0, \omega)$ and a $T$-independent constant $C=C(\omega, p, \Gamma, \delta)>0$ for which one has
\begin{align}
&	\| \partial_x v(t)\|_{\elq}  \leq \|\partial_xe^{(\Delta-\omega)t}v_0\|_{\elq}  +\int_{0}^{t}\|\partial_xe^{(\Delta-\omega)(t-s)}\big(u(s)+\sigma_0v(s))\|_{\elq}ds \nonumber \\
	&\leq  Ct^{-\frac12-\frac12\left(1-\frac1q\right)}e^{-\delta t}\|v_0\|_{\el{1}}\no\\
	&\qquad+C\sup_{t\in [0,\min\{T,T_{\max}\})}( \|u(t)\|_{L^1(\Gamma)}+\|v(t)\|_{L^1(\Gamma)}) \int_{0}^{t}(t-s)^{-\frac12-\frac12\left(1-\frac1q\right)}e^{-\delta (t-s)} ds\no\\
	&\leq Ct^{-\frac12-\frac12\left(1-\frac1q\right)}e^{-\delta t}\|v_0\|_{\el{1}}+C\sup_{t\in [0,\min\{T,T_{\max}\})}( \|u(t)\|_{L^1(\Gamma)}+\|v(t)\|_{L^1(\Gamma)})\no
\end{align}
for all  $t\in [0, \min\{T,T_{\max}\})$. Combining this with  \eqref{5.4}, \eqref{5.1} we get
\begin{align}
	&	\| \partial_x v(t)\|_{\elq} \leq Ct^{-\frac12-\frac12\left(1-\frac1q\right)}e^{-\delta t}\|v_0\|_{\el{1}}+C(1+{  M_0'}),
\end{align}
for some $T$-independent constant $C=C(u_0,v_0, \Gamma, \omega, \delta)>0$.  By a similar argument, one obtains
\begin{align}
	&	\| v(t)\|_{\elq} \leq Ct^{-\frac12\left(1-\frac1q\right)}e^{-\delta t}\|v_0\|_{\el{1}}+C(1+{  M_0'}).
\end{align}
Then for $0<\varepsilon\ll 1$ there exists $C=C(u_0,v_0, \Gamma, \omega, \delta, \varepsilon)>0$ such that
\begin{equation}\lb{5.14}
	\|v(t;u_0,v_0)\|_{W^{1,q}(\Gamma)}\le C ({  M_0'}+1), \text{\ for all\ } t\in[\varepsilon, \min(T, T_{\max})).
\end{equation}
By the second identity in \eqref{milduvsols}, for sufficiently small $\varepsilon$ we also have
\begin{equation}\lb{5.15}
	\|v(t;u_0,v_0)\|_{W^{1,q}(\Gamma)}\le C, \text{\ for all\ } t\in[0,\varepsilon].
\end{equation}
Combining \eqref{5.14}, \eqref{5.15} we obtain \eqref{5.6}.

\end{proof}

\begin{lemma}
	\label{Lp-boundedness-lm}
	Assume hypotheses of Theorem \ref{global-existence-thm}. Then the following assertions hold for arbitrary $q\geq1$.
	
	\begin{enumerate}
		\item\lb{Lpitem1} If $\tau=0$ then there exists  $M(q,T,u_0)>0$ such that
		
	\begin{equation}\lb{item1eq}
	\|(u(t;u_0), v(t;u_0))\|_{\elq\times \hatt W^{1,q}(\Gamma)}\le M(q,T,u_0),\ t\in [0,\min\{T,T_{\max}(u_0)\}).
	\end{equation}
		Moreover, if $m>0$ or $k=l=m=0$, then $M(q,T,u_0)$ can be chosen independently of $T$.
		\item\lb{Lpitem2} If $\tau>0$ then there exists  $M(q,T,u_0, v_0)>0$ such that
		\begin{equation}\lb{item2eq}
			\|(u(t;u_0), v(t;u_0))\|_{\elq\times \hatt W^{1,q}(\Gamma)}\le M(q,T,u_0, v_0),\ t\in [0,\min\{T,T_{\max}(u_0, v_0)\}).
		\end{equation}
		Moreover, if $m>0$ or $k=l=m=0$, then $M(q,T,u_0)$ can be chosen independently of $T$.
	\end{enumerate}
\end{lemma}

\begin{proof}
First, for $\tau\geq0$ and $n\geq 2$ we claim the following auxiliary identity 
		\begin{equation}\lb{ibp3}
		\frac{1}{n} \frac{d}{dt}\int_{\Gamma} u^{n}dx=  -(n-1) \int_{\Gamma}  u^{n-2}|u_{x}|^2dx +(n-1)\chi  \int_{\Gamma}u^{n-1}u_x v_x dx +  \int_{\Gamma}u^{n-1} g(u) dx.
	\end{equation}
	
To begin the derivation of \eqref{ibp3} let us multiply both sides of the first equation in \eqref{logparabolic-parabolic-eq} by $u^{n -1}$ and  integrate over $\Gamma$
	\begin{equation} \lb{ibp1}
		\int_{\Gamma}  u^{n-1} u_tdx =  \int_{\Gamma}  u^{n-1}u_{xx}dx -\chi  \int_{\Gamma}u^{n-1}(u v_x)_x dx +  \int_{\Gamma}u^{n-1} g(u) dx. 
	\end{equation}
	Then using
	\begin{equation}
		\int_{\Gamma}  u^{n-1} u_t dx = \frac{1}{n} \frac{d}{dt}\int_{\Gamma} u^{n}dx,
	\end{equation}
	we infer
	\begin{equation} \lb{ibp2}
		\frac{1}{n} \frac{d}{dt}\int_{\Gamma} u^{n}dx=   \int_{\Gamma}  u^{n-1}u_{xx}dx-\chi  \int_{\Gamma}u^{n-1}(u v_x)_xdx +  \int_{\Gamma}u^{n-1} g(u) dx.
	\end{equation}
	Next, integrating by parts we obtain
	\begin{align}
		\int_{\Gamma}  u^{n-1}u_{xx}dx&= -\int_{\Gamma} ( u^{n-1})_xu_{x}dx+\sum_{\theta\in \cV}\sum_{e\sim \theta}u^{n-1}_e(\theta)\partial_{\nu}u_e(\theta)\nonumber\\
		& = -\int_{\Gamma} ( u^{n-1})_xu_{x}dx+\sum_{\theta\in \cV}u^{n-1}(\theta)\sum_{e\sim \theta}\partial_{\nu}u_e(\theta)=-\int_{\Gamma} ( u^{n -1})_xu_{x}dx,
	\end{align}
	where we used the fact that $u$ satisfies Neumann--Kirchhoff conditions, that is, $u_e(\vartheta)=u(\vartheta)$ for every edge $e$ adjacent to vertex $\vartheta$ and 
	\begin{equation}
		\sum_{e\sim \theta}\partial_{\nu}u_e(\theta)=0.
	\end{equation}
	Integrating by parts the second term in the right-hand side of \eqref{ibp2} in a similar fashion we arrive at \eqref{ibp3}.

	\eqref{Lpitem1} It suffices to prove  the assertion for $q=n\in\bbN$. We prove this by induction. 
	
	The base case $q=1$ is given by Lemma \ref{L1-boundedness-lm}. Assume that the it holds for $q=n-1$ ($n\ge 2$). To make the induction step we will derive the following inequality 	
	\begin{align}\label{LP-estimate-eq5}
		\frac{1}{n} \frac{d}{dt}\int_{\Gamma} u^{n}dx\leq C\left( 1+\int_{\Gamma} u^ndx-\left(\int_\Gamma u^{n}dx\right)^{1+1/n}\right), t\in (0,\min\{T,T_{\max}\}),
	\end{align}
 for some $C=C(M(n-1,T,u_0))>0$ (with $x\mapsto C(x)$ having no blow-ups).  Assuming \eqref{LP-estimate-eq5} one obtains the assertions item \eqref{Lpitem1} as follows: 
 \begin{itemize}
 \item 	if $m>0$ or $k=l=m=0$ then, by induction assumption, $M(n-1, T, u_0)$ can be chosen independently from $T$ in which case \eqref{LP-estimate-eq5} yields $M(n,u_0)>0$ such that
 $$
 \int_\Gamma u^n(t;u_0)dx\le M(n,u_0),
 $$
 \item if $m=0$ and $k^2+l^2\not =0$ then \eqref{LP-estimate-eq5} yields 
 $$
 \int_\Gamma u^n(t;u_0)dx\le M(n,T,u_0). 
 $$
 \end{itemize}
 
Let us now derive \eqref{LP-estimate-eq5}. To that end let us rewrite the second term in the right-hand side of \eqref{ibp3}.
Multiplying both sides of the second equation in \eqref{logparabolic-parabolic-eq} (recall that $\tau=0$) by $u^{n}$ and integrating by parts as before yields
\begin{equation} \notag
	n \int_{\Gamma} u^{n-1} u_{x}v_{x} dx=  -  \int_{\Gamma}u^{n} vdx +  \int_{\Gamma}u^{n+1} dx.
\end{equation}
Combining this with \eqref{ibp3} one obtains
\begin{align}\notag
	\frac{1}{n} \frac{d}{dt}\int_{\Gamma} u^{n}dx&=  -(n-1) \int_{\Gamma}  u^{n-2}|u_{x}|^2dx -\frac{(n-1)\chi}{n}  \int_{\Gamma}u^{n}vdx\\
	&\quad + \frac{(n-1)\chi}{n}  \int_{\Gamma}u^{n+1}dx+ \int_{\Gamma}u^{n-1} g(u) dx.
\end{align}
This together with \eqref{gest3} implies that
	\begin{align} 
		\frac{1}{n} \frac{d}{dt}\int_{\Gamma} u^{n}dx\leq & -(n-1) \int_{\Gamma}  u^{n -2}|u_{x}|^2dx  +\frac{(n -1)\chi}{n}  \int_{\Gamma}u^{n+1}dx\nonumber\\
		& + k \int_{\Gamma}u^{n-1}  dx+l\int_{\Gamma} u^ndx,
	\end{align}
here, and throughout the proof, the negative term $\int_{\Gamma}-mu^{n+\varepsilon}dx$ has been dropped. 
	Noticing that
	$$
	\int_\Gamma u^{n-2}|u_x|^2dx=\frac{4}{n^2}\int_\Gamma  \left|(u^{\frac{n}{2}})_x\right|^2 dx,   
	$$
	we proceed
	\begin{align} \label{gamma_est1}
		\frac{1}{n} \frac{d}{dt}\int_{\Gamma} u^{n}dx\leq & -\frac{4(n-1)}{n^2} \int_{\Gamma}  |(u^{\frac{n}{2}})_{x}|^2dx \nonumber\\
		& + k \int_{\Gamma}u^{n-1}  dx+l\int_{\Gamma} u^ndx+\frac{(n -1)\chi}{n}\int_\Gamma u^{n+1}dx.
	\end{align}
	Let us observe that
	\begin{equation}
	\int_\Gamma u^{n+1}dx= \int_\Gamma \big(u^{\frac{n}{2}}\big)^{\frac{2(n+1)}{n}}dx=\|u^{\frac{n}{2}}\|_{L^{\frac{2(n+1)}{n}}(\Gamma)}^{\frac{2(n+1)}{n}},
	\end{equation}
	and use the Gagliardo–Nirenberg interpolation inequality to obtain
	\begin{equation}
		\label{LP-estimate-eq1}
		\|u^{\frac{n}{2}}\|_{L^{\frac{2(n+1)}{n}}(\Gamma)}\le C\|(u^{\frac{n}{2}})_x\|_{L^2(\Gamma)}^\theta \|u^{\frac{n}{2}}\|_{L^{\frac{2(n-1)}{n}}(\Gamma)}^{1-\theta}+C\|u^{\frac{n}{2}}\|_{L^{\frac{2(n-1)}{n}}(\Gamma)},
	\end{equation}
	where $\theta=\frac{2n}{(n+1)(2n-1)}$ and $C=C(n, \Gamma)>0$. Therefore, one has
	\begin{equation}
		\label{LP-estimate-eq2}
		\int_\Gamma u^{n+1}dx\le C \Big(\int_\Gamma |(u^{\frac{n}{2}})_x|^2dx\Big)^{\frac{2}{2n-1}}
		\Big(\int_\Gamma u^{n-1}dx\Big)^{\frac{2(2n^2-n-1)}{(2n-1)n}}+C\Big(\int_\Gamma u^{n-1}dx\Big)^{\frac{n+1}{n-1}}.
	\end{equation}
	Applying Young's inequality to the first term in the right-hand side above we obtain $C=C(n)>0$ such that
	\begin{align}
		\label{LP-estimate-eq3}
		\int_\Gamma u^{n+1}dx\le \frac{2}{n\chi} \int_{\Gamma}  |(u^{\frac{n}{2}})_{x}|^2dx &+ C
		\Big(\int_\Gamma u^{n-1}dx\Big)^{\frac{2(2n^2-n-1)}{(2n-1)n}\frac{2n-1}{2n-3}}\nonumber\\
		&+C\Big(\int_\Gamma u^{n-1}dx\Big)^{\frac{n+1}{n-1}},
	\end{align}
	This together with \eqref{gamma_est1} yield
	\begin{align}
		\label{LP-estimate-eq4}
		\frac{1}{n} \frac{d}{dt}\int_{\Gamma} u^{n}dx\le & k\int_{\Gamma}u^{n-1}  dx+l\int_{\Gamma} u^ndx-\frac{(n -1)\chi}{2n}\int_\Gamma u^{n+1}dx\nonumber\\
		&\leq M(n-1,T, u_0)+l\int_{\Gamma} u^ndx-\frac{(n -1)\chi}{2n}\int_\Gamma u^{n+1}dx.
	\end{align}
	Then using 
	\begin{equation}\lb{hold}
		\int_{\Gamma} u^n dx\leq \left(\int_{\Gamma} u^{n+1}dx\right)^{\frac{n}{n+1}}|\Gamma|^{\frac1{n+1}},
	\end{equation}
	we obtain \eqref{LP-estimate-eq5}.
	
	\eqref{Lpitem2} \sel{It suffices to prove \eqref{item2eq} for $q=n\in\bbN$}. We prove this by induction. The base case $q=1$ is given by Lemma \ref{L1-boundedness-lm}. Assume that it holds for $q=n-1$ ($n\ge 2$). To make the induction step we will derive \eqref{LP-estimate-eq5}  with  $C=C( M(n-1, T, u_0, v_0))>0$. Given such an inequality, the assertions of  item \eqref{Lpitem2} follow as in the proof of item \eqref{Lpitem1}.
	
	Let us now derive \eqref{LP-estimate-eq5}  with  $C=C( M(n-1, T, u_0, v_0))>0$ for $\tau=1$ (the derivation for $\tau>0$ is analogous). Combining \eqref{gest3} and \eqref{ibp3} one obtains
	\begin{align*}
		\frac{1}{n}\frac{d}{dt}\int_{\Gamma}u^{n}dx&\le -(n-1)\int_{\Gamma}  u^{n-2} |u_x|^2dx+(n-1)\chi\int_{\Gamma}u^{n-1}u_x v_xdx \\  
		&\quad +\int_{\Gamma}u^{n-1}(k+lu)dx\\
		&\le      -\frac{(n-1)}{2}\int_{\Gamma}  u^{n-2} |u_x|^2dx+\frac{(n-1)\chi^2}{2}\int_{\Gamma}u^{n} |v_x|^2dx \\   &\quad +\int_{\Gamma}u^{n-1}(k+lu)dx,
	\end{align*}
	where in the second step we used Young's inequality for 
	\begin{equation}
		\chi u^{n-1} u_x v_x= (u^{\frac n2-1}u_x )(\chi u^{\frac n2} v_x). 
	\end{equation}
	Next, using Young's inequality and Lemma  \ref{L1-boundedness-lm}  \eqref{Lpitem2} one obtains
	\begin{align}
		\begin{split}\lb{ondt1}
			\frac{(n-1)\chi^2}{2}\int_\Gamma &u^n|v_x|^2dx\le \int_\Gamma u^{n+1}dx+C\int_\Gamma |v_x|^{2(n+1)}dx\\
			&\le  \sel{\int_\Gamma u^{n+1}dx +C\left(1+N\right)^{1/(2(n+1))}}\\
			&\leq\int_\Gamma u^{n+1}dx +C(M(1, T, u_0, v_0))^{1/(2(n+1))}, 
		\end{split}
	\end{align}
	for some $C>0$.  Therefore, one has
	\begin{align}
		\label{parabolic-parabolic-global-eq2-new}
		\frac{1}{n}\frac{d}{dt}\int_{\Gamma}u^{n}dx
		&\le       -\frac{(n-1)}{2}\int_{\Gamma}  u^{n-2} |u_x|^2dx+\int_\Gamma u^{n+1}dx  \nonumber\\
		&\quad + k \int_{\Gamma}u^{n-1}dx  +l \int_\Gamma u^ndx\nonumber \\
		&=-\frac{2(n-1)}{n^2}\int_\Gamma |( u^{\frac{n}{2}})_x|^2dx +\int_\Gamma u^{n+1}dx \nonumber  \\
		&\quad + k \int_{\Gamma}u^{n-1}dx +l \int_\Gamma u^ndx  + C(M(1, T, u_0, v_0))^{1/(2(n+1))},
	\end{align}
	for some $C>0$. Using \eqref{LP-estimate-eq2} and Young's inequality one obtains
	\begin{align}
		\lb{onedt2}
		\int_\Gamma u^{n+1}dx\le \frac{(n-1)}{n^2} \int_{\Gamma}  |(u^{\frac{n}{2}})_{x}|^2dx &+ C
		\Big(\int_\Gamma u^{n-1}dx\Big)^{\frac{2(2n^2-n-1)}{(2n-1)n}\frac{2n-1}{2n-3}}\nonumber\\
		&+C\Big(\int_\Gamma u^{n-1}dx\Big)^{\frac{n+1}{n-1}}.
	\end{align}
	Then combining \eqref{parabolic-parabolic-global-eq2-new}, \eqref{onedt2}, \eqref{hold} and the induction hypothesis we get  \eqref{LP-estimate-eq5}  with  $C=C( M(n-1, T, u_0, v_0))>0$.
\end{proof}
\subsection{Proof of Theorem \ref{global-existence-thm}}

\sel{\begin{proof}[Proof of Theorem \ref{global-existence-thm}] The finite time blow-up of $t\mapsto \|(u(s), v(s))\|_{L^{q}(\Gamma)\times W^{1,q}(\Gamma)}$ has been ruled out in Lemma \ref{Lp-boundedness-lm}, that is, $T_{\max}=\infty$. 
	
To prove \eqref{1.18} and \eqref{1.20} 	it suffices to obtain a uniform, with respect to time, bound on $L^{\infty}(\Gamma)$-norm of each term in the right-hand side of \eqref{reshur}, where we assume without loss of generality that $\sigma>\sigma_0$ for $\sigma_0$ is as in Corollary \ref{cor2.1}. This choice, in particular, allows us to use \eqref{3.22} with $t\in(0, T]$ replaced by $t\in(0,\infty)$, some $\delta\in(0, \sigma)$ and $t$-independent constant $C>0$. 

Let us now switch to the first term in \eqref{reshur}. Using comparison principle from Theorem \ref{general-existence-thm} with $f(t,x,u,v)=u$ one obtains
	\begin{equation}\lb{reshurest3}
		\|e^{t(\Delta -\sigma)}u_0\|_{L^{\infty}(\Gamma)}\leq \max_{x\in\overline{\Gamma}}|u_0(x)|, t>0. 
	\end{equation}
	Let us estimate $L^{\infty}-$norm of the second term in \eqref{reshur}. Let us pick  $q>2$, $\beta \in (0,1/8)$ with  $2\beta-q^{-1}>0$  and recall that $u(t)\in \cX^{\beta}_{q}$, $t\in (0, T_{\max})$. Using the embedding $\cX^{\beta}_{q}\hookrightarrow L^{\infty}(\Gamma)$ we obtain
		\begin{align}
			\begin{split}
			\label{reshurest1}
			& \int_0^t \left\|e^{(\Delta-\sigma)(t-s)}\partial_x \big( u(s)\partial_x v(s)\big)\right\|_{\cX_q^{\beta}}  ds\\
			& \leq C\int_0^t (t-s)^{^{-\beta-\frac{q+1}{2q}} }e^{{-\delta(t-s)}}\|(u(s), v(s))\|_{L^{q}(\Gamma)\times W^{1,q}(\Gamma)}ds\\
			&\leq C \sup\limits_{s\in[0, \infty)}\|(u(s), v(s))\|_{L^{q}(\Gamma)\times W^{1,q}(\Gamma)}<\infty, t>0,
			\end{split}
		\end{align}
	where $C$ is $t$-independent and the last inequality holds due to the forgoing assumption that either  $m>0$ or  $k=l=m=0$ and Lemma \ref{Lp-boundedness-lm}.
	
		Next, employing \eqref{gest3}, Lemma \ref{Lp-boundedness-lm} and the comparison principle we obtain
		\begin{align}
			\begin{split}
\lb{reshurest2}
\int_{0}^{t} e^{(t-s)(\Delta-\sigma)}(g(u(s))+\sigma u(s))ds&\leq  \int_{0}^{t} e^{(t-s)(\Delta-\sigma)}(C+l u(s))ds\\
&\leq C(u_0)\int_{0}^{t} e^{-(t-s)\sigma/2}ds, t>0
			\end{split}
		\end{align}
		Combining \eqref{reshur}, \eqref{reshurest3}, \eqref{reshurest1}, \eqref{reshurest2} one obtains
		\begin{equation}
				\limsup_{t\to\infty} \|u(t)\|_{L^\infty(\Gamma)}<\infty.
		\end{equation}

\end{proof}
}
\appendix 
\section{Functional spaces, interpolation spaces, Sobolev embedding theorems  on metric graphs}\label{app:FunctionalSpaces}
In this section we record several facts about fractional power spaces $\cX^{\alpha}_p$, $\alpha\in(0,1)$, $1\leq p< \infty$ generated by the Neumann--Kirchhoff Laplacian on compact metric graphs, more concretely, by the analytic semigroup in $e^{t\Delta}$. By definition, see, e.g., \cite[Section 1.3]{Henry}., $\cX^{\alpha}_p:=\dom(({ \sigma}I_{\elpee}-\Delta)^{\alpha})$  is equipped with the graph norm of $({\sigma}I_{\elpee}-\Delta)^{\alpha}$ { for some $\sigma>0$ (we note that $\dom(({ \sigma}I_{\elpee}-\Delta)^{\alpha})=\dom(({\tilde  \sigma}I_{\elpee}-\Delta)^{\alpha})$ and
the graph norms of $\dom(({ \sigma}I_{\elpee}-\Delta)^{\alpha})=\dom(({\tilde  \sigma}I_{\elpee}-\Delta)^{\alpha})$ are equivalent
 for $\sigma,\tilde\sigma>0$)}. Throughout this paper we use {  compact}  embeddings of $\cX^{\alpha}_p$ into various function spaces. Such embeddings are well known in the case of classical domains $\Omega\subset\bbR^n$, cf. \cite[Chapter 1]{Henry}. To the best of our knowledge, these type of embedding for metric  graphs, although expected, have not appeared in print. For completeness of exposition we present them in Theorem \ref{analytic-semigroup-thm2} below. Let us first introduce some notation.  The edge-wise direct sum of Banach spaces of functions will be denoted by  $\hatt\  $, in particular, we write
\begin{align}\lb{hats}
	\hatt C_0^{\infty}(\Gamma):= \bigoplus_{e\in\cE} C_0^{\infty}(e), \,\,\, &\hatt W^{k,p}(\cG):= \bigoplus_{e\in\cE}W^{k,p}(e), \,\, k\in\bbN_0,\,\,\, \hatt C^{\nu}(\overline{\cG}) := \bigoplus_{e\in\cE}\hatt C^{\nu}(\overline{e}),
\end{align}
where $\nu>0$ and $C^{\nu}(\overline{e})$ denotes the usual space of H\"older continuous functions defined on the closed interval $\overline{e}$ 
endowed with the norm
\begin{equation}
	\|u\|_{C^\nu(\bar e)}=\sum_{\alpha\in\bbN_0,\alpha\le [\nu]}\sup_{x\in\bar e} |u^{(\alpha)}(x)|+\sup_{x,y\in\bar e,x\not =y}
	\frac{|u^{([\nu])}(x)-u^{([\nu])}(y)|}{|x-y|^{\nu-[\nu]}}.
\end{equation}
Let us note that the edge-wise direct sums introduced in \eqref{hats} induce no vertex conditions as oppose to the  space of continuous functions on the closure $\overline{\cG}$ of the graph
\begin{equation}
	C(\overline \Gamma)=\{u\in \hatt C(\overline \Gamma)\, :\, u\,\, {\rm is \,\, continuous\,\, at\,\, the\,\, vertices\,\, of}\,\, \Gamma\}.
\end{equation}

In the following theorem $(L^p(\Gamma),\hatt{W}^{2,p}(\Gamma))_{\theta,q}$ denotes the interpolation space between $L^p(\Gamma)$ and
$\hatt{W}^{2,p}(\Gamma)$ via the $K$-method, where $0<\theta<1$ and $1\le q<\infty$ (see  \cite[Section 1.3.2]{Tri} for definition). 

\begin{theorem}
	\label{analytic-semigroup-thm2} Suppose that $1\leq p<\infty$. Then one has
	\begin{itemize}
		\item[(1)] For any $q\ge p$ and $s-\frac{1}{p}> t-\frac{1}{q}$, one has
			\begin{equation}
				\label{sobolev-embedding-eq3-0}
				\hatt W^{s,p} (\Gamma) \doublehookrightarrow
\hatt W^{t,q}(\Gamma),
			\end{equation}
			\begin{equation}
				\label{sobolev-embedding-eq2}
				\hatt{W}^{s,p}(\Gamma) \doublehookrightarrow \hatt{C}^r(\overline \Gamma)
				\quad r<s-\frac{1}{p},
			\end{equation} 
			and
			for $s\in (0,1)\setminus\{\frac{1}{2}\}$ one has
			\begin{equation}
				\label{sobolev-embedding-eq3}
				(L^p(\Gamma), \hatt{W}^{2,p}(\Gamma))_{s,p}=\hatt{W}^{2s,p}(\Gamma),
			\end{equation}
		where $\doublehookrightarrow$ denotes compact embedding. 
		\item[(2)] One has
		\begin{equation}
			\label{sobolev-embedding-eq4}
			(\hatt{L}^p(\Gamma), X_p^\alpha)_{\theta,p}=(\hatt{L}^p(\Gamma),\mathcal{D}(\cA_p))_{{\alpha}\theta,p},\quad \, 0<\theta<1,
		\end{equation}
		
		\begin{equation}
			\label{sobolev-embedding-eq5-0}
			X_p^\alpha \doublehookrightarrow \hatt {W}^{2\alpha\theta,p}(\Gamma),\quad \, 0<\theta<1,
		\end{equation}
		
		and 
		\begin{equation}
			\label{sobolev-embedding-eq5}
			X_p^\alpha \doublehookrightarrow \hatt C^\nu(\overline \Gamma),\quad \, 0<\theta<1,\,\,  0<\nu <2\alpha\theta -\frac{1}{p}.
		\end{equation}
	\end{itemize}
\end{theorem}

\begin{proof}
	(1) \eqref{sobolev-embedding-eq3-0} and \eqref{sobolev-embedding-eq2} follow from \cite[Theorem 11.5]{Ama},  and
		\eqref{sobolev-embedding-eq3} follows from \cite[Theorem 11.6]{Ama}.	
	(2) First, \eqref{sobolev-embedding-eq4} follows from   \cite[Section { 1.15.2}, (unlabaled) Theorem on page 101]{Tri}.
	
	To prove \eqref{sobolev-embedding-eq5-0},  for a given $0<\nu<2\alpha-\frac{1}{p}$, let us 
	choose $\theta\in (0,1)$ such that $\frac{\alpha}{2}\theta\not =1$ and $2 \alpha \theta-\frac{1}{p}>\nu$. Then one has
	\begin{align*}
		\cX_p^\alpha &\subset (L^p(\Gamma), \cX_p^\alpha)_{\theta,p}=(L^p(\Gamma),\mathcal{D}(\cA_p))_{{\alpha}\theta,p}\underset{\eqref{sobolev-embedding-eq4}}{\subset} ((L^p(\Gamma),\hatt{W}^{2,p}(\Gamma))_{{\alpha \theta},p}\underset{ \eqref{sobolev-embedding-eq3}}{=}\hatt{W}^{2\alpha\theta,p}(\Gamma).
	\end{align*}
The embedding	\eqref{sobolev-embedding-eq5} follows from \eqref{sobolev-embedding-eq2} and \eqref{sobolev-embedding-eq5-0}.
\end{proof}

\begin{proposition}{\cite[Theorem 1.4.3]{MR610244}}
Let $\sigma>0$, $\alpha\in(0,1)$, $p\in[1, \infty)$ and $\delta\in(0, \sigma)$. Then there exists $C>0$ such that for arbitrary $t>0$ and $u\in\elpee$, $v\in\cX^{\alpha}$ one has
\begin{align}
	&\|(\sigma-\Delta)^{\alpha}e^{(\Delta-\sigma)t}u\|_{\elpee}\leq C e^{-\delta t}t^{-\alpha} \|u\|_{\elpee},\lb{fracpow}\\
	&	\|(e^{(\Delta-\sigma)t}-I)v\|_{\elpee}\leq C t^{\alpha} \|v\|_{\cX^{\alpha}}. \lb{fracpow2}
\end{align}
\end{proposition}

\section{The $L^p-L^q$ properties of auxiliary nonlinear mappings }\label{app:LpLq}
In this section we establish $L^p-L^q$ properties of nonlinear mappings in \eqref{phi}, \eqref{psi}, \eqref{geq}. 
\begin{hypothesis}\lb{hypA}
	Let $F: \bbR^2\rightarrow \bbR$, $p\geq 1$, $R>0$. Assume that  $F\in C^{1,1}(\bbR^2)$ and that the following polynomial growth bounds hold
	\begin{align}\lb{pol1}
		&{ \sum_{|\alpha|\leq 1}}|D^{\alpha} F(s,t)| \leq C(1+|s|^{\mu}+|t|^{\nu}),\ (s,t)\in\bbR^2,
	\end{align}
	for some $C>0$, $\mu, \nu\geq 0$ with $\gamma:=\mu+\nu\geq 1$. Assume that $p\geq 2 \gamma$. 
\end{hypothesis} 

\begin{proposition}\lb{lem3.1}
	Assume Hypothesis \ref{hypA}. For $U=(u_1, u_2)\in(\elpee)^2$ and let us  define 
	\begin{equation}\lb{cf}
		\cF(U)(x):=F(u_1(x), u_2(x)),\ x\in\Gamma.
	\end{equation} 
	Then the following assertions hold.
	\begin{enumerate}
		\item\lb{item1} $\cF$ maps $\elpee$ to $\elpgamma$. 
		\item\lb{item2} There exists $C>0$ such that for $\|U\|_{(\elpee)^2}\leq R$ one has
		\begin{equation}\lb{fbdd}
			\|\cF(U)\|_{\elpgamma}\leq C\left(1+\|U\|_{(\elpee)^2}\right),
		\end{equation}
		that is $\cF:\elpee\rightarrow\elpgamma$ is a bounded (but in general non-linear) mapping.
		\item \lb{item3} There exists $C>0$ such that for $\|U\|_{(\elpee)^2}\leq R$, $\|V\|_{(\elpee)^2}\leq R$ one has 
		\begin{equation}\lb{flip}
			\|\cF(U)-\cF(V)\|_{\elpgamma}\leq C\|U-V\|_{(\elpee)^2},
		\end{equation}
		that is, $\cF:\elpee\rightarrow\elpgamma$  is locally Lipschitz continuous. 
	\end{enumerate}
\end{proposition}
\begin{proof}
	(\ref{item1})\&(\ref{item2}). Using \eqref{pol1} one infers
	\begin{align}
		\int_{\Gamma}\left|\cF(U)\right|^{\frac{p}{\gamma}}dx&\leq 3^{\frac{p}{\gamma}}C(F)\int_{\Gamma}\big(1+|u_1|^{\frac{\mu p}{\gamma}}+|u_1|^{\frac{\nu p}{\gamma}}\big)dx\nonumber\\
		&\leq C(F, \Gamma,p)(1+\|u_1\|_{\elpee}^p+\|u_2\|_{\elpee}^p).
	\end{align}
	
	(\ref{item3}) One has
	\begin{align}
		\lb{triin}
		\int_{\Gamma}\left|\cF(U)-\cF(V)\right|^{\frac{p}{\gamma}}dx\leq& \ 2^{{\frac{p}{\gamma}}}\int_{\Gamma}\left|F(u_1,u_2)-F(v_1,u_2)\right|^{\frac{p}{\gamma}}dx\nonumber\\
		&+2^{{\frac{p}{\gamma}}}\int_{\Gamma}\left|F(v_1,u_2)-F(v_1,v_2)\right|^{\frac{p}{\gamma}}dx.
	\end{align}
	Let us estimate the second term and note that the first term can be estimated analogously. We have 
	\begin{align}
		\lb{triin1}
		&\int_{\Gamma}\left|F(v_1(x),u_2(x))-F(v_1(x),v_2(x))\right|^{\frac{p}{\gamma}}dx=\int_{\Gamma}\left|\int_{v_2(x)}^{u_2(x)}\partial_2F(v_1(x),\xi)d\xi \right|^{\frac{p}{\gamma}} dx\nonumber\\
		&\quad=\int_{\Gamma}|u_2(x)-v_2(x))|^{\frac{p}{\gamma}}\left|\int_0^1\partial_2F(v_1(x),v_2(x)-\rho(v_2(x)-u_2(x)))d\rho \right|^{\frac{p}{\gamma}} dx.
	\end{align}
	Using H\"older's inequality with exponents $\gamma$, $\gamma(\gamma-1)^{-1}$ we obtain 
	\begin{align}
		\lb{triin2}
		&\left\|F(v_1,u_2)-F(v_1,v_2)\right\|_{\elpgamma}^{\frac{p}{\gamma}}dx\leq \int_{\Gamma}\left|\int_{v_2(x)}^{u_2(x)}\partial_2F(v_1(x),\xi)d\xi \right|^{\frac{p}{\gamma}} dx\nonumber\\
		&\quad=\|u_2(x)-v_2(x)\|_{\elpee}^{\frac{p}{\gamma}}\left\|\left|\int_0^1\partial_2F(v_1(x),v_2(x)-\rho(v_2(x)-u_2(x)))d\rho\right|^{\frac{p}{\gamma}} \right\|_{L^{\frac{\gamma}{\gamma-1}}(\Gamma)}\
	\end{align}
	Let us estimate the second term in the above product
	\begin{align}
		\begin{split}\lb{triin3}
			&\left\|\left|\int_0^1 \partial_2F(v_1,v_2-\rho(v_2-u_2))d\rho  \right|^{p/\gamma}\right\|_{L^{\frac \gamma{\gamma-1}}(\Gamma)}=\left\|\int_0^1 \partial_2F(v_1,v_2-\rho(v_2-u_2))d\rho\right\|_{L^{\frac p{\gamma-1}}(\Gamma)}^{p/\gamma}\\
			&\leq\left(\int_0^1 \|\partial_2F(v_1,v_2-\rho(v_2-u_2))\| _{L^{\frac p{\gamma-1}}(\Gamma)} d\rho \right)^{p/\gamma}\\
			&\leq C(F, \Gamma, p)\left(\int_0^1\left( \int_{\Gamma} 1+|v_1|^{\frac{\hatt \mu p}{\gamma-1}} +|v_2-\rho(v_2-u_2)|^{\frac{\hatt \nu p}{\gamma-1}}  dx\right)^{\frac{\gamma-1}{p}} d\rho \right)^{p/\gamma}\\
			&\leq C(F, \Gamma, p)\left(\int_0^1 \left(|\Gamma|+ \|v_1 \|^p_{\elpee}+ \|v_2-\rho(v_2-u_2) \|^p_{\elpee}\right)^{\frac{\gamma-1}{p}}d\rho \right)^{p/\gamma}\leq C(F, \Gamma, p, R),
		\end{split}
	\end{align}
	Combining \eqref{triin}, \eqref{triin1}, \eqref{triin2}, \eqref{triin3} we arrive at \eqref{flip}. 
\end{proof}

\begin{proposition}\lb{kprop3.1new}
	Assume Hypothesis \ref{hypA}, let \sel{$T\in(0, \infty)$}, $q \geq r\geq 2\gamma$, $\sigma>0$, $\delta\in(0, \sigma)$, fix $\phi\in L^q(\Gamma)$, $\psi\in L^r(\Gamma)$ and $\rho>\max\{\|\phi\|_{L^q(\Gamma)}, \|\psi\|_{L^r(\Gamma)}\}$.  For $u\in L^q(\Gamma)$ define a mapping
	\begin{equation}
		\cK (u)(x):=F\left(u(x), \phi(x)\right) \psi(x), x\in\Gamma. 
	\end{equation}
	Then $ \cK$ maps $L^q(\Gamma)$ into $L^{\frac{r}{2\gamma}}(\Gamma)$, $e^{(\Delta-\sigma I)t}\cK$ maps $L^q(\Gamma)$ into $L^q(\Gamma)$ and both mappings are bounded and locally Lipschitz continuous. That is, for $R>0$ there exists a constant $C>0$, \sel{which in general depends on $T$}, such that for ${ u}, v\in \elpee$ with $\|u\|_{L^q(\Gamma)}\leq R$, $\|v\|_{L^q(\Gamma)}\leq R$ one has
	\begin{align}
		\|\cK(u)\|_{{L^{\frac{r}{2\gamma}}(\Gamma)}}&\leq C{ (1+\|u\|_{L^q(\Gamma)})},\label{k322newn}\\
		\|\cK(u)-\cK(v)\|_{{L^{\frac{r}{2\gamma}}(\Gamma)}}&\leq C\|u-v\|_{L^q(\Gamma)},\label{k322new}\\
		\|e^{(\Delta-\sigma I)t}\cK(u)-e^{(\Delta-\sigma I)t}\cK(v)\|_{L^q(\Gamma)}&\leq Ce^{-\delta t}t^{-\frac12\left(\frac{2\gamma}{r}-\frac1q\right)}\|u-v\|_{{L^q(\Gamma)}}, \sel{ t\in(0,T]}.\label{kehu}
	\end{align}
	Furthermore, the mapping $e^{(\Delta-\sigma I)t}\partial_x$, originally defined on $C_0^{\infty}(\Gamma)$, can be extended to a bounded linear operator in $\cB(  L^{\frac{r}{2\gamma}}(\Gamma), L^q(\Gamma))$ so that $e^{(\Delta-\sigma I)t}\partial_{x}\cK$ maps $L^q(\Gamma)$ into $L^q(\Gamma)$ and, for $u, v$ as above, one has
	\begin{equation}\label{klipcontedh}
		\|e^{(\Delta-\sigma I)t}\partial_x\cK(u)-e^{(\Delta-\sigma I)t}\partial_x\cK(v)\|_{L^q(\Gamma)}\leq Ce^{-\delta t}t^{-\frac12\left(\frac{2\gamma}{r}-\frac1q+1\right)}\|u-v\|_{{L^q(\Gamma)}}, \sel{ t\in(0,T]}.
	\end{equation}
\sel{Furthermore,  if $\sigma>\sigma_0$, where $\sigma_0>0$ is as in Corollary \ref{cor2.1}, then there exists a $T$-independent constant $C>0$ and $0<\delta\ll\sigma$ such that \eqref{kehu}, \eqref{klipcontedh} hold with $ t\in(0,T]$ replaced by $ t\in(0,\infty)$. 
}
\end{proposition}
\begin{proof}
	To show \eqref{k322new} (hence, \eqref{k322newn}) we note
	\begin{align}
		&\|\cK(u)- \cK(v)\|_{{L^{\frac{r}{2\gamma}}}(\Gamma)}\leq  \|\big(F(u, \phi)-F(v, \phi)\big)\psi\|_{{L^{\frac{r}{2\gamma}}(\Gamma)}}\nonumber\\
		&\leq \|F(u, \phi)-F(v, \phi)\|_{{L^{\frac{r}{\gamma}}}(\Gamma)}\|\psi\|_{{L^{\frac{r}{\gamma}}}(\Gamma)}\leq C \|u-v\|_{L^q(\Gamma)},
	\end{align}
	where we used Proposition \ref{lem3.1} item \eqref{item3}.  Then \eqref{kehu} and \eqref{klipcontedh} follow from \eqref{k322new} combined with \eqref{new-norm-estimate1}  and \eqref{tgradlplq} correspondingly.\sel{The last statement follows from Corollary \ref{cor2.1}.}
	
\end{proof}
\begin{proposition}\lb{cor3.1} Assume the setting of Proposition \ref{lem3.1}, fix \sel{$T\in(0,\infty)$}, $\sigma>0$ and $\delta\in(0, \sigma)$. 
	\begin{enumerate}
		\item\lb{item1cor31}   $e^{(\Delta-\sigma I)t}\cF(U)$ maps $(\elpee)^2$ into $\wpee$ boundedly, in fact, this mapping is locally Lipschitz continuous. That is, for $R>0$ there exists $C>0$,  which in general depends on $T$, such that for $\|U\|_{(\elpee)^2}\leq R$, $\|V\|_{(\elpee)^2}\leq R$, one has  
		\begin{align}
			&\|e^{(\Delta-\sigma I)t}\cF(U)\|_{\wpee}\leq Ce^{-\delta t}\max\{ t^{-\frac{\gamma-1}{2p}}, t^{-\frac{p+\gamma-1}{2p}}\}\left(1+\|U\|_{(\elpee)^2}\right), \sel{ t\in(0,T]}, \lb{efbdd}
		\end{align}
		\begin{align}
			\begin{split}\lb{eflip}
				\|e^{(\Delta-\sigma I)t}\cF(U)&-e^{(\Delta-\sigma I)t}\cF(V)\|_{\wpee}\\
				&\leq Ce^{-\delta t}\max( t^{-\frac{\gamma-1}{2p}}, t^{-\frac{p+\gamma-1}{2p}})\|U-V\|_{(\elpee)^2}, \sel{ t\in(0,T]}. 
			\end{split}
		\end{align}
		\item\lb{item2newcor31} For $\beta>0$, $h> 0$ one has 
		\begin{align}
			&\hspace{1cm}\|e^{(\Delta-\sigma I)t}\cF(U)\|_{\cX^{\beta}_p}\leq Ce^{-\delta t} t^{-\beta-\frac{\gamma-1}{2p}} \left(1+\|U\|_{(\elpee)^2}\right),\sel{ t\in(0,T]}, \lb{3.19}\\
			&\|(e^{(\Delta-\sigma I)h}-I)e^{(\Delta-\sigma I)t}\cF(U)\|_{\cX^{\beta}_p}\leq Ch^{\beta}e^{-\delta t} t^{-2\beta-\frac{\gamma-1}{2p}} \left(1+\|U\|_{(\elpee)^2}\right), \sel{ t\in(0,T]}.\label{eminusone}
		\end{align}
		\item\lb{item2cor31} The mapping $e^{(\Delta-\sigma I)t}\partial_x$ originally defined on $C_0^{\infty}(\Gamma)$ can be extended to a bounded linear operator in $\cB(\elpgamma, \elpee)$ so that $e^{(\Delta-\sigma I)t}\partial_{x}\cF(U)$ maps $(\elpee)^2$ into $\elpee$ boundedly, this mapping is locally Lipschitz continuous, that is, for $R\geq 0$ there exists $C>0$, \sel{which in general depends on T} such that for $\|U\|_{(\elpee)^2}\leq R$, $\|V\|_{(\elpee)^2}\leq R$, one has   
		\begin{align}
			&\hspace{1.5cm}\|e^{(\Delta-\sigma I)t}\partial_x\cF(U)\|_{\elpee}\leq Ce^{-\delta t}t^{-\frac{p+\gamma-1}{2p}}\left(1+\|U\|_{(\elpee)^2}\right), \sel{ t\in(0,T]}, \lb{efdbdd}\\
			&\|e^{(\Delta-\sigma I)t}\partial_x\cF(U)-e^{(\Delta-\sigma I)t}\partial_x\cF(V)\|_{\elpee}\leq Ce^{-\delta t}t^{-\frac{p+\gamma-1}{2p}}\|U-V\|_{(\elpee)^2}, \sel{ t\in(0,T]}. \lb{efdlip}
		\end{align}
	\end{enumerate}
Furthermore,  if $\sigma>\sigma_0$, where $\sigma_0>0$ is as in Corollary \ref{cor2.1}, then there exists a $T$-independent constant $C>0$ and  $0<\delta\ll\sigma$ such that \eqref{efbdd}--\eqref{efdlip} hold with $ t\in(0,T]$ replaced by $ t\in(0,\infty)$. 

\end{proposition}
\begin{proof}
	\eqref{item1cor31} The inequalities \eqref{efbdd}, \eqref{eflip} follow from Theorem \ref{lplq estimate} combined with \eqref{fbdd}, \eqref{flip} correspondingly.
	
	\eqref{item2newcor31} The inequality \eqref{3.19}  follows  \eqref{fbdd} and \eqref{fracpow}. To prove \eqref{eminusone} we write
	\begin{align}
		\|(e^{(\Delta-\sigma I)h}-I)e^{(\Delta-\sigma I)t}\cF(U)\|_{\cX^{\beta}_p}&\leq Ch^{\beta}e^{-\delta t} \|e^{(\Delta-\sigma I)t}\cF(U)\|_{\cX^{2\beta}_p}\nonumber\\
		&\leq Ch^{\beta}e^{-\delta t} t^{-2\beta-\frac{\gamma-1}{2p}} \left(1+\|U\|_{(\elpee)^2}\right).
	\end{align}
	\eqref{item2cor31}  The inequalities \eqref{efdbdd}, \eqref{efdlip} follow from \eqref{normtdx} combined with \eqref{fbdd}, \eqref{flip} correspondingly.
	
	\sel{The last statement follows from Corollary \ref{cor2.1}.}
\end{proof}

\begin{proposition}\lb{prop2.1} Let $\sigma>0$, $p\geq 1$ then there exists $C>0$ such that
	\begin{align}
		&\|(\Delta-\sigma)^{-1}u\|_{L^{\infty}(\Gamma)}\leq C \|u\|_{\elpee}, \lb{lplinf}\\
		&\|\partial_x(\Delta-\sigma)^{-1}u\|_{L^{\infty}(\Gamma)}\leq C\|u\|_{\elpee}, \lb{lplinf2}\\
		&\|(\Delta-\sigma)^{-1}u\|_{\hatt W^{2,p}(\Gamma)}\leq C\|u\|_{\elpee},\lb{resw2p}
	\end{align}
	for all $t>0$. 
\end{proposition}	
\begin{proof}
	Let us recall, e.g. from \cite[Corollary 7 in Section 4.2]{MR1622690}, the following inequality
	\begin{equation}\lb{burenkov}
		\|f\|_{L^{\infty}(\Gamma)}\leq C\left(\|f\|_{\elpee}+\|\Delta f\|_{\elpee}\right), f\in \hatt W^{2,p}(\Gamma), C>0.
	\end{equation}
	To prove \eqref{lplinf} we denote $f:=(\Delta-\sigma)^{-1}u$, that is, $\Delta f-\sigma f=u$, $u\in \elpee$, and employ \eqref{burenkov} to obtain
	\begin{align}
		\|f\|_{L^{\infty}(\Gamma)}&\leq C\left(\|f\|_{\elpee}+\|\Delta f\|_{\elpee}\right)= C\left((1+|\sigma|)\|f\|_{\elpee}+\|u\|_{\elpee}\right)\nonumber\\
		&\leq  C\left((1+|\sigma|)\|(\Delta-\sigma)^{-1}\|_{\cB(\elpee)}\|u\|_{\elpee}+\|u\|_{\elpee}\right),
	\end{align}
	where we implicitly used $p-$independence of the spectrum and resolvent of $\cA_p$ obtained in \cite[Theorem 3, Section 4.2]{BFM}. The third inequality follows from  \eqref{lplinf2} and 
	\begin{equation}
		\|\Delta(\Delta-\sigma)^{-1}u\|_{\elpee}\leq C\|u\|_{\elpee}.
	\end{equation}
\end{proof}

\begin{corollary}\label{cor3.2}Assume Hypothesis \ref{hyp} and suppose that $\tau=0$. Define
	\begin{equation}\lb{bigf3}
		\cF_3(u):=(\Delta-\sigma)^{-1}f(u)\text{ for } u\in \elpee.
	\end{equation}
	Then $\cF_3$ is a bounded mapping between $\elpee$ and $\wpee$, that is, there exists a constant $C>0$ such that if $\|u\|_{\elpee}\leq R$ then one has
	\begin{equation}
		\|\cF_3(u)\|_{\wpee}\leq C(1+\|u\|_{\elpee}). 
	\end{equation}
	Moreover, $\cF_3$ is locally Lipschitz continuous, that is, for $\|u\|_{\elpee}\leq R$, $\|v\|_{\elpee}\leq R$ one has
	\begin{equation}
		\|\cF_3(u)-\cF_3(v)\|_{\wpee}\leq C\|u-v\|_{\elpee}. 
	\end{equation}
\end{corollary}
\begin{proof}
	By Proposition \ref{lem3.1} with $F(u,v)=f(u)$, the mapping $u\mapsto f(u)$ from $\elpee$ to $L^{\frac{p}{\gamma}}(\Gamma)$ is bounded and locally Lipschitz. Next, by Proposition \ref{prop2.1}, $(\Delta-\sigma)^{-1}$ maps $L^{\frac{p}{\gamma}}(\Gamma)$ into $\hatt W^{1,\infty}(\Gamma)$ which is boundedly embedded into $\wpee$. The assertions follow from the fact that $\cF_3$ is a composition of these mappings. 
\end{proof}

\begin{proposition}\lb{prop3.1new}
	Assume Hypothesis \ref{hypA} and let \sel{$T\in(0,\infty)$}, $\sigma> 0$, $\delta\in(0, \sigma)$.  For $U=(u_1, u_2)\in \elpee\times\wpee$ let
	\begin{equation}
		\cH (U)(x):=F\left(u_1(x), u_2(x)\right)\partial_x u_2(x), x\in\Gamma. 
	\end{equation}
	Then $ \cH$ maps $\elpee\times\wpee$ into $L^{\frac{p}{2\gamma}}(\Gamma)$, $e^{(\Delta-\sigma I)t}\cH$ maps $\elpee\times\wpee$ into $\elpee$ and both mappings are bounded and locally Lipschitz continuous. That is, for $R>0$ there exists $C>0$,\sel{which in general depends on $T$}, such that for $U, V\in \elpee\times\wpee$ with $\|U\|_{\elpee\times\wpee}\leq R$, $\|V\|_{\elpee\times\wpee}\leq R$ one has
	\begin{equation}\lb{322new}
		\|\cH(U)-\cH(V)\|_{{L^{\frac{p}{2\gamma}}(\Gamma)}}\leq C\|U-V\|_{\elpee\times\wpee}.
	\end{equation}
	In addition, one has
	\begin{equation}
		\|e^{(\Delta-\sigma I)t}\cH(U)-e^{(\Delta-\sigma I)t}\cH(V)\|_{\elpee}\leq Ce^{-\delta t}t^{-\frac{2\gamma-1}{2p}}\|U-V\|_{_{\elpee\times\wpee}}, \sel{t\in(0,T]}.\lb{ehu}
	\end{equation}
	Furthermore, the mapping $e^{(\Delta-\sigma I)t}\partial_x$, originally defined on $C_0^{\infty}(\Gamma)$, can be extended to a bounded linear operator in $\cB(L^{\frac{p}{2\gamma}}(\Gamma), \elpee)$ so that $e^{(\Delta-\sigma I)t}\partial_{x}\cH$ maps $\elpee\times\wpee$ into $\elpee$ and, for $U, V$ as above, one has
	\begin{equation}\lb{lipcontedh}
		\|e^{(\Delta-\sigma I)t}\partial_x\cH(U)-e^{(\Delta-\sigma I)t}\partial_x\cH(V)\|_{\elpee}\leq Ce^{-\delta t}t^{-\frac{p+2\gamma-1}{2p}}\|U-V\|_{{\elpee\times\wpee}}, \sel{t\in(0,T]}.
	\end{equation}
	In addition,  for $\beta\geq 0$, $h>0$ one has $e^{(\Delta-\sigma I)t}\partial_x \cH(U)\in \cX^{\beta}_p$ and the following inequalities hold 
	\begin{align}
		&\hspace{1cm}\|e^{(\Delta-\sigma I)t}\partial_x \cH(U)\|_{\cX^{\beta}_p}\leq Ce^{-\delta t} t^{-\beta-\frac{p+2\gamma-1}{2p}}\|U\|_{{\elpee\times\wpee}}, \sel{t\in(0,T]}, \lb{3.22}\\
		&\|(e^{(\Delta-\sigma I)h}-I)e^{(\Delta-\sigma I)t}\partial_x \cH(U)\|_{\cX^{\beta}_p}\leq Ch^{\beta}e^{-\delta t} t^{-2\beta-\frac{p+2\gamma-1}{2p}}\|U\|_{{\elpee\times\wpee}}, \sel{t\in(0,T]}.\lb{3.23}
	\end{align}
\sel{Furthermore,  if $\sigma>\sigma_0$, where $\sigma_0>0$ is as in Corollary \ref{cor2.1}, then there exists a $T$-independent constant $C>0$ and  $0<\delta\ll\sigma$ such that \eqref{ehu}--\eqref{3.23} hold with $ t\in(0,T]$ replaced by $ t\in(0,\infty)$. 
}
\end{proposition}
\begin{proof}
	For $\cF$ as in \eqref{cf} one has
	\begin{align}
		&\|\cH(U)- \cH(V)\|_{{L^{\frac{p}{2\gamma}}}(\Gamma)}\leq \|\cF\big(U\big)(\partial_xu_2-\partial_xv_2)\|_{L^{\frac{p}{2\gamma}}(\Gamma)}+ \|\big(\cF(U)-\cF(V)\big)\partial_xv_2\|_{{L^{\frac{p}{2\gamma}}(\Gamma)}}.
	\end{align}
	Using the Cauchy--Schwartz inequality we obtain
	\begin{align}
		\lb{fpar1}
		\|\big(\cF(U)-\cF(V)\big)\partial_xv_2\|_{{L^{\frac{p}{2\gamma}}}(\Gamma)}^{\frac p{2\gamma}}&\leq \|\big(\cF(U)-\cF(V)\big)\|_{\elpgamma}^{\frac{p}{2\gamma}}\|\partial_xv_2\|^{\frac{p}{2\gamma}}_{\elpgamma}\nonumber\\
		&\leq C\|U-V\|_{(\elpee)^2}^{\frac{p}{2\gamma}}\|v_2\|^{\frac{p}{2\gamma}}_{\wpee},
	\end{align}
	where we used Proposition \ref{lem3.1} item \eqref{item3} and $p/\gamma\leq p$. By a similar argument one obtains
	\begin{align}
		\lb{fpar2}
		\| \cF\big(U\big)(\partial_xu_2-\partial_xv_2)\|_{{L^{\frac{p}{2\gamma}}}(\Gamma)}^{\frac p{2\gamma}}&\leq \|\cF(U)\|_{\elpgamma}^{\frac{p}{2\gamma}}\|\partial_xu_2-\partial_xv_2\|^{\frac{p}{2\gamma}}_{\elpgamma}\nonumber\\
		&\leq C\left(1+\|U\|_{(\elpee)^2}^{\frac{p}{2\gamma}}\right)\|u_2-v_2\|^{\frac{p}{2\gamma}}_{\wpee},
	\end{align}
	where we used Proposition \ref{lem3.1} item \eqref{item2} and $p/\gamma\leq p$.
	
	The inequality \eqref{ehu} follows from \eqref{322new} and \eqref{new-norm-estimate1}.

	The inequality \eqref{lipcontedh} follows from \eqref{322new} and \eqref{normtdx}.  
	
	To prove \eqref{3.22} we write 
	\begin{align}
		&\|e^{(\Delta-\sigma I)t}\partial_x \cH(U)\|_{\cX^{\beta}_p}= \|e^{\frac12(\Delta-\sigma I)t}e^{\frac12(\Delta-\sigma I)t}\partial_x \cH(U)\|_{\cX^{\beta}_p}\nonumber\\
		&\quad \underset{\eqref{fracpow}}{\leq }Ce^{-\frac{\delta t}2} t^{-\beta}\|e^{\frac12(\Delta-\sigma I)t}\partial_x \cH(U)\|_{\elpee}\underset{\eqref{ehu}}{\leq}C e^{-\delta t}t^{-\beta-\frac{p+2\gamma-1}{2p}}\|U\|_{{\elpee\times\wpee}}. 
	\end{align}
	The inequality \eqref{3.23} is proved analogously by using \eqref{fracpow2}. 
	
	\sel{The last statement follows from Corollary \ref{cor2.1}.}
\end{proof}


\section{General theory of semilinear parabolic equations on graphs}\lb{appA}
\label{nonlinear-equations}
In this section, we  discuss local existence of classical solutions for
\begin{align}
&u_t=u_{xx}+f(t,x,u, u_x),\quad x\in \Gamma, 	\label{new-nonlinear-eq1}\\
& \sum\limits_{\vartheta \sim e} \partial_{\nu}u_e(\vartheta)=0,\ u_e(\vartheta)=u_{e'}(\vartheta), e\sim \vartheta, e'\sim\vartheta,\\
&u{ (t_0,\cdot)}=u_0(\cdot)\in \cX^{\alpha}_p, p\in[1,\infty).
\end{align}
To that end, we first focus on local existence of strong solutions for
\begin{equation}
\label{new-nonlinear-eq2}
u_t=\Delta u+F(t,u), u(t_0)=u_0\in \cX^{\alpha}_p, 
\end{equation}
where $\Delta=\cA_p$, cf. \eqref{apnum}, and  { $\alpha\ge 0$}.

\begin{definition}{\cite[Definition 3.3.1]{MR610244}}
	\label{strong-solution-def}
	{\it A {\rm  strong solution}  of \eqref{new-nonlinear-eq2} on $[t_0,t_1)$ with initial condition $u(t_0)=u_0\in \cX_p^\alpha$
	is a continuous function  $u: [t_0, t_1) \to L^p(\Gamma)$ such that $u(t_0)=u_0$,  $u(t) \in \dom(\cA_p)$ for $t\in (t_0,t_1)$,  $\frac{d u}{d t}$ exists for $t\in (t_0,t_1)$, $(t_0,t_1)\ni t \to F(t,u(t))\in L^p(\Gamma)$ is locally H\"older continuous, and $\int^{t_0+\sigma}_{t_0} \|F(t,u(t))\|_{L^p(\Gamma)} dt < \infty$ for some $\sigma>0,$ and the differential equation $u_t+A_p u=F(t,u) $ is satisfied on $(t_0,t_1)$ in { $L^p(\Gamma)$.}}
\end{definition}

In the following theorem we adopt classical results regarding strong solutions of \eqref{new-nonlinear-eq2} to the setting of metric graphs.

\begin{theorem}
	\label{strong-solu-thm}
	Let $p\in [1,\infty)$ and $\alpha\in { [}0,1)$. Assume that
	$F:[0,\infty)\times X_p^\alpha \to L^p(\Gamma)$ is locally H\"older continuous in $t$ with exponent $\nu\in (0,1)$  and locally  Lipschitz continuous in $u$, that is, for every $(t, u)\in [0,\infty)\times \cX^{\alpha}_p$ there exists a neighborhood  $V\subset  [0,\infty)\times \cX^{\alpha}_p$ containing $(t,u)$ and a constant $L=L(V)$ such that 
	\begin{equation}
\label{holder-lipschitz-cond}
	\|F(t_1, u_1)-F(t_2, u_2)\|_{\elpee}\leq L(|t_1-t_2|^{\nu}+\|u_1-u_2\|_{\cX^{\alpha}_p}), (t_k, u_k)\in V, k=1,2.
	\end{equation} Then the following assertions hold.
	
	\begin{enumerate}
		\item\lb{thma1item1}  For any $t_0\in[0,\infty)$ and $u_0\in \cX_p^\alpha$,   there exists $T_{\max}=T_{\max}(t_0,u_0)>0$  such that { \eqref{new-nonlinear-eq2}}  has a unique strong solution
		$u(t;t_0,u_0)$ on $[t_0,t_0+T_{\max})$ with initial value $u(t_0;t_0,u_0)=u_0.$  Moreover, 
		\begin{equation}
		\label{classical-solu-eq1-1}
		u(\cdot;t_0,u_0)\in C([t_0,t_0+T_{\max}),\cX_p^\alpha),
		\end{equation}
		and
		if $T_{\max}< \infty,$ then
		\begin{equation}
		\label{classical-solu-eq1-2}
		\limsup_{t \nearrow T_{\max}}  \left\|u(t+t_0;t_0,u_0) \right\|_{{X_p^{\alpha}}}  =\infty.
		\end{equation}

		\item\lb{thma1item2} For any $t_0\in[0,\infty)$ and $u_0\in \cX_p^\alpha$, if  $u(t)$ is a strong  solution of  {\eqref{new-nonlinear-eq2}}  on $[t_0,t_1)$
		with initial condition $u(t_0)=u_0$,  then  $u$ satisfy the following  integral equation 
		\begin{equation}
		\label{mild-solution}
		u(t)=e^{\Delta( t-t_0)}u_0 +\int_{t_0}^te^{\Delta( t-s)} F(s,u(s))ds.
		\end{equation}
		Conversely, if $t\mapsto u(t)$ is continuous as a mapping from $(t_0,t_1)$ into $ X^\alpha,$ one has $\int^{t_0+\rho}_{t_0} \|F(t,u(t))\|_{\elpee}dt < \infty$ for some $\rho>0,$ and if the integral equation \eqref{mild-solution} holds for $t_0< t<t_1,$ then $u(t)$ is a strong solution of the differential equation  \eqref{new-nonlinear-eq2}  on $(t_0,t_1).$ Furthermore,
		\begin{equation}
		\label{classical-solu-eq1-3}
		u\in C^\delta((t_0,T_{\max}), \cX_p^\alpha), \delta\in(0, 1-\alpha).
		\end{equation}

		\item For any $t_0\in\mathbb{R}$ and $u_0\in X_p^\alpha$, 
		\begin{equation}
		\label{classical-solu-eq1-4}
		u(\cdot;t_0,u_0)\in C^1((t_0,t_0+T_{\max}(t_0,u_0)), \cX_p^\beta), \beta\in (0, \nu).
		\end{equation}
		
		\item  If $u_n,u_0\in \cX_p^\alpha$ and $\|u_n-u_0\|_{\cX_p^\alpha}\to 0$ as $n\to\infty$, then for any 	$t_0\in\mathbb{R}$, one has
		\begin{equation}
		\liminf_{n\to\infty} T_{\max}(t_0,u_n)\ge T_{\max}(t_0,u_0),
		\end{equation}
		and 
		\begin{equation}
		\lim_{n\to\infty}\|u(t;t_0,u_n)-u(t;t_0,u_0)\|_{\cX_p^\alpha}=0,
		\end{equation}
		uniformly in compact intervals of $[t_0,t_0+T_{\max}(t_0,u_0))$.
	\end{enumerate}
\end{theorem}

\begin{proof}
	(1) It follows from \cite[Theorem 3.3.3 ]{MR610244} and  \cite[Theorem 3.3.4 ]{MR610244} together with the fact that $e^{t\Delta}$ is analytic in $\elpee$.

(2) The first statement in part (2) follows from \cite[Lemma 3.3.2]{MR610244}, the fact that $u \in C^\delta((t_0,t_1);X^\alpha)$ is obtained in the course of the proof of \cite[Lemma 3.3.2]{MR610244}.
	
(3)	The proof of part (3) stems from that  of \cite[Theorem 3.5.2]{MR610244}. For completeness, we provide the details. It suffices to prove
	$$
	u(\cdot;t_0,u_0)\in C^1((\tau,\tau+h),\cX_p^\beta),  t_0<\tau<\tau+h<t_0+T_{\max}(t_0,u_0),\,\,0<\beta<\nu,
	$$
	for all $\tau\in(t_0, t_0+T_{\max}(t_0, u_0))$ and some $h=h(\tau)>0$. 
{ For any fixed $\tau\in (t_0,  T_{\max}(t_0,u_0))$},
	let $u_1=u(\tau;t_0,u_0)$. Then one has
	$
	u(t;t_0,u_0)=u(t;\tau,u_1),\  \tau<t<T_{\max}(t_0,u_0)
	$	and { for any $\sigma>0$},
	\begin{equation}
	u(t;t_0,u_0)=u(t;\tau,u_1)={ e^{(\Delta-\sigma)(t-\tau)}u_1+\int_{\tau}^t e^{(\Delta-\sigma)(t-s)}g(s)ds,}
	\end{equation}
	where $g(t):={ u(t;\tau,u_1)}+F(t,u(t;\tau,u_1))$. By  \cite[Lemma 3.5.1]{MR610244},  in order to prove that the mapping
	\begin{equation}
	(\tau, \tau+h)\ni t\mapsto G(t):=\int_{\tau}^t e^{(\Delta-\sigma)(t-s)}g(s)ds 
	\end{equation} 
	is continuously differentiable into $\cX^{\beta}_p$ for some $h=h(\tau)>0$, hence, $u(\cdot , t_0, u_0)\in C^1((\tau,\tau+h),\cX_p^\beta)$ it suffices to show that 
	\begin{equation}\lb{keq}
	\|g(s)-g(t)\|_{L^p(\Gamma)}\le K(t) (s-t)^{\nu},\   \tau <t<s<\tau+h,
	\end{equation}
	for some continuous $K(\cdot)$ and $\int_{\tau}^{\tau+h}K(t)dt<\infty$. Our objective, therefore, is to derive \eqref{keq}.
	To that end,  { by \eqref{holder-lipschitz-cond}, there exist  a neighborhood  $V\subset  [0,\infty)\times \cX^{\alpha}_p$ containing $(\tau,u_1)$,  and positive constants $L>0$,  $h>0$ such that }
	\begin{align}
	\label{classical-solu-eq2}
	\|g(t)-g(s)\|_{L^p(\Gamma)}& \le L\big(|t-s|^\nu +\|u(t;\tau,u_0)-u(s;\tau,u_0)\|_{X_p^\alpha}\big),
	\end{align}
	for  $s,t\in(\tau, \tau+h)$ and $t<s$.
 Note that
	\begin{align*}
	u(s;\tau,u_1)-u(t;\tau,u_1)=&\big(e^{(\Delta-\sigma) (s-t)}-I\big)e^{(\Delta-\sigma)(t-\tau)}u_1\\
	&+\int_{\tau}^t e^{(\Delta-\sigma)(t-\kappa)}\big(g(\kappa+s-t)-g(\kappa)\big)d\kappa +\int_{\tau}^{\tau+s-t} e^{(\sigma-\Delta)(s-\kappa)}g(\kappa)d\kappa.
	\end{align*}
	Then there exist constants $C_1, C_2>0$ such that
	\begin{align*}
	&\|u(s;\tau,u_1)-u(t;\tau,u_1)\|_{\cX_p^\alpha}\\
	&\le C_1\big( (s-t)\cdot  \|(\sigma-\Delta)^{1+\alpha}  e^{(\Delta-\sigma)(t-\tau)} {  u_1}\|_{L^p(\Gamma)}
	+\int_{\tau}^t (t-\kappa)^{-\alpha}
	\|g(\kappa+s-t)-g(\kappa)\|_{L^p(\Gamma)}d\kappa\\
	&\qquad +\int_{\tau}^{\tau+s-t}\|(\sigma-\Delta)^\alpha e^{\Delta(s-\kappa)}
	g(\kappa)\|_{L^p(\Gamma)}d\kappa\big)\\
	&\le C_2\big ((s-t) (t-\tau)^{-\alpha}\|{ (\sigma-\Delta) u_1}\|_{\elpee}+ \int_{\tau}^t (t-\kappa)^{-\alpha}\|g(\kappa+s-t)-g(\kappa)\|_{L^p(\Gamma)}d\kappa\\
	&\qquad+(s-t)(t-\tau)^{-\alpha}
	\big).
	\end{align*}
	This together with \eqref{classical-solu-eq2} imply
	\begin{align}
	\label{classical-solu-eq3}
	\|g(s)-g(t)\|_{L^p(\Gamma)}\le &L (s-t)^\nu +  C_2\big ((s-t) (t-\tau)^{-\alpha}\|{ (\sigma-\Delta)u_1}\|_{\elpee}+ (s-t) (t-\tau)^{-\alpha}\big)\nonumber\\
	&+C_2\int_{\tau}^t (t-\kappa)^{-\alpha}\|g(k+s-t)-g(\kappa)\|_{L^p(\Gamma)}d\kappa.
	\end{align}
	Then Gronwall's inequality (see \cite[Section 1.2.1]{MR610244} yields $L_1>0$ such that
	\begin{align}
	\|g(s)-g(t)\|_{\elpee}\le  (s-t)^{\nu} K(t),
\end{align}
where
$$
	K(t):=L_1 \Big( 1+ { h}^{1-\nu} (t-\tau)^{-\alpha}\|{ (\sigma-\Delta)u_1}\|_{\elpee}+  { h}^{1-\nu}(t-\tau)^{-\alpha}\Big),
$$
	as required in \eqref{keq}. 

(4)  Follows from \cite[Theorem 3.4.1]{MR610244}.
\end{proof}

Equation \eqref{new-nonlinear-eq1} is related to \eqref{new-nonlinear-eq2} via the following elementary proposition. 
\begin{lemma}
	\label{classical-solu-lm1}  
	Let $f:[0,\infty)\times\Gamma\times \bbR\times\bbR\rightarrow\bbR$. Assume that there exists $\nu\in(0,1)$, such that for every triple $(t,u,v)\in [0,\infty)\times \bbR\times \bbR$ there exists an open neighborhood $U\subset [0,\infty)\times \bbR\times \bbR$ and a constant $L=L(U)$ such that for all $(t_k, u_k, v_k)\in U$, $k=1,2$ one has
	\begin{equation}\lb{lipeq}
	|f(t_1, x, u_1, v_1)-f(t_2, x, u_2, v_2)|\leq L(|t_1-t_2|^{\nu}+|u_1-u_2|+|v_1-v_2|),
	\end{equation}
	that is, $f=f(t,x,u,v)$ is locally H\"older continuous in $t$ with exponent $\nu\in (0,1)$, locally Lipschitz continuous in $u$ and $v$, uniformly for $x\in\Gamma$. Assume in addition that $f(t, \cdot, u,v)\in \elpee$. Let $p\in [1,\infty)$ and $\alpha\in(0,1)$ be such that $2\alpha-p^{-1}>1$. For $u\in\cX^{\alpha}_p$ let $F(t,u)=f(t,\cdot ,u,u_x)$ then the mapping $F:[0,\infty)\times \cX_p^\alpha \to L^p(\Gamma)$ is locally H\"older continuous in $t$ { with exponent $\nu$} and Lipschitz continuous in $u$.
\end{lemma}

\begin{proof}
	Fix $(t_0, u_0)\in [0,\infty)\times \cX^{\alpha}_p$. Then due to $2\alpha-\frac{1}{p}>1$, one has $	X_p^\alpha \hookrightarrow \hatt C^1(\hatt \Gamma)$, see \eqref{sobolev-embedding-eq5}. Therefore we may choose an open neighborhood $V\subset[0,\infty)\times \cX^{\alpha}_p$ containing $(t_0, u_0)$ such that for $(t_k, u_k)\in V$, $k=1,2$, we have $(t_k, u_k(x), \partial_xu_k(x))\in U$ for all $x\in\Gamma$. Then by \eqref{lipeq} we have
	\begin{align}\no
	&|f(t,x,u(x),u_x(x))-f(s,x,v(x),v_x(x))|\leq\\
	&\quad \leq L  \big(|s-t|^\nu +|u(x)-v(x)|+|u_x(x)-v_x(x)|,\ x\in\Gamma.
	\end{align}
	Hence, $f(t,\cdot,u(\cdot),u_x(\cdot))-f(s,\cdot,v(\cdot),v_x(\cdot))\in L^p(\Gamma)$ and
	\begin{align*}
	&\|f(t,\cdot,u(\cdot),u_x(\cdot))-f(s,\cdot,v(\cdot),v_x(\cdot))\|_{L^p(\Gamma)}\\
	&\le L \big(|t-s|^\nu \sum_{e\in\mathcal{E}}|e|^{\frac{1}{p}}
	+\|u(\cdot)-v(\cdot)\|_{L^p(\Gamma)}+\|u_x(\cdot)-v_x(\cdot)\|_{L^p(\Gamma)}\big)\\
	&\le \tilde L \big(|t-s|^\nu+\|u-v\|_{X_p^\alpha} \big),
	\end{align*}
	as required.
\end{proof}

Applying the above theorem, we have

\begin{theorem}
	\label{general-existence-thm}
	Assume that $f=f(t,x,u,v)$ is as in Lemma \ref{classical-solu-lm1} and suppose in addition that  $f(t, \cdot, u, v)\in \hatt C(\overline{\Gamma})$. Assume that $0<\alpha<1$ and  $1<p<\infty$ are such that $2\alpha-p^{-1}>1$ and $2\nu-p^{-1}>0$. Then the following assertions hold. 
	
\medskip

\noindent (1)  $($Existence of classical solutions$)$
	For given $t_0\in\bbR$ and $u_0\in \cX_p^\alpha$, \eqref{new-nonlinear-eq1}
	has a unique classical solution $u(t,x;t_0,u_0)$ on some interval $[t_0,t_0+T)$  with initial condition
	$u(t_0,x;t_0,u_0)=u_0(x)$ in $\cX_p^\alpha$, and
	\begin{equation}
	\label{local-1-eq1}
	u(\cdot,\cdot;t_0,u_0)\in C([t_0,t_0+T_{\max}),\cX_p^\alpha)\cap C^\delta((t_0,t_0+T_{\max}),\cX_p^\alpha)\cap C^1((t_0,t_0+T_{\max},\cX_p^\beta)
	\end{equation}
	for all  $\delta\in (0,1-\alpha)$ and $\beta\in (0,\nu)$.
	Moreover if $T_{\max}< \infty,$ then
	\begin{equation}
	\label{local-1-eq2}
	\limsup_{t \rightarrow T^-_{\max}}   \left\| u(t+t_0,\cdot;t_0,u_0) \right\|_{\cX_p^\alpha}=\infty.
	\end{equation}

\noindent (2)
 $($Continuity with respect to initial conditions$)$ If $u_n,u_0\in \cX_p^\alpha$ and $\|u_n-u_0\|_{\cX_p^\alpha}\to 0$ as $n\to\infty$, then for any $t_0\in\mathbb{R}$,
	$\liminf_{n\to\infty} T_{\max}(t_0,u_n)\ge T_{\max}(t_0,u_0)$ and
	$$
	\lim_{n\to\infty}\|u(t,\cdot;t_0,u_n)-u(t,\cdot;t_0,u_0)\|_{\cX_p^\alpha}=0
	$$
	uniformly in compact intervals of $[t_0,t_0+T_{\max}(t_0,u_0))$.
	
\medskip

\noindent (3)   $($Comparison principle$)$
	If $u_1, u_2\in\hatt C(\overline{\Gamma})$ and $u_1(x)\le u_2(x)$, $x\in\Gamma$ then $u(t,x;t_0,u_1)\le u(t,x;t_0,u_2)$, $x\in \Gamma$ for $t>t_0$ such that both $u(t,x;t_0,u_1)$ and $u(t,x;t_0,u_2)$ exist.

\end{theorem}

\begin{proof} [Proof of Theorem \ref{general-existence-thm}] 
	(1) First, let us define a mapping $F:[0,\infty)\times \cX^{\alpha}_p\rightarrow \elpee$ as follows 
	\begin{equation}
	F(t,u)(x)=f(t,x,u(x),u_x(x)), u\in \cX^{\alpha}_p.
	\end{equation}
	Then by Lemma \ref{classical-solu-lm1} $F=F(t,u)$ satisfies the hypotheses of Theorem \ref{strong-solu-thm}. Therefore for every $u_{0}\in \cX^{\alpha}_p$ there exists a { unique} strong solution $u(t;t_0,u_0)$ of  { \eqref{new-nonlinear-eq2}}  defined on the  maximal time interval $(t_0, T(t_0, u_0))$ and satisfying the initial condition $u(t_0;t_0,u_0)=u_0$. Let us denote 
	\begin{equation}
	u(t, x; t_0, u_0):=u(t; t_0, u_0)(x), x\in \Gamma. 
	\end{equation}
	Then \eqref{local-1-eq1} follows from \eqref{classical-solu-eq1-1}, \eqref{classical-solu-eq1-3} and \eqref{classical-solu-eq1-4}, while \eqref{local-1-eq2} follows from \eqref{classical-solu-eq2}. 
	
	{ Next, we} prove that $u(t,x;t_0,u_0)$ is a classical solution
	of \eqref{new-nonlinear-eq1}. To that end we note that $2\alpha-p^{-1}>1$ and $2\nu-p^{-1}>0$. Then there is  $\beta\in (0,\nu)$ such that
	$2\beta-p^{-1}>0$, hence, by \eqref{sobolev-embedding-eq5} one has
	\begin{equation*}
	\cX^{\alpha}_p\hookrightarrow \hatt C^{1}(\overline \Gamma),\quad \cX_p^\beta \hookrightarrow \hatt C(\overline \Gamma).
	\end{equation*}
	This together with \eqref{classical-solu-eq1-1} and \eqref{classical-solu-eq1-4} yield
	\begin{equation}
	\label{classical-solu-eq4}
	u(\cdot,\cdot;t_0,u_0)\in \hatt C^{0,1}([t_0,t_0+T_{\max}(t_0,u_0))\times \overline\Gamma)\cap \hatt C^{1,0}((t_0,t_0+T_{\max}(t_0,u_0))\times\overline \Gamma).
	\end{equation}
	Next, defining   
	$$
	\tilde f(t,x)=f(t,x,u(t,x;t_0,u_0),u_x(t,x;t_0,u_0)).
	$$
	and using  \eqref{classical-solu-eq4} one obtains
	\begin{equation}
	\label{classical-solu-eq5}
	\tilde f\in \hatt C^{0,0}((t_0,t_0+T_{\max}(t_0,u_0))\times \overline\Gamma),
	\end{equation}
	Re-writing  { \eqref{new-nonlinear-eq1}}  we obtain
	$$
	-u_{xx}(t,x;t_0,u_0)=-u_t(t,x;t_0,u_0)+\tilde f(t,x).
	$$
	This together with \eqref{classical-solu-eq4} and \eqref{classical-solu-eq5} yield
	$$
	u(\cdot,\cdot;t_0,u_0)\in   \hatt C^{1,2}((t_0,t_0+T_{\max}(t_0,u_0))\times \overline\Gamma).
	$$
	Therefore, $u(t,x;t_0,u_0)$ is a classical solution of \eqref{new-nonlinear-eq1} with initial condition $u(t_0,x;t_0,u_0)=u_0(x)\in \cX_p^\alpha$.

{ Now, we claim that  \eqref{new-nonlinear-eq1} has a unique classical solution  $u(t,x;t_0,u_0)$ satisfying $u(t_0,x;t_0,u_0)=u_0(x)$ and \eqref{local-1-eq1}.  In fact, if $u(t,x;t_0,u_0)$ is a classical solution of   \eqref{new-nonlinear-eq1} satisfying $u(t_0,x;t_0,u_0)=u_0(x)$ and \eqref{local-1-eq1},  then
$u(t,\cdot;t_0,u_0)$ satisfies 
\begin{equation*}
		u(t,\cdot;t_0,u_0)=e^{\Delta( t-t_0)}u_0(\cdot) +\int_{t_0}^te^{\Delta( t-s)} F(s,u(s,\cdot;t_0,u_0))ds,
		\end{equation*}
and hence is a strong solution of \eqref{new-nonlinear-eq2} with initial condition $u(t_0,\cdot;t_0,u_0)=u_0(\cdot)$. The claim then follows from the uniqueness of strong solutions of \eqref{new-nonlinear-eq2}.
}
	
\smallskip

	(2) Follows from  Theorem \ref{strong-solu-thm} (4).
	
\smallskip

	(3) We divide the proof into 3 steps. 
\smallskip

	\noindent {\bf Step 1}.   
	Let us introduce the following notation
	\begin{align}
	&w_0:=u_2-u_1, w(t,x):=u(t,x;t_0,u_2)-u(t,x;t_0,u_1), \\
	&U_k:=u(t,x;t_0,u_k), V_k:=\partial_x U_k,  k=1,2.  
	\end{align}
	Then by assumption $w_0\ge 0$ and one has
	\begin{align} 
	&-w_t+w_{xx}+f(t,x,U_1,V_1) -f(t,x,U_2,V_2)=0, \label{non_linear_equality}\\
	& \sum\limits_{\vartheta \sim e} \partial_{\nu}w_e(\vartheta)=0,\lb{kirw}
	\end{align}
	For given $U_1=u(t,x;t_0,u_1)$ and $U_2=u(t,x;t_0,u_2)$ we will rewrite \eqref{non_linear_equality} as a linear parabolic equation for $w(x,t)$. To that end, let 
	\begin{equation} \notag
	\Phi(x,t,s):=f\left(t,x,sU_1+(1-s)U_2,sV_1+(1-s)V_2\right).
	\end{equation}
	Then one has
	\begin{equation}
	\Phi(x,t,1)-\Phi(x,t,0)=f(t,x,U_1,(U_1)_{x}) -f(t,x,U_2,(U_2)_{x}),	
	\end{equation}
	and
	\begin{align*} 
	\Phi(x,t,1)-\Phi(x,t,0)&=\int_0^1 \frac{\partial \Phi}{\partial s}(t,x,s)ds \\
	&=\int_0^1  \frac{\partial f}{\partial u}(t,x,sU_1+(1-s)U_2,sV_1+(1-s)V_2)(U_1-U_2)ds \\ &\quad + \int_0^1 \frac{\partial f}{\partial v}(t,x,sU_1+(1-s)U_2,sV_1+(1-s)V_2)(V_1-V_2) ds \\ &
	= B(x,t)(U_1-U_2) + C(x,t)(V_1-V_2),
	\end{align*}
	where we denoted
	\begin{align*}
	B(x,t)=\int_0^1  \frac{\partial f}{\partial u}(t,x,sU_1+(1-s)U_2,sV_1+(1-s)V_2)ds \quad \text{ and } \\
	C(x,t)=\int_0^1 \frac{\partial f}{\partial v}(t,x,sU_1+(1-s)U_2,sV_1+(1-s)V_2) ds
	\end{align*}
	
	Therefore \eqref{non_linear_equality} can be rewritten as follows
	\begin{equation} \notag
	L(w):=-w_t+w_{xx}+B(x,t)w_x+C(x,t)w=0.
	\end{equation}
\smallskip

\noindent 	{\bf Step 2}. Let us define $ \hatt{w}(x,t)$via $ w(x,t) = e^{-\lambda t}\hat{w}(x,t) $, where $\lambda = \sup_{x,t}|C(x,t)|$. Then one has
	\begin{align*} 
	L(w) &= -( \hatt{w}_t e^{-\lambda t} - \lambda \hatt{w} e^{-\lambda t}) +\hatt{w}_{xx} e^{-\lambda t} + B(x,t)\hatt{w}_x e^{-\lambda t} + \hatt{w} e^{-\lambda t} \\
	&= e^{-\lambda t} \left(- \hatt{w}_t   +\hatt{w}_{xx}  + B(x,t)\hatt{w}_x  + (C(x,t) + \lambda)\hatt{w} \right)=0.
	\end{align*}
	Thus $\hat{w}$ satisfies $\hat{L}(\hatt{w})=0$ where $\hatt{L}$ is defined as
	\begin{equation} \label{CP-linear-parabolic-eq}
	\hatt{L}(\hatt{w})= - \hatt{w}_t   +\hatt{w}_{xx}  + B(x,t)\hatt{w}_x  + \hatt{C}(x,t) \hatt{w}, \text{ where  } \hatt{C}(x,t)= C(x,t) + \lambda \geq 0
	\end{equation}
	
\smallskip

\noindent 	{\bf Step 3}.  In this step we show that for any $u_1,u_2 \in \cX_p^\alpha$ with $u_2-u_1 \geq 0$ we have $u(t,x;t_0,u_2)-u(t,x;t_0,u_1)\geq 0$, $t\geq t_0$
	
	
	First, assume that $u_1,u_2 \in C(\overline\Gamma)\cap \cX_p^\alpha$ and  $\inf_{x\in\Gamma}(u_2-u_1)(x)>0$. Then by definition $\hatt{w}_0 \in C(\overline\Gamma)\cap \cX_p^\alpha$ and 
	$\inf_{x\in\Gamma}\hatt{w}_0(x)=\inf_{x\in\Gamma}e^{\lambda t}(u_2-u_1)(x)>0$.  
	We claim that $\hatt{w}(t,x;\hat{w}_0)>0$ for all $t>t_0$ and $x\in\Gamma$. Assuming the contrary, there exist  $t_1>t_0$ and $x_1\in\Gamma$ such that
	$\hatt{w}(t_1,x_1;t_0,\hatt{w}_0)=0$ and $\hatt{w}(t,x;t_0,\hatt{w}_0)>0$ for all $t_0<t<t_1$ and $x\in\Gamma$. Let $e_1\in\mathcal{E}$ be such that $x_1\in\bar e_1$.
	Using the maximum principle for parabolic equations (on bounded intervals) we infer that $x_1$ is an end point of edge $e_1$, that is, $x_1\in \partial e_1$. Then Hopf's Lemma gives
	$$
	\partial_{\nu} \hatt{w_e}(t_1,x_1;t_0,\hat{w}_0)< 0,\text{\ for every edge\ }\,\, e \in\mathcal{E}\, {\rm with}\,\, x_1\in\partial e,
	$$ 
	where $\frac{\partial \hatt{w}}{\partial \nu}$, as usual, denotes the derivative of $\hatt{w}$ with respect to the unit outer normal. This implies, in particular, that 
	$$\sum_{e\sim x_1}{\partial_{\nu} \hatt w_e(x_1)} < 0.$$
	This contradicts the vertex condition \eqref{kirw}. Thus  $u(t,x;t_0,u_2)-u(t,x;t_0,u_1)\geq 0$ as asserted. 
	
	Next, for any $u_1,u_2\in \cX_p^\alpha$, due to $2\alpha-p^{-1}>0$ and \eqref{sobolev-embedding-eq5},  there exist $u_{1,n},u_{2, n}\in C(\overline \Gamma)\cap \cX_p^\alpha$ such that
	$$
	\lim\limits_{n\to\infty}(u_{2, n}-u_{1, n})= u_2-u_1\quad {\rm in}\quad L^p(\Gamma).
	$$
	Moreover, since $\inf_{x\in\Gamma}(u_2-u_1)\geq 0$ we may assume that $\inf_{x\in\Gamma}(u_{2,n}-u_{1,n})>0$. Therefore one has
	\begin{equation}
	\hatt{w}_n:=e^{\lambda t}(u_{2, n}-u_{1, n})\in C(\overline \Gamma)\cap \cX_p^\alpha,\  \inf_{x\in\Gamma} \hatt{w}_n(x)>0, 	\hatt{w}_n\to \hatt{w}\text{ in } L^p(\Gamma),
	\end{equation}
	and, by the previous argument,
	\begin{equation}
	\hatt{w}(t,x; t_0, \hat{w}_n)>0, t>t_0, x\in\Gamma, n\geq 1.
	\end{equation}
	Finally, using 
	\begin{equation}
	\hatt{w}(t,\cdot;t_0, \hat{w}_n)\to \hatt{w}(t,\cdot;t_0,\hat{w}_n)\quad {\rm as}\quad n\to\infty\quad {\rm in}\quad \elpee,
	\end{equation}
	we obtain  $u(t,x;t_0,u_2)-u(t,x;t_0,u_1)\geq 0$, $t\geq t_0$ as required.
	
\end{proof}

\begin{proposition}{\lb{propHenry}}{\cite[p. 190]{MR610244}}. Let $\alpha\in[0,1)$ $a_1, a_2>0$, $T>0$. Assume that an integrable function $u:[0,T]\rightarrow [0,\infty)$ satisfies 
	\begin{equation}\label{gronwell1}
		0\leq u(t)\leq a_1 t^{-\alpha}+a_2 \int_0^t(t-s)^{-\alpha}u(s)ds,
	\end{equation}	
	for $t\in [0,T]$. Then there exists $C=C(\alpha, a_2, T)>0$ such that
	\begin{equation}
		0\leq u(t)\leq \frac{a_1C}{1-\alpha} t^{-\alpha}, \ t\in(0, T). \label{gronwell2}
	\end{equation}
\end{proposition}


\end{document}